\documentclass[11pt,a4paper]{amsart}
\usepackage{amssymb,amsmath}
\usepackage{mathrsfs}
\usepackage[alphabetic]{amsrefs}
\usepackage{graphicx,color}
\usepackage[utf8]{inputenc}
\usepackage{amsthm}
\usepackage{latexsym}
\usepackage{slashed}
\usepackage[all]{xy}
\usepackage{hyperref}
\usepackage[mathscr]{eucal}
\usepackage{tikz}
\usepackage{mathtools}
\usepackage{accents}

\def\theequation{\thesection.\arabic{equation}}
\makeatletter
\@addtoreset{equation}{section}
\makeatother

\makeatletter
\newcommand{\eqnum}{\refstepcounter{equation}\textup{\tagform@{\theequation}}}
\makeatother

\newcounter{copy}
\makeatletter
\renewcommand{\thecopy}{\ifnum0=\c@section\arabic{copy}\else\thesection.\arabic{copy}'\fi}
\makeatother
\newcommand{\copynum}[1][]{\refstepcounter{copy}#1{\thecopy}} 

\BibSpec{book}{%
    +{}  {\PrintPrimary}                {transition}
    +{.} { \textit}                     {title}
    +{.} { }                            {part}
    +{:} { \textit}                     {subtitle}
    +{,} { \PrintEdition}               {edition}
    +{}  { \PrintEditorsB}              {editor}
    +{,} { \PrintTranslatorsC}          {translator}
    +{,} { \PrintContributions}         {contribution}
    +{,} { }                            {series}
    +{,} { \voltext}                    {volume}
    +{,} { }                            {publisher}
    +{,} { }                            {organization}
    +{,} { }                            {address}
    +{,} { }                            {status}
    +{,} { \PrintDOI}                   {doi}
    +{,} { \PrintISBNs}                 {isbn}
    +{}  { \parenthesize}               {language}
    +{}  { \PrintTranslation}           {translation}
    +{;} { \PrintReprint}               {reprint}
    +{,} { \PrintDate}                  {date}
    +{.} { }                            {note}
    +{.} {}                             {transition}
}
\BibSpec{article}{%
    +{}  {\PrintAuthors}                {author}
    +{,} { \textit}                     {title}
    +{.} { }                            {part}
    +{:} { \textit}                     {subtitle}
    +{,} { \PrintContributions}         {contribution}
    +{.} { \PrintPartials}              {partial}
    +{,} { }                            {journal}
    +{}  { \textbf}                     {volume}
    +{}  { \PrintDatePV}                {date}
    +{,} { \issuetext}                  {number}
    +{,} { \eprintpages}                {pages}
    +{,} { }                            {status}
    +{,} { \PrintDOI}                   {doi}
    +{}  { \parenthesize}               {language}
    +{}  { \PrintTranslation}           {translation}
    +{;} { \PrintReprint}               {reprint}
    +{.} { }                            {note}
    +{,} { \eprint}                     {eprint}
    +{.} {}                             {transition}
}
\BibSpec{collection.article}{%
    +{}  {\PrintAuthors}                {author}
    +{,} { \textit}                     {title}
    +{.} { }                            {part}
    +{:} { \textit}                     {subtitle}
    +{,} { \PrintContributions}         {contribution}
    +{,} { \PrintConference}            {conference}
    +{}  {\PrintBook}                   {book}
    +{,} { }                            {booktitle}
    +{,} { \PrintDateB}                 {date}
    +{,} { pp.~}                        {pages}
    +{,} { }                            {publisher}
    +{,} { }                            {organization}
    +{,} { }                            {address}
    +{,} { }                            {status}
    +{,} { \PrintDOI}                   {doi}
    +{,} { \eprint}        {eprint}
    +{}  { \parenthesize}               {language}
    +{}  { \PrintTranslation}           {translation}
    +{;} { \PrintReprint}               {reprint}
    +{.} { }                            {note}
    +{.} {}                             {transition}
}

\BibSpec{misc}{
  +{}{\PrintAuthors}  {author}
  +{,}{ \textit}     {title}
  +{}{ (}             {date}
  +{),}{ }             {note}
  +{.}{}              {transition}
}

\theoremstyle{definition}
\newtheorem{defn}[equation]{Definition}
\newtheorem{notn}[equation]{Notation}
\theoremstyle{plain}
\newtheorem{thm}[equation]{Theorem}
\newtheorem{prp}[equation]{Proposition}
\newtheorem{lem}[equation]{Lemma}
\newtheorem{cor}[equation]{Corollary}

\theoremstyle{remark}
\newtheorem{rmk}[equation]{Remark}

\newcommand{\bB}{\mathbb{B}}
\newcommand{\bC}{\mathbb{C}}
\newcommand{\bD}{\mathbb{D}}

\newcommand{\bK}{\mathbb{K}}

\newcommand{\bM}{\mathbb{M}}
\newcommand{\bN}{\mathbb{N}}

\newcommand{\bQ}{\mathbb{Q}}
\newcommand{\bR}{\mathbb{R}}

\newcommand{\bZ}{\mathbb{Z}}

\newcommand{\cB}{\mathcal{B}}
\newcommand{\cC}{\mathcal{C}}

\newcommand{\cE}{\mathcal{E}}

\newcommand{\cG}{\mathcal{G}}

\newcommand{\cI}{\mathcal{I}}

\newcommand{\cM}{\mathcal{M}}

\newcommand{\cP}{\mathcal{P}}
\newcommand{\cQ}{\mathcal{Q}}

\newcommand{\cS}{\mathcal{S}}

\newcommand{\cU}{\mathcal{U}}
\newcommand{\cV}{\mathcal{V}}
\newcommand{\cW}{\mathcal{W}}
\newcommand{\cX}{\mathcal{X}}

\newcommand{\cZ}{\mathcal{Z}}

\newcommand{\fB}{\mathfrak{B}}
\newcommand{\fC}{\mathfrak{C}}
\newcommand{\fD}{\mathfrak{D}}

\newcommand{\fT}{\mathfrak{T}}

\newcommand{\fv}{\mathfrak{v}}

\newcommand{\bu}{\mathbf{u}}
\newcommand{\bv}{\mathbf{v}}

\newcommand{\sA}{\mathscr{A}}

\newcommand{\sE}{\mathscr{E}}

\newcommand{\sH}{\mathscr{H}}

\newcommand{\sK}{\mathscr{K}}

\newcommand{\sP}{\mathscr{P}}

\newcommand{\op}{\mathrm{op}}

\renewcommand{\Im}{\mathrm{Im} \hspace{0.1em}}
\newcommand{\pt}{\mathrm{pt}}

\newcommand{\id}{\mathrm{id}}
\newcommand{\ev}{\mathrm{ev}}

\DeclareMathOperator{\hotimes}{\hat{\otimes }}
\DeclareMathOperator{\ch}{\mathrm{ch}} 
\DeclareMathOperator{\K}{\mathrm{K}}
\DeclareMathOperator{\KR}{\mathrm{KR}}

\DeclareMathOperator{\KK}{\mathrm{KK}}

\DeclareMathOperator{\Hom}{\mathrm{Hom}}
\DeclareMathOperator{\ind}{\mathrm{ind}}

\DeclareMathOperator{\Ad}{\mathrm{Ad}}
\DeclareMathOperator{\diag}{\mathrm{diag}}

\newcommand{\Bdl}{\mathrm{Bdl}}
\newcommand{\qRep}{\mathrm{qRep}}
\newcommand{\pr}{\mathrm{pr}}

\newcommand{\Td}{\mathrm{Td}}

\title[The relative higher index and almost flat bundles II]{The relative Mishchenko--Fomenko higher index and almost flat bundles II: \\ Almost flat index pairing}
\author{Yosuke KUBOTA}
\address{iTHEMS Research Program, RIKEN, 2-1 Hirosawa, Wako, Saitama 351-0198, Japan}
\email{yosuke.kubota@riken.jp}

\date{}

\begin{document}
\begin{abstract}
This is the second part of a series of papers which bridges the Chang--Weinberger--Yu relative higher index and geometry of almost flat hermitian vector bundles on manifolds with boundary. In this paper we apply the description of the relative higher index given in Part I to provide the relative version of the Hanke--Schick theorem, which relates the relative higher index with index pairing of a K-homology cycle with almost flat relative vector bundles. We also deal with the quantitative version and the dual problem of this theorem.
\end{abstract}

\maketitle
\tableofcontents

\section{Introduction}
This paper is a sequel of \cite{Kubota1}. In this part II we apply the Mishchenko--Fomenko description of the Chang--Weinberger--Yu relative higher index developed in the part I to the index pairing with almost flat bundles on manifolds with boundary. Here we also make use of the foundations of almost flat (stably) relative bundles prepared in \cite{Kubota3}.

The notion of almost flat bundle is introduced as a geometric counterpart of the higher index theory by Connes--Gromov--Moscovici~\cite{MR1042862} for the purpose of proving the Novikov conjecture for a large class of groups. 
It also plays a fundamental role in the study of positive scalar curvature metrics in \cite{MR720933,MR1389019}. 
Its central concept is the almost monodromy correspondence, that is, the rough one-to-one correspondence between almost flat bundles and quasi-representations of the fundamental group. 
In \cite{Kubota3}, the author consider the relative and stably relative vector bundles on a pair of topological spaces as a representative of the relative $\K^0$-group and introduce the notion of almost flatness for them. Moreover, the almost monodromy correspondence is generalized to this relative setting.

The relation between the role of almost flat index pairing in geometry and the C*-algebraic higher index theory is clearly understood in the work of Hanke and Schick. 
In \cite{MR2259056,MR2353861}, they proved that the higher index $\alpha_\Gamma ([M])$ of the K-homology fundamental class $[M] \in \K_*(M)$ of an enlargeable closed spin manifold $M$ with $\pi_1(M)=\Gamma$ does not vanish without any assumption on the fundamental group concerned with the Baum--Connes conjecture. 
As is reorganized in \cite{MR3289846}, this is essentially a consequence of the fact that $\alpha_\Gamma ([M]) \neq 0$ if $M$ admits an almost flat bundle with non-trivial index pairing. 
The idea of Hanke and Schick relies on the fact that the dual higher index is related to the monodromy correspondence of flat bundles of Hilbert C*-modules. 

The Mishchenko--Fomenko higher index map $\alpha_\Gamma$ is given by the Kasparov product with the KK-class $\ell_\Gamma \in \KK(\bC, C(M) \otimes C^*\Gamma)$ represented by the Mishchenko line bundle $\tilde{M} \times_\Gamma C^*\Gamma$. For a unitary representation $\pi \colon \Gamma \to \bB(P)$ to a finitely generated projective Hilbert $A$-module, the Kasparov product with $\ell_\Gamma$ over $C^*\Gamma$ (which is called the dual higher index map in this paper) maps the element $[\pi] \in \KK(C^*\Gamma, A)$ to the associated bundle $\cP:=\tilde{M} \times_\Gamma P$. Hence the associativity of the Kasparov product relates the index pairing $[\cP] \otimes _{C(M)}[M]$ with the higher index as
\begin{align}\alpha_\Gamma ([M]) \otimes_{C^*\Gamma} [\pi] = \ell_\Gamma \otimes _{C(M)}[M] \otimes_{C^*\Gamma}[\pi] = [\cP] \otimes _{C(M)}[M]. \label{form:indexpair}\end{align} 
The idea of Hanke--Schick is to construct a nice flat Hilbert C*-module bundle from a family of almost flat bundles.

The first purpose of this paper, studied in Section \ref{section:3}, is to establish a relative version of the result of Hanke--Schick. There are two ingredients of it. One is the basic theory of almost flat bundles on manifolds with boundary (particularly the almost monodromy correspondence) developed in \cite{Kubota3}, and the other is the relative version of index pairing (\ref{form:indexpair}). 
Here the higher index is replaced with the Chang--Weinberger--Yu relative higher index map \cite{mathKT150603859}, which is a homomorphism
\[ \alpha_{\Gamma, \Lambda} \colon \K_*(X,Y) \to \K_*(C^*(\Gamma, \Lambda)), \]
defined for a pair of CW-complexes $(X,Y)$ with $\pi_1(X)=\Gamma$ and $\pi_1(Y)=\Lambda$ (for more details, see Subsection \ref{section:2.1}). 
It is proved in \cite{Kubota1} that this map is given by the Kasparov product with an element $\ell_{\Gamma , \Lambda} \in \KK (\bC , C_0(X^\circ) \otimes C^*(\Gamma, \Lambda))$.
The key observation is Theorem \ref{prp:bundle}, corresponding to the fact $\ell _\Gamma \otimes _{C^*\Gamma}[\pi]=[\cP]$ in the above paragraph. 
Roughly speaking, here we show that the Kasparov product with $\ell_{\Gamma, \Lambda} $ maps a relative representation of $(\Gamma, \Lambda)$, i.e.\ a pair of representations of $\Gamma $ which is identified on $\Lambda$, to the associated relative bundle.
To realize the concept in full generality, we employ the equivalence relation generated by unitary equivalence, stabilization and homotopy as the `identification' of a pair of representation. 
Then we get the result corresponding to (\ref{form:indexpair}) by the same argument using the associativity of the Kasparov product. 
The relative Hanke--Schick theorem (Theorem \ref{thm:Karea}) is now obtained in the same way as \cite{MR2259056,MR2353861} with the help of the relative almost monodromy correspondence. 

In addition, there is another application of Theorem \ref{prp:bundle} to the index theoretic refinement of the Hanke--Pape--Schick codimension $2$ index obstruction for the existence of positive scalar curvature metric \cite{MR3449594}, which is discussed in Subsection \ref{section:3.2}. Here, the higher index of a codimension 2 submanifold $W$ of $M$ (with a condition on homotopy groups) is related to the relative higher index of the manifold $M \setminus U$, where $U$ is a tubular neighborhood of $W$.

In the rest part of the paper we discuss in-depth problems related to relative index theory of almost flat bundles. 
In Section \ref{section:5}, we study the quantitative version of Theorem \ref{thm:Karea}. A key idea of \cite{MR2259056} is to treat an infinite family of almost flat bundles simultaneously and relate the asymptotics of the index pairings with the higher index. On the other hand, the $\ell^1(\Gamma)$-valued higher index, instead of the usual $C^*(\Gamma)$-valued one, is mapped to a projection up to a small correction by a single quasi-representation. This map is studied in \cite{MR1042862} and compared with the index pairing with the associated almost flat bundle.
 In \cite{MR2982445}, Dadarlat gives an alternative approach using Lafforgue's Banach KK-theory.  Here, we reformulate the result of \cite{MR2982445} in terms of the quantitative K-theory introduced in Oyono-Oyono--Yu~\cite{MR3449163} instead of Banach KK-theory. After that, we generalize the result of Connes--Gromov--Moscovici to the relative setting.

In Section \ref{section:8}, we study the dual problem of Theorem \ref{thm:Karea}, in other words, the relative version of the problem proposed by Gromov in \cite[Section $4\frac{2}{3}$]{MR1389019}. It is a consequence of the almost monodromy correspondence that any almost flat bundle is obtained by pull-back from the classifying space $B\Gamma$. Then it is a natural question whether any element of $\K^0(B\Gamma)$ (or $\K_0(B\Gamma) \otimes \bQ$) is represented by an almost flat bundle.
This question is first considered in \cite[Section $8\frac{14}{15}$]{MR1389019} geometrically for the fundamental group of a Riemannian manifold with non-positive sectional curvature. After that, Dadarlat gives a KK-theoretic approach to this problem in \cite{MR3275029}. 
Here we follow this idea to study the almost flat K-theory class of the pair $(B\Gamma , B\Lambda)$. 
The celebrated Tikuisis--White--Winter theorem \cite{MR3583354} in the theory of C*-algebras enables us to include a large class of residually amenable groups to the range of our discussion. 
We show that any element of the range of the dual higher index map 
\[ \beta_{\Gamma, \Lambda} \colon \K^0(C^*(\Gamma, \Lambda)) \to \K^0(X,Y),\]
i.e.\ the Kasparov product with $\ell_{\Gamma, \Lambda} $ over $C^*(\Gamma, \Lambda)$, is represented by an almost flat stably relative vector bundle. Moreover, we also show that such elements are represented by an almost flat relative vector bundle if $\phi \colon \Lambda \to \Gamma$ is injective.

\begin{notn}\label{notation}
Throughout this paper we use the following notations.
\begin{itemize}
\item For a C*-algebra $A$, let $A^+$ denote its unitization $A + \bC \cdot 1$. 
\item For a C*-algebra $A$, let $\cM(A)$ denote its multiplier C*-algebra and $\cQ(A):=\cM(A)/A$.
\item For a C*-algebra $A$ and $a<b \in \bR \cup \{ \pm \infty \}$,let $A(a,b):=A \otimes C_0(a,b)$. Similarly we define $A[a,b)$ and $A[a,b]$. For a Hilbert $A$-module $E$, let $E(a,b)$ denote the Hilbert $A(a,b)$-module $E \otimes C_0(a,b)$.
\item For a $\ast$-homomorphism $\phi \colon A \to B$, let $C\phi$ denote the mapping cone C*-algebra defined as
\[ C\phi = \{(a, b_s) \in A \oplus B[0,1) \mid \phi(a)=b_0 \}.\] 
\item For a Hilbert $A$-module $E$, let $\bB (E)$ and $\bK(E)$ denote the C*-algebra of bounded adjointable and compact operators on $E$ respectively. Let $\mathrm{U} (E)$ denote the unitary group of $\bB(E)$.
\item For a compact space $X$ and a Hilbert $A$-module $P$, let $\underline{P}_X$ denote the trivial bundle $X \times P$ on $X$.
\end{itemize}
\end{notn}

\subsection*{Acknowledgment}
The author would like to Martin Nitsche and Thomas Schick for their careful reading and helpful comments on a previous version of this paper. This work was supported by RIKEN iTHEMS Program.

\section{Prelimilaries}\label{section:2}
In this section we summarize the results of \cite{Kubota1} and \cite{Kubota3} which will be used in this paper. Throughout this paper we focus on the complex coefficient $\K$-theory, C*-algebra, vector bundle and so on.

\subsection{Relative Mishchenko--Fomenko higher index}\label{section:2.1}
Let $(\Gamma , \Lambda)$ be a pair of discrete groups with a homomorphism $\phi \colon \Lambda \to \Gamma$. Note that $\phi$ induces $B\phi \colon B\Lambda \to B\Gamma$ (we map assume that $B\phi$ is an inclusion by replacing $B\Gamma$ with the mapping cylinder $B\Gamma \sqcup_{B\phi} B\Lambda \times [0,1]$).
Let $(X,Y)$ be a pair of finite CW-complexes with a reference map $f \colon (X,Y) \to (B\Gamma , B\Lambda)$, which associates a $\Gamma$-covering $\tilde{X} \to X$ and a $\Lambda$-covering $\tilde{Y} \to Y$. 
\begin{notn}\label{not:Xr}
We write as
\begin{align*}
Y_r&:= \left\{ \begin{array}{ll} Y \times [0,r] & \text{for $r \in [0,\infty) $}, \\ Y \times [0,\infty) & \text{for $r = \infty$},  \end{array} \right. \\
X_r &:= X \sqcup _Y Y_r.
\end{align*}
For $r \in [1,\infty)$, let $Y_r':= Y \times [1, r] \subset X_\infty$. We write $X_r^\circ$, $Y_r^\circ$ and $(Y_r')^\circ$ for the interiors of $X_r$, $Y_r$ and $Y_r'$ as subspaces of $X_\infty$.
\end{notn}

The \emph{Chang--Weinberger--Yu relative higher index} is a group homomorphism
\[ \mu ^{\Gamma , \Lambda} _* \colon \K_*(X,Y) \to \K_*(C^*(\Gamma , \Lambda)), \]
where $C^*(\Gamma, \Lambda) $ is the relative (maximal) group C*-algebra defined as
\[ C^*(\Gamma , \Lambda):=S C(\phi \colon C^*\Lambda \to C^*\Gamma). \]

In \cite[Section 3]{Kubota1}, the author gives a definition of $\mu^{\Gamma, \Lambda}_*$ inspired from the Mishchenko--Fomenko index pairing. Let us write the Mishchenko line bundles on $X$ and $Y$ as $\cV  := \tilde{X}\times _\Gamma C^*\Gamma$ and $\cW :=\tilde{Y} \times _\Lambda C^*\Lambda$ respectively, and let $\cX := \tilde{Y} \times _\Lambda C\phi$.  We define
\begin{align*}
\begin{split}
\sE_2 &:= SC(X, \cV) \oplus _{C(Y, \cX)}C_0(Y \times [0,2) , \cX ) \\
&=\{ (\xi ,\eta ) \in C(X, S\cV) \oplus C_0(Y \times [0,2) , \cX ) \mid \psi_Y(\xi |_{Y}) = \eta|_{Y \times \{ 0 \} }  \},
\end{split}
\end{align*}
where $\psi _Y \colon S\cV|_Y \to \cX$ is the bundle map induced from $\psi \colon SC^*\Gamma \to C\phi$, and 
\begin{align*} 
\rho (r,s) = \rho_s(r) := \min \{ 1 , 2s+2r-3 \} \in C([1,2] \times [0,1]). \label{form:rho}\end{align*}
We regard this $\rho$ as a continuous function on $X_2 \times [0,1]$ by $\rho(x,s):=2s-1$ for $x \in X_1$, $\rho(y,r,s):=\rho(r,s)$ for $(r,y) \in Y_2' $, which acts on $\sE_2$ by multiplication. 
\begin{figure}[t]
\centering 
\begin{tikzpicture}[scale=0.8]
\shade [top color = black!90!white, bottom color = black!90!white, middle color = white, shading angle = 135] (-2,-2) rectangle (4,4);
\shade [top color = black!42!white, bottom color = black!42!white, middle color = white] (-2,0) rectangle (0,4);
\fill [color = black!42!white] (4,0) -- (4,4) -- (0,4) -- (4,0);
\fill [color = white] (-2.1,-2.1) rectangle (4.1,0);
\draw [->] (-3,-0.3) -- (-3,4.5);
\draw [->] (-1,-1) -- (4.5,-1);
\draw [thick,dashed] (-2,4) -- (4,4);
\draw [thick, dashed] (4,4) -- (4,0);
\draw [thick] (4,0) -- (0,0);
\draw (0,0) -- (0,4);
\draw [thick,dashed] (0,0) -- (-2,0);
\node at (4.3,-0.7) {$r$};
\node at (-2.2,4.3) {$s$};
\node at (-1,4.5) {$X_1$};
\node at (2,4.5) {$Y_2'$};
\fill (4,-1) coordinate (O) circle[radius=1pt] node[below=2pt]  {$2$};
\fill (0,-1) coordinate (O) circle[radius=1pt] node[below=2pt]  {$1$};
\fill (-3,4) coordinate (O) circle[radius=1pt] node[left=2pt]  {$1$};
\fill (-3,0) coordinate (O) circle[radius=1pt] node[left=2pt]  {$0$};
\end{tikzpicture}
\caption{The shading shows the value of $|\rho (r,s)|$.}
\end{figure}
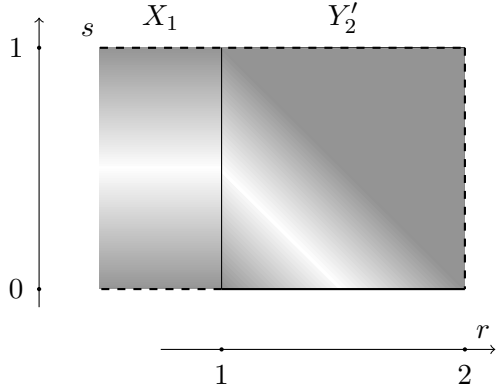
The \emph{relative Mishchenko line bundle} $\ell_{\Gamma, \Lambda}$ is defined as 
\[ \ell_{\Gamma, \Lambda}=\ell_{\Gamma, \Lambda} :=[ \sE_2 , 1, \rho ] \in \KK_{-1}(\bC , C_0(X_2 ) \otimes C\phi).\]

\begin{defn}[{\cite[Definition 3.3]{Kubota1}}]
The \emph{relative Mishchenko--Fomenko higher index} $\mu^{\Gamma, \Lambda}_*$ is defined by the Kasparov product
\[ \ell_{\Gamma, \Lambda } \hotimes_{C_0(X^\circ)} { \cdot }\  \colon \KK_*(C_0(X^\circ ), \bC ) \to \K_*(C^*(\Gamma, \Lambda )). \] 
\end{defn}
We also use the symbol $\alpha_{\Gamma, \Lambda}$ for this homomorphism.

\begin{prp}[{\cite[Proposition 3.6]{Kubota1}}]\label{prp:dual}
The dual relative higher index map $\beta_{\Gamma, \Lambda} \colon \KK(C^*(\Gamma, \Lambda) , \bC ) \to \K^*(X,Y)$ defined as the Kasparov product $\ell_{\Gamma, \Lambda} \hotimes _{C^*(\Gamma, \Lambda)} {\cdot }$ satisfies
\[  \langle \alpha_{\Gamma, \Lambda} (x) , \xi \rangle = \langle x, \beta_{\Gamma, \Lambda} (\xi) \rangle \in \KK(\bC, \bC) \cong \bZ  , \]
where the bracket $\langle {\cdot } , {\cdot } \rangle$ means the index pairing, i.e., the Kasparov product of the K-homology and K-cohomology groups of a C*-algebra. 
\end{prp}

Here we give a remark on a realization the relative Mishchenko line bundle $\ell_{\Gamma, \Lambda}$, which is an element of the $\K_1$-group $\K_1(C_0(X_2^\circ) \otimes C\phi)$, as a unitary of $C_0(X_2^\circ) \otimes C\phi$.
Let $\cU:= \{ U_\mu\}_{\mu \in I}$ be a finite open cover of $X$ such that the restriction of $\tilde{X}$ to each $U_\mu$ is a trivial bundle. We choose a local trivialization $\theta_\mu \colon \tilde{X}|_{U_\mu} \cong U_\mu \times \Gamma $ and let $\gamma _{\mu \nu}$ denote the transformation function $\theta_\nu(x) \theta_\mu^*(x)$ (which is independent of $x \in U_\mu \cap U_\nu$). 

Let $\{ \eta_\mu \}_{\mu \in I}$ be a family of continuous functions such that $\mathrm{supp} (\eta_\mu) \subset U_\mu$, $0 \leq \eta_\mu(x) \leq 1$ and $\sum \eta_\mu ^2 =1$. We write $\bM_I$ for the matrix algebra on $\bC^I$ and let $\{ e_{\mu \nu}\}_{\mu , \nu \in I}$ denote the matrix unit. Then, 
\begin{align}
P_\cV := \sum_{\mu , \nu \in I} \eta_\mu \eta_\nu \otimes u_{\gamma_{\mu \nu}} \otimes e_{\mu \nu} \in C(X) \otimes C^*(\Gamma) \otimes \bM_I \label{form:projMF}
\end{align}
is a projection whose support is isomorphic to $\cV$ as Hilbert $C^*\Gamma$-module bundles on $X$. This means that $\ell_\Gamma = [P_\cV]$.

\begin{lem}\label{lem:uniMF}
The element $ \ell_{\Gamma, \Lambda} \in \KR_{-1}(C_0(X_1^\circ) \otimes C\phi )$ is represented by the transpose-invariant unitary $(U_\cW,V_{\cV,s}) \in (C_0(X_1^\circ) \otimes C\phi )^+$, where
\begin{align*}
\begin{split}
U_\cW &:= -e^{-\pi i \rho_0}P_\cW + 1-P_\cW \in (C_0(Y_1^\circ) \otimes C^*\Lambda \otimes \bM_I)^{+}\\
V_{\cV,s} &:= -e^{-\pi i \rho_s} P_\cV + 1-P_\cV \in (C_0(X^\circ_1) \otimes C^*\Gamma \otimes \bM_I)^{+}. 
\end{split}
\end{align*}
\end{lem}
\begin{proof}
Let $\cC$ denote the kernel of the $\ast$-homomorphism $C_0(X_2^\circ) \otimes C\phi \to C(X) \otimes C^*\Lambda$ and let $\iota \colon \cC \to C_0(X_2^\circ) \otimes C\phi$ denote the inclusion. Then $P_\cE:=(P_\cW, P_\cV) $ determines an element of the multiplier C*-algebra $\cM(\bM_I(\cC))$ such that the exterior tensor product $(P_\cE \cC^{ \oplus |I|}) \otimes _\iota (C_0(X) \otimes C\phi)$ is isomorphic to $\sE$ as a Hilbert $C(X) \otimes C\phi$-module. Hence the composition
\[ \K_0(\cQ(\cC )) \to \KK_1(\bC, \cC ) \xrightarrow{\iota} \KK_1(\bC, C_0(X) \otimes C\phi ) \] 
maps the projection $P_\cE \cdot (\frac{\rho +1}{2}) + (1-P_\cE) $ to $[\sE,1, \rho]=\ell_{\Gamma, \Lambda}$. This element is mapped by the boundary map $\partial \colon \K_0(\cQ(\cC)) \to \K_1(\cC)$ to the element represented by the unitary 
\[ e^{-2 \pi i (P_\cE \cdot \frac{\rho +1}{2} + (1-P_\cE))} = -e^{-\pi i \rho}P_\cE  + 1-P_\cE = (U_\cV , V_{\cV, s}) . \] 
This finishes the proof.
\end{proof}

\subsection{Rational surjectivity of the dual relative higher index map}
The rational injectivity of the relative higher index map and the rational surjectivity of its dual are studied in \cite[Section 6]{Kubota1}. In Section \ref{section:8} we apply the latter for the existence of almost flat relative vector bundle representing an arbitrary element of relative $K^0$-group of the pair $(B\Gamma, B\Lambda)$. 

We consider the following assumptions for $(\Gamma, \Lambda)$.
\begin{itemize}
\item[\eqnum \label{cond:BC1}] The group $\Gamma$ has the $\gamma$-element $\gamma _\Gamma $. \setcounter{copy}{\value{equation}} 
\item[\eqnum \label{cond:BC2}] For any finite subgroup $K \subset \Gamma$, the subgroup $\phi^{-1}(K) \leq \Lambda$ satisfies $\gamma=1$. 
\item[\eqnum \label{cond:BC3}] The subgroup $\ker \phi$ is torsion-free. 
\end{itemize}
For example, the condition (\ref{cond:BC1}) is satisfied if $\Gamma$ is coarsely embeddable into a separable Hilbert space \cite{MR1905840} and the condition (\ref{cond:BC2}) is satisfied if $\ker \phi$ has the Haagerup property \cite{MR1821144}. 

We also consider a stronger alternative of (\ref{cond:BC2}).   
\begin{itemize}
\item[{(\copynum[\label{cond:BC2'}])}] The subgroup $\ker \phi$ of $\Lambda$ is amenable.
\end{itemize}
If (\ref{cond:BC2'}) is satisfied, the reduced relative group C*-algebra 
\begin{align} C^*_r(\Gamma, \Lambda) := SC(\phi_r \colon C^*_r \Lambda \to C^*_r\Gamma ) \label{form:reduced} \end{align}
makes sense. We write $\epsilon _{\Gamma, \Lambda } \colon C^*(\Gamma, \Lambda ) \to C^*_r(\Gamma, \Lambda )$ for the quotient.

We write $j_\phi$ for the functor from the category of $\Gamma$-C*-algebras to the category of C*-algebras mapping $A$ to the relative crossed product defined as $A \rtimes (\Gamma, \Lambda):= SC( \id_A \rtimes \phi \colon A \rtimes \Lambda \to A \rtimes \Gamma)$. By the universality of the Kasparov category, it induces the functor $j_\phi \colon \mathfrak{KK}^\Gamma \to \mathfrak{KK}$, which maps the $\gamma$-element of $\Gamma $ to $j_\phi (\gamma_\Gamma ) \in \KK(C^*(\Gamma, \Lambda) , C^*(\Gamma , \Lambda))$.

\begin{thm}[{\cite[Theorem 6.6, Proposition 6.10]{Kubota1}}]\label{thm:BC}
Let $\phi \colon \Lambda \to \Gamma $ be a homomorphism of groups. 
\begin{enumerate}
\item If (\ref{cond:BC1}), (\ref{cond:BC2}) and (\ref{cond:BC3}) are satisfied, then the composition $\beta_{\Gamma, \Lambda} \circ j_\phi (\gamma _\Gamma ) \colon \K^*(C^*(\Gamma, \Lambda)) \to \K^*(B\Gamma, B\Lambda)$ is rationally surjective.
\item If (\ref{cond:BC2'}) is satisfied, then $\Im (\epsilon_{\Gamma, \Lambda}^*) \subset \K^*(C^*(\Gamma, \Lambda))$ includes $\Im j_\phi (\gamma_\Gamma )$. 
\end{enumerate}
\end{thm}
Therefore, if (\ref{cond:BC1}), (\ref{cond:BC2'}) and (\ref{cond:BC3}) are satisfied, then $\beta _{\Gamma, \Lambda} \circ \epsilon _{\Gamma, \Lambda}^*$ is rationally surjective.

\begin{rmk}\label{rmk:nonrel}
Theorem \ref{thm:BC} (1) is a relative analogue of the following statement: Let $\beta_\Gamma \colon \K^*(C^*\Gamma ) \to \K^*(B\Gamma)$ denote the dual higher index map, i.e., the Kasparov product $\ell_\Gamma \otimes_{C^*\Gamma}{\cdot}$. If $\Gamma$ has the $\gamma$-element, then $\beta_\Gamma \circ j_\Gamma (\gamma_\Gamma) \colon \K^*(C^*(\Gamma)) \to \K^*(B\Gamma )$ is rationally surjective. This is proved in the same way as \cite[Theorem 6.6]{Kubota1} by using the Dirac-dual Dirac method and the rational injectivity of the higher index map $\alpha _\Gamma \colon \K_*(B\Gamma) \to \K_*(C^*\Gamma)$ shown in \cite[Section 15]{MR928402}. Also, it is shown in the same way as \cite[Proposition 6.10]{Kubota3} that the image $\Im j_\Gamma (\gamma _\Gamma)$ is included to $\Im \epsilon_\Gamma^*$, where $\epsilon _\Gamma \colon C^*(\Gamma) \to C^*_r(\Gamma)$ is the quotient. 
\end{rmk}
\begin{rmk}
As is pointed out in \cite[Remark 6.8]{Kubota1}, we do not need to restrict the situation to the case that $B\Gamma$ and $B\Lambda$ have the homotopy type of finite CW-complexes.
\end{rmk}

\subsection{Almost flat relative bundles}
Here we briefly review the notion of almost flat (stably) relative vector bundles and its almost monodromy correspondence. 
Let $X$ be a finite CW-complex with a good open cover $\cU:= _{\mu \in I}$. Then the fundamental group $\Gamma:=\pi_1(X)$ is generated by $\cG:=\{ \gamma_{\mu \nu} \}_{\mu,\nu \in I} $ once we fix a translation function $\{ \gamma _{\mu \nu}\}_{\mu , \nu \in I}$ of the $\Gamma$-Galois covering $\tilde{X}$.

\begin{defn}
Let $X$, $\cU$, $\Gamma$, $\cG$ be as above. Let $A$ be a unital C*-algebra, let $P$ be a finitely generated projective Hilbert $A$-module and let $T$ be a maximal subtree of the $1$-skeleton $N_\cU^{(1)}$ of the nerve of $\cU$. 
\begin{itemize}
\item  A $\mathrm{U}(P)$-valued \v{C}ech $1$-cocycle $\bv=\{ v_{\mu \nu}\}_{\mu,\nu \in I}$ on $\cU$ is an \emph{$(\varepsilon , \cU)$-flat bundle} on $X$ with the typical fiber $P$ if $\|  v_{\mu\nu}(x)-v_{\mu\nu}(y)\| <\varepsilon$ for any $x,y \in U_{\mu\nu}:=U_\mu \cap U_\nu$. It is said to be normalized on $T$ if $\| v_{\mu \nu} -1 \| <\varepsilon $ for any $\langle \mu,\nu \rangle  \in T$.
\item A map $\pi \colon \Gamma \to \mathrm{U}(P)$ is a \emph{$(\varepsilon , \cG)$-representation} of $\Gamma$ on $P$ if $\pi(e)=1$ and
\[ \|\pi(g)\pi(h)-\pi(gh)\|<\varepsilon   \]
for any $g,h \in \cG$. 
\end{itemize}
We write $\Bdl_P^{\varepsilon , \cU}(X)$ for the set of $(\varepsilon , \cU)$-flat bundles with the typical fiber $P$ and $\qRep_{P}^{\varepsilon , \cG}(\Gamma)$ for the set of $(\varepsilon, \cG)$-representations of $\Gamma$ on $P$. We define the metrics on $\Bdl_P^{\varepsilon , \cU}(X)$ and $\qRep_{P}^{\varepsilon , \cG}(\Gamma)$ as $d(\bv, \bv'):= \max _{\mu,\nu} \| v_{\mu\nu} -v'_{\mu\nu} \| $ and $d(\pi, \pi') = \sup_{\gamma \in \cG} \| \pi(\gamma) - \pi'(\gamma) \|$ respectively.
\end{defn}

\begin{rmk}\label{rmk:MF}
The bundle $E_\bv$ associated to a $\mathrm{U}(P)$-valued \v{C}ech $1$-cocycle is constructed as following. As in (\ref{form:projMF}), let $ \{ \eta_\mu \}_{\mu \in I}$ be a family of positive continuous functions on $X$ such that $\sum_{\mu \in I} \eta_\mu ^2 =1 $ and let $e_{\mu \nu} \in \bM_{I}$ denotes the matrix element, i.e., $e_{\mu \nu} e_\sigma = \delta_{\nu , \sigma} e_\mu$ where $e _\mu$ is the standard basis of $\bC^I$. Let
\begin{align*}
p_{\bv}(x)&:= \sum_{\mu , \nu} \eta _\mu(x) \eta_\nu(x) v_{\mu \nu}(x) \otimes e_{\mu\nu} \in C(X) \otimes \bB(P) \otimes \bM_I, \\
\psi_\mu^\bv (x) &:= \sum_{\nu } \eta_\nu (x) v_{\nu \mu} (x)  \otimes e_\nu \in C_b(U_\mu) \otimes \bB(P) \otimes \bC^I.
\end{align*}
Then we have $p_\bv(x) \psi_\mu ^{\bv}(x)=\psi_\mu(x)$ for $x \in U_\mu$ and  $\psi_\mu^{\bv}(x)^*\psi_{\nu}^\bv(x) = v_{\mu \nu}(x)$ for $x\in U_{\mu \nu}$. That is, $p_\bv$ is a projection with the support $E_\bv$ and $\psi_\mu ^\bv$ is a local trivialization of $E_\bv$.
\end{rmk}

It is essentially proved in \cite[Theorem 3.1, Theorem 3.2]{mathOA150306170} (see also \cite[Lemma 6.9]{Kubota3}) that there is a constant $C>0$ depending only on $\cU$ and maps
\begin{align}
\begin{split}
\alpha& \colon \Bdl_{P}^{\varepsilon ,\cU} (X)_T \to \qRep ^{C\varepsilon, \cG}_{P}(\Gamma ), \\
\beta& \colon \qRep^{ \varepsilon, \cG}_{P} (\Gamma) \to \Bdl_{P}^{C \varepsilon, \cU} (X)_T,
\end{split}
\label{form:beta}
\end{align}
satisfying 
\begin{itemize}
\item $d(\alpha (\bv), \alpha (\bv ' )) \leq d(\bv, \bv') + C \varepsilon$, $d(\beta \circ \alpha  (\bv) , \bv) \leq C \varepsilon$ for any $\bv, \bv' \in \Bdl_P^{\varepsilon, \cU} (X)$, and 
\item $d(\beta (\pi), \beta (\pi ')) \leq d(\pi , \pi' ) + C \varepsilon$, $d(\alpha \circ \beta (\pi), \pi) \leq C \varepsilon$ for any $\pi, \pi' \in \qRep _P^{\varepsilon, \cG}(\Gamma )$.
\end{itemize}
\begin{rmk}\label{rmk:beta}
The construction of the map $\beta$ is essentially given in \cite[Lemma 4.4]{Kubota3} (see also Definition 6.7). Here it is mentioned that, for a $(\varepsilon, \cG)$-representation $\pi$ of $\Gamma$, the associated bundle $\bv:=\beta (\pi)$ satisfies $\|v_{\mu \nu}(x) - \pi(\gamma _{\mu\nu}) \|<4\varepsilon$. Indeed, this inequality characterizes $\beta(\pi)$ up to small correction. 
\end{rmk}

\begin{defn}
Let $(X,Y)$ be a pair of compact spaces. A \emph{stably relative bundle} on $(X,Y)$ with the typical fiber $(P,Q)$ is a quadruple $(E_1,E_2, E_0, u)$, where $E_1$ and $E_2$ are $P$-bundles on $X$, $E_0$ is a $Q$-bundle on $Y$ and $u$ is a unitary bundle isomorphism $E_1|_Y \oplus E_0 \to E_2|_Y \oplus E_0$. 
\end{defn}
A stably relative bundle of Hilbert $\bC$-modules with the typical fiber $(\bC^n, \bC^m)$ is simply called a stably relative vector bundle of rank $(n,m)$. We simply call a stably relative bundle of the form $(E_1,E_2,0,u)$ a \emph{relative bundle}. 
\begin{rmk}\label{rmk:module}
A stably relative bundle associates an element of the relative $\K^0$-group $\K^0(X,Y;A):=\K_0(C_0(X_1^\circ ) \otimes A)$ in the following way. 
Let $f_1(r):=\min \{ 1 , \max \{ 0 , 1-3r \} \}$ and $f_2(r):= \min \{ 1 , \max \{ 0 , 3r-2 \} \}$.
The inverse of $\kappa$ is given by mapping $(E_1,E_2,E_0,u)$ to 
\[ [E_1,E_2,E_0,u] :=[\sE_1 \oplus \sE_2^\op, 1,  \big( \begin{smallmatrix}0 & \tilde{u}^* \\ \tilde{u} & 0\end{smallmatrix} \big) ] \in \KK(\bC, C_0(X^\circ_1) \otimes A),\]
where 
\begin{align*}
 \sE_1&:=C_0(X^\circ_1, E_1) \oplus C_0(Y_1^\circ , E_0), \\ 
\sE_2&:=C_0(X^\circ_1, E_2) \oplus C_0(Y_1^\circ , E_0), \\
\tilde{u}&:= f_1(r) 1_{E_0} + f_2(r) u \in \bB(\sE_1, \sE_2).
\end{align*}
In particular, $[E_1, E_2, E_0, u]=0$ if $E_1=E_2$ and $u=1_{E_1 |_Y \oplus E_0}$.
\end{rmk}

Let $(X,Y)$ be a pair of finite CW-complexes. We say that a good open cover is a good open cover $\cU=\{ U_\mu \}_{mu \in I}$ of $X$ such that $\cU|_Y := \{ Y \cap U_\mu \}$ is a good cover of $Y$.  
\begin{defn}
Let $(X,Y)$ and $\cU$ be as above and let $P$ be a finitely generated Hilbert $A$-module.  Let $T$ be a maximal subtree of the $1$-skeleton of $N(\cU)$ such that $T|_{N(\cU|_Y)}$ is also a maximal subtree. 
\begin{itemize}
\item For two $(\varepsilon , \cU)$-flat bundles $\bv_1=\{ v_{\mu\nu}^1\}$ and $\bv_2=\{ v_{\mu\nu}^2\}$, a \emph{morphism of $(\varepsilon , \cU)$-flat bundles} $\mathbf{u} \in \Hom_\varepsilon (\bv_1, \bv_2)$ is a family of unitaries $\bu=\{ u_\mu \} _{\mu \in I} \in \mathrm{U}(P)^I$ such that 
\[ \sup _{\mu, \nu \in I} \sup_{x \in U_{\mu\nu}} \| u_\mu v_{\mu \nu}^{1}(x) u_\nu^* -v_{\mu\nu}^2(x) \|  < \varepsilon . \]
\item An \emph{$(\varepsilon , \cU )$-flat stably relative bundle} on $(X,Y)$ with the typical fiber $(P , Q)$ is a quadruple $\fv:=(\bv_1,\bv_2,\bv_0, \bu)$, where 
\begin{itemize}
\item $\bv_1 $ and $\bv_2 $ are $(\varepsilon, \cU)$-flat $P$-bundle on $X$, 
\item $\bv_0$ is a $(\varepsilon , \cU _Y)$-flat $Q$-bundle on $Y$ and
\item $\bu \in \Hom_{\varepsilon} (\bv_1|_Y \oplus \bv_0 ,  \bv_2|_Y \oplus \bv_0)$.
\end{itemize}
It is said to be normalized on $T$ if $\bv_1$, $\bv_2$ are normalized at $T$ and $\bv_0$ is normalized at $T \cap Y$.
\end{itemize}
We write the set of $(\varepsilon , \cU )$-flat stably relative bundles on $(X,Y)$ normalized at $T$ with the typical fiber $(P , Q)$ as $\Bdl _{P,Q}^{\varepsilon , \cU}(X,Y)_T$. We define the metric on $\Bdl _{P,Q}^{\varepsilon , \cU}(X,Y)_T$ as $d(\fv,\fv') := \max \{ d(\bv_1, \bv_1') , d(\bv_2,\bv_2'), d(\bv_0,\bv_0'), d(\bu,\bu') \}$, where $d(\bu,\bu'):= \max _\mu \| u_\mu - u_\mu' \|$. 
\end{defn}
\begin{rmk}\label{rmk:relative}
An $(\varepsilon , \cU)$-flat stably relative bundle $\fv=(\bv_1,\bv_2,\bv_0,\bu)$ associates an element of the relative $\K^0$-group $\K^0(X,Y,A)$. We give some remarks on this element.
\begin{enumerate}
\item If $\varepsilon >0$ is sufficiently small and $\fv , \fv' \in \Bdl^{\varepsilon, \cU} _{P,Q}(X,Y)$ satisfies $d(\fv, \fv') < \varepsilon$, then we have $[\fv]=[\fv']$ (\cite[Lemma 6.11]{Kubota3}).  
\item We quickly remind the definition of $[\fv]$. As is proved in \cite[Lemma 3.4]{Kubota3}, there is a family $\{ \bar{u}_\mu \colon U_\mu \to \mathrm{U}(P \oplus Q) \}_{\mu \in I}$ such that $\bar{u}_\mu (v_{\mu\nu}^1 \oplus v_{\mu\nu}^0 ) \bar{u}_\nu^* = v_{\mu\nu}^2 \oplus v_{\mu\nu}^0 $ and $\| \bar{u}_\mu - u_\mu \| < C \varepsilon$, where $C>0$ is a constant depending only on $\cU$. This family $\{ \bar{u}_\mu \}_{\mu \in I}$ induces a bundle map $\bar{u} \colon E_{\bv_1|_Y} \oplus E_{\bv_0} \to E_{\bv_2|_Y} \oplus E_{\bv_0}$. 
\item Let $\{ \eta_\mu \}$ and $e_{\mu\nu}$ be as in Remark \ref{rmk:MF}. The element
\[ \bar{w} := \sum \eta_\mu \eta_\nu \cdot  (v_{\mu\nu}^2 \oplus v_{\mu\nu}^0)\bar{u}_\nu \otimes e_{\mu
\nu} \in C(Y) \otimes \bB(P) \otimes \bM_I \]
is a partial isometry such that $\bar{w}^*\bar{w}= p_{\bv_1|_Y} \oplus p_{\bv_0}$, $\bar{w}\bar{w}^*= p_{\bv_2|_Y} \oplus p_{\bv_0}$ and 
\[ (\psi_{\bv_2|_Y\oplus \bv_0})^* \bar{w}(\psi_{\bv_1|_Y \oplus \bv_0}) = \bar{u}_\mu. \]
That is, $w$ is identified with $\bar{u}$ in (2) under the canonical isomorphism $E_{\bv_i|_Y} \cong  p_{\bv_1|_Y} \underline{P}^I_Y$ for $i=1,2$ and $E_{\bv_0} \cong p_{\bv_0}\underline{Q}^I_Y$. We remark that this $\bar{w}$ satisfies
\[ \| \bar{w} -  p_{\bv_2} \cdot \mathrm{diag} (u_\mu )_{\mu \in I}  \| \leq |I|^2 \cdot \sup_{\mu \in I} \| \bar{u}_\mu-u_\mu \| <|I|^2 \varepsilon, \]
where $\mathrm{diag}(u_\mu)_{\mu \in I}$ is a unitary in $\bB(P) \otimes \bM_I$.

\end{enumerate}
\end{rmk} 

We say that an element $\xi \in \K^0(X,Y ; A)$ is \emph{(resp.\ stably) almost flat} with respect to a good open cover $\cU$ if for any $\varepsilon >0$ there is a $(\varepsilon, \cU)$-flat (resp.\ stably) relative vector bundle $\fv$ of finitely generated projective Hilbert $A$-modules such that $x=[\fv]$. It is shown in \cite[Corollary 3.16]{Kubota3} that (stably) almost flatness does not depend on the choice of open covers $\cU$. 
\begin{defn}
Let $(X,Y)$ be a pair of finite CW-complexes.
\begin{enumerate}
\item We write $\K_{\mathrm{af}}^0(X,Y;A)$ (resp.\ $\K_{\mathrm{s\mathchar`-af}}^0(X,Y;A)$) for the subgroup of (resp.\ stably) almost flat elements.
\item We say that a $\K$-homology class $\xi \in \K_*(X,Y)$ has \emph{infinite (resp.\ stably) relative $\K$-area} if there is an (resp.\ stably) almost flat $\K$-theory class $x \in \K^0(M,N)$ such that the index pairing $ \langle x, \xi \rangle$ is non-zero.
\item We say that $\xi \in \K_*(X,Y)$ has \emph{infinite (resp.\ stably) relative C*-$\K$-area} if for any $\varepsilon>0$ there is a C*-algebra $A_\varepsilon$ and a (resp.\ stably) relative $(\varepsilon , \cU)$-flat bundle $\fv$ of finitely generated projective Hilbert $A_\varepsilon$-modules such that the index pairing $ \langle [\fv], \xi \rangle \in \K_0(A_\varepsilon)$ is non-zero. 
\end{enumerate}
\end{defn}
In particular, we say that a spin manifold $M$ with the boundary $N$ has (stably) relative infinite (C*-)K-area if so is the K-homology fundamental class $[M,N] \in \K_*(M,N)$.

\begin{thm}[{\cite[Theorem 5.1]{Kubota3}}]\label{thm:enlarge}
Let $M$ be a compact spin manifold with boundary $N$. If $M_\infty := M \sqcup_N N \times [0,\infty)$ is area-enlargeable, then $(M,N)$ has infinite stably relative C*-K-area.
\end{thm}

Finally we review the almost monodromy correspondence in the relative setting.

\begin{defn}
Let $(\Gamma , \Lambda)$ be a pair of discrete groups and let $\phi \colon \Lambda \to \Gamma $ be a homomorphism.
\begin{itemize}
\item Let $\pi_1$ and $\pi_2$ be $(\varepsilon , \cG)$-representations of $\Gamma$. An \emph{$\varepsilon$-intertwiner} $u \in \Hom _\varepsilon ( \pi_1 , \pi_2)$ is a unitary $u \in \mathrm{U}(P)$ such that $\| u\pi_1(\gamma )u^* - \pi_2(\gamma) \|< \varepsilon$ for any $\gamma \in \cG$.
\item A \emph{stably relative $(\varepsilon , \cG )$-representation} of $(\Gamma, \Lambda) $ is a quadruple $\boldsymbol{\pi}:=(\pi _1, \pi_2, \pi_0 , u)$, where 
\begin{itemize}
\item $\pi_1 \colon \Gamma \to \mathrm{U}(P)$ and $\pi_2 \colon \Gamma \to \mathrm{U}(P)$ are $(\varepsilon , \cG _\Gamma) $-representations,
\item $\pi_0 \colon \Lambda \to \mathrm{U}(Q)$ is a $(\varepsilon , \cG_\Lambda )$-representation, and
\item $u \in \Hom _\varepsilon (\pi_1 \circ \phi \oplus \pi_0 ,  \pi_2 \circ \phi \oplus \pi_0)$.
\end{itemize}
\end{itemize}
We write $\qRep _{P,Q}^{\varepsilon , \cG} (\Gamma , \Lambda )$ for the set of stably relative $( \varepsilon , \cG)$-representations of $(\Gamma, \Lambda )$ on $(P,Q)$. We define the metric on $\qRep _{P,Q}^{\varepsilon , \cG} (\Gamma , \Lambda )$ as $d(\boldsymbol{\pi},\boldsymbol{\pi}') := \max \{ d(\pi_1, \pi_1') , d(\pi_2,\pi_2'), d(\pi_0,\pi_0'), \| u - u'\| \}$. 
\end{defn}

\begin{thm}[{\cite[Definition 6.11, Theorem 6.12]{Kubota3}}]\label{thm:monodromy}
There is a constant $C_{\mathrm{am}}>0$ depending only on $\cU$ and continuous maps
\begin{align*}
\boldsymbol{\alpha}& \colon \Bdl_{P,Q}^{\varepsilon ,\cU} (X,Y)_T \to \qRep ^{C_{\mathrm{am}}\varepsilon, \cG}_{P,Q}(\Gamma, \Lambda ), \\
\boldsymbol{\beta}& \colon \qRep^{ \varepsilon, \cG}_{P,Q} (\Gamma, \Lambda) \to \Bdl_{P,Q}^{C_{\mathrm{am}} \varepsilon, \cU} (X,Y)_T,
\end{align*}
satisfying 
\begin{enumerate}
\item For $\fv, \fv' \in \Bdl^{\varepsilon,\cU}_{P,Q} (X,Y)_T$, we have $d(\boldsymbol{\alpha} (\fv) , \boldsymbol{\alpha}(\fv')) \leq d(\fv , \fv') + C_{\mathrm{am}}\varepsilon$ and $d(\boldsymbol{\beta} \circ  \boldsymbol{\alpha}(\fv ) , \fv ) \leq C_{\mathrm{am}}\varepsilon$.
\item For $\boldsymbol{\pi}, \boldsymbol{\pi}' \in \qRep^{\varepsilon, \cG}_{P,Q} (\Gamma, \Lambda)$, we have $d(\boldsymbol{\beta}(\boldsymbol{\pi}), \boldsymbol{\beta} ( \boldsymbol{\pi}')) \leq d(\boldsymbol{\pi}, \boldsymbol{\pi}') + C_{\mathrm{am}} \varepsilon$ and $d(\boldsymbol{\alpha} \circ \boldsymbol{\beta }(\boldsymbol{\pi}) , \boldsymbol{\pi} )\leq C_{\mathrm{am}} \varepsilon$.
\end{enumerate}

\end{thm}
For the latter use, we only remind the definition of $\boldsymbol{\beta}$ given in \cite[Definition 6.10]{Kubota3}. For a $(\varepsilon , \cG)$-representation $\boldsymbol{\pi}=(\pi_1,\pi_2,\pi_0,u)$ of $(\Gamma, \Lambda)$, 
\begin{align}
\boldsymbol{\beta} (\boldsymbol{\pi}) := (\beta (\pi_1), \beta (\pi_2) , \beta (\pi_0) , \Delta _I(u)), \label{form:bbeta}
\end{align}
where $\beta$ is the map in (\ref{form:beta}) and $\Delta_I \colon \mathrm{U}(P \oplus Q) \to \mathrm{U}(P \oplus Q)^I$ is the diagonal embedding.

\section{Relative index pairing with coefficient in a C*-algebra}\label{section:3}
In this section, we establish an obstruction for the relative higher index to vanish arising from an index pairing with coefficient in a C*-algebra.
It has two applications; a relative version of the Hanke--Schick theorem \cite{MR2259056,MR2353861} and the non-vanishing of relative higher index in the setting of Hanke--Pape--Schick~\cite{MR3449594}.

\subsection{Index pairing with stably h-relative representations}
For a representation of the fundamental group $\Gamma = \pi_1(M)$ of a closed spin manifold $M$ on a finitely generated projective Hilbert $A$-module $P$, i.e., a homomorphism $\pi \colon \Gamma \to \mathrm{U}(P)$, the Kasparov product $\alpha_{\Gamma}([M]) \hotimes_{C^*\Gamma} [\pi] \in \K_0(A)$ coincides with the index pairing with the flat $P$-bundle associated to $\pi$. 
Here we develop its relative version. The relative counterpart of group representation is a pair of representations of $\Gamma$ whose restriction to $\Lambda$ are identified `up to stabilization and homotopy' in the following sense. 
\begin{defn}
Let $A$ be a unital C*-algebra and let $P_1$, $P_2$, $Q$ be finitely generated projective Hilbert $A$-modules. A \emph{stably h-relative representation} of $(\Gamma, \Lambda)$ on $(P_1,P_2,Q)$ is a quintuple $\Pi:= (\pi_1,\pi_2, \pi_0 , u, \tilde{\pi} )$, where
\begin{itemize}
\item $\pi_i \colon \Gamma \to \mathrm{U}(P_i)$ for $i=1,2$ and $\pi_0 \colon \Lambda \to \mathrm{U}(Q) $ are representations,
\item $u \colon P_1 \oplus Q \to P_2 \oplus Q$ be a unitary, and
\item $\tilde{\pi}=\{ \tilde{\pi}_\kappa\}_{\kappa \in [1,2]}$ is a continuous family of representations of $\Lambda $ to $P_2 \oplus Q$ (that is, $\tilde{\pi}$ is a homomorphism from $\Lambda$ to $\mathrm{U}(\bB(P_2 \oplus Q)[1,2])$) such that $\tilde{\pi}_{1}=\Ad (u) \circ ( \pi_1 \circ \phi \oplus \pi_0 )$ and $\tilde{\pi}_{2}=\pi_2 \circ \phi \oplus \pi_0$.
\end{itemize}
\end{defn}

We associate the following two objects to a stably h-relative representation. 
First, let $\cP_i := \tilde{X} \times _{\Gamma, \pi_i} P$ for $i=1,2$, let $\cQ:= \tilde{Y} \times _{\Lambda , \pi_0 } Q$ and let $V_\kappa$ be a continuous family of bundle isomorphisms
\[ V_\kappa \colon \tilde{Y} \times _{\Lambda , \tilde{\pi}_\kappa } (P_1 \oplus Q) \to \tilde{Y} \times _{\Lambda, \tilde{\pi}_2 } (P_2 \oplus Q) \]
for $\kappa \in [1,2]$ such that $V_2$ is the identity. Note that such $V_\kappa$ exists and unique up to homotopy. Then, $(\cP_1, \cP_2, \cQ, V_1u)$ is a stably relative Hilbert $A$-module bundle with the typical fiber $(P,Q)$.

Second, let $\tilde{P}_i$ denote the Hilbert $A(-1,1)$-module $\tilde{P}_i:=P_i(-1,1) \oplus Q(-1,0)$. We define a $\KK$-class
\begin{align}
\boldsymbol{\Pi} = \bigg[ \tilde{P}_1 \oplus \tilde{P}_2, \Pi_1 \oplus \Pi_2, \begin{pmatrix} 0 & U^* \\ U & 0 \end{pmatrix} \bigg] \in \KK(C\phi , A(-1,1)),
\label{form:qhom}
\end{align}
where
\begin{align*}
\Pi_1(a,b_s)(s)&:=\left\{ \begin{array}{ll}(\pi _1 \oplus \pi_0 \circ \phi )(a) & s \in (-1,0), \\ \pi_1(b_s) & s \in [0,1), \end{array}\right. \\
\Pi_2(a,b_s)(s)&:=\left\{ \begin{array}{ll}\tilde{\pi}_{2+s}(a) & s \in (-1,0), \\ \pi_{2}(b_s) & s \in [0,1), \end{array}\right. 
\end{align*}
and $U$ is defined by using functions $f_1$ and $f_2$ as in Remark \ref{rmk:module} as
\[ U:= f_1(-s) 1_{Q} + f_2(-s) \bar{u}  \in \bB(\tilde{P}). \]
By a reparametrization of $\tilde{\pi}_\kappa$, we may assume that $\tilde{\pi}_\kappa = \tilde{\pi}_1$ for $\kappa \in [1,\frac 4 3]$ and $\tilde{\pi}_\kappa = \tilde{\pi}_2$ for $\kappa \in [\frac 5 3, 2]$. Then $U$ intertwines $\Pi_1$ with $\Pi_2$, that is, $U\Pi_1(x)=\Pi_2(x)U$ for any $x \in C\phi$.

\begin{thm}\label{prp:bundle}
The Kasparov product $\ell_{\Gamma, \Lambda } \otimes_{C^*(\Gamma, \Lambda)} \boldsymbol{\Pi} \in \KK (\bC, C_0(X^\circ ) \otimes A) $ is represented by the stably relative bundle $(\cP_1, \cP_2, \cQ, V_1 u )$ on $(X,Y)$.
\end{thm}
For the proof, we use the following lemma. 
\begin{lem}[{\cite[Lemma A.2]{Kubota1}}]\label{lem:Kas}
Let $A$, $B$ and $D$ be $\sigma$-unital Real C*-algebras such that $A$ is separable, let $(E _1, \pi _1, T_1)$ be a real self-adjoint Kasparov $A$-$B$ bimodule and let $(E_2, \varphi_2 , F_2)$ be a real Kasparov $B$-$D$ bimodule. Set $E:=E_1 \otimes _B E_2$ with the trivial $\bZ_2$-grading, $\pi :=\pi_1 \otimes _B 1$ and $\tilde{T}_1:=T_1 \otimes _B 1$. Let $G=\big( \begin{smallmatrix} 0 & G_0^* \\ G_0 & 0\end{smallmatrix}
\big) \in \bB (E)$ be an odd $F$-connection and assume that $[\pi(A), T] \subset \bK(E)$, where
\[ T= \begin{pmatrix}\tilde{T}_1 & (1-\tilde{T}^2_1)^{1/4}G^*_0 (1-\tilde{T}^2_1)^{1/4} \\  (1-\tilde{T}^2_1)^{1/4}G_0 (1-\tilde{T}^2_1)^{1/4} & -\tilde{T}_1 \end{pmatrix} \in \bB(E). \]
Then, the real self-adjoint Kasparov $A$-$D$ bimodule $(E , \pi, T)$ represents the Kasparov product $[E_1,\pi_1, T_1] \hotimes_B [E_2, \pi_2 , F_2]$.
\end{lem}

\begin{proof}[Proof of Theorem \ref{prp:bundle}]
The Hilbert $C_0(X^\circ_2) \otimes A$-module $\sE_2 \otimes_{\Pi_2} \tilde{P}_2$ is the section space of the continuous field 
\[ \tilde{\cP}_2:= \bigsqcup_{s \in (0,1)} \cP_2 \cup \bigsqcup_{s \in (-1,0]} \tilde{Y} \times _{\tilde{\pi}_{2+s}} (P \oplus Q) \]
of Hilbert $A$-modules over $X_2^\circ \times (-1,1)$. 
Let $Z$ denote its support, that is, $Z:=X_2^\circ (0,1) \cup (Y_2')^\circ (-1,0]$.  

For $i=1,2$, set
\[ \bar{\sP} _i := C_0(Z, \cP _i) \oplus C_0((Y_2')^\circ , \cQ). \]
Then $\bar{\sP}_1$ is canonically identified with $\sE_2 \otimes_{\tilde{\Pi}_1} \tilde{P}$ and
\[ \bar{V} (\varphi)(x,s) = \left\{ \begin{array}{ll}\varphi(x,s) & s \in (0,1), \ x \in X_2^\circ , \\ V_{2+s}(\varphi(x, s)) & s \in (-1,0], \ x \in (Y_2')^\circ, \end{array} \right. \]
gives a unitary isomorphism 
\[ \bar{V} \colon \sE \otimes _{\tilde{\Pi}_2} \tilde{P} \to \bar{\sP}_2 . \]
Moreover, since $U$ intertwines $\Pi_1$ with $\Pi_2$, it induces an operator
\[ \bar{U} \colon \sE_2 \otimes_{\tilde{\Pi}_1} \tilde{P} \to \sE_2 \otimes_{\tilde{\Pi}_2} \tilde{P}.\]
In particular, $\bar{U}$ is a $U$-connection. By Lemma \ref{lem:Kas} we obtain that
\begin{align*}
\ell_{\Gamma , \Lambda} \otimes _{C\phi} \boldsymbol{\Pi} &= \bigg[ (\sE_2 \otimes_{\tilde{\Pi}_1} \tilde{P}_1) \oplus (\sE_2 \otimes_{\tilde{\Pi}_2} \tilde{P}_2)^{\mathrm{op}}, 1, \begin{pmatrix}\bar{\rho} & \bar{\sigma}^2\bar{U}^* \\ \bar{\sigma}^2\bar{U} & -\bar{\rho} \end{pmatrix} \bigg] \\
&=\bigg[ \bar{\sP}_1 \oplus \bar{\sP}_2 , 1, \begin{pmatrix}\bar{\rho} & \bar{\sigma}^2\bar{U}^*\bar{V}^* \\ \bar{\sigma}^2\bar{V}\bar{U} & -\bar{\rho} \end{pmatrix} \bigg]_{\textstyle ,}
\end{align*}
where 
\[ \bar{\rho}(x,s):=(\rho \otimes_{\Pi_i} 1)(x,s)= \left\{ \begin{array}{ll}\rho_s(x) & (x,s) \in X^\circ_2 (0,1), \\ \rho_0(x) & (x,s) \in (Y_2')^\circ (-1,0], \end{array} \right. \]
and $\bar{\sigma}=(1-\bar{\rho}^2)^{1/4}$. Note that $\bar{V}\bar{U}=f_1(-s)1_{\cQ} + f_2(-s)V_1u$.

On the other hand, let $\tilde{\sP}_i:=C_0(X_2^\circ , \cP_i) \oplus C_0((Y_2')^\circ , \cQ)$ for $i=1,2$ and $\tilde{U}:=f_1(r-1)1_{\cQ}+f_2(r-1)V_1u$. As is mentioned in Remark \ref{rmk:module}, we have
\[ [\cP_1 , \cP_2 , \cQ , V_1u] = \bigg[ \tilde{\sP}_1 \oplus \tilde{\sP}_2^{\mathrm{op}}, 1 , \begin{pmatrix} 0 & \tilde{U} \\ \tilde{U}^* & 0 \end{pmatrix} \bigg]_{\textstyle .} \]
Hence Lemma \ref{lem:Kas} implies that
\[ \beta \otimes [\cP_1,\cP_2,\cQ,V_2 ] = \bigg[ \tilde{\sP}_1(0,1)  \oplus \tilde{\sP}_2 ^{\mathrm{op}}(0,1) , 1, \begin{pmatrix}2s-1 & \tau^2
 \tilde{U}^* \\ \tau^2 \tilde{U} & 1-2s \end{pmatrix} \bigg]_{\textstyle ,}  \]
where $\tau=(1-(2s-1)^2)^{1/4}$. 

\begin{figure}[t]
\begin{tikzpicture}[scale=0.4]
\shade [top color = black!90!white, bottom color = black!90!white, middle color = white, shading angle = 135] (-3,-2) rectangle (6,4);
\shade [top color = black!42!white, bottom color = black!42!white, middle color = white] (-3,0) rectangle (0,4);
\shade [left color = black!42!white, right color = black!42!white, middle color = white] (0,0) rectangle (6,-4);
\fill [color = black!42!white] (6,0) -- (6,4) -- (0,4) -- (6,0);
\fill [color = white] (-3.1,-2.1) rectangle (0,0);
\draw [->] (-4.5,-4.3) -- (-4.5,4.5);
\draw [->] (-1.5,-5) -- (7.5,-5);
\draw [thick,dashed] (-3,4) -- (6,4);
\draw [thick, dashed] (6,4) -- (6,-4);
\draw (6,0) -- (0,0);
\draw (0,0) -- (0,4);
\draw [thick,dashed] (0,0) -- (-3,0);
\draw [thick,dashed] (0,0) -- (0,-4);
\draw [thick] (0,-4) -- (6,-4);
\node at (6.3,-4.6) {$r$};
\node at (-3.4,4.3) {$s$};
\node at (-1.5,5) {$X_1$};
\node at (3,5) {$Y_2'$};
\fill (6,-5) coordinate (O) circle[radius=1pt] node[below=2pt]  {$2$};
\fill (0,-5) coordinate (O) circle[radius=1pt] node[below=2pt]  {$1$};
\fill (-4.5,4) coordinate (O) circle[radius=1pt] node[left=2pt]  {$1$};
\fill (-4.5,0) coordinate (O) circle[radius=1pt] node[left=2pt]  {$0$};
\fill (-4.5,-4) coordinate (O) circle[radius=1pt] node[left=2pt]  {$-1$};
\end{tikzpicture}
\begin{tikzpicture}[scale=0.4]
\shade [top color = black!42!white, bottom color = black!42!white, middle color = white] (-3,-2) rectangle (6,2);
\draw [->] (-4.5,-3.8) -- (-4.5,3.8);
\draw [->] (-1.5,-3.5) -- (7.5,-3.5);
\draw [thick,dashed] (-3,2) -- (6,2);
\draw [thick, dashed] (6,2) -- (6,-2);
\draw [thick, dashed] (6,-2) -- (-3,-2);
\draw (0,-2) -- (0,2);
\node at (6.4 ,-2.6) {$r$};
\node at (-3.6,3.7) {$s$};
\node at (-1.5,3) {$X_1$};
\node at (3,3) {$Y_2'$};
\fill (6,-3.5) coordinate (O) circle[radius=1pt] node[below=2pt]  {$2$};
\fill (0,-3.5) coordinate (O) circle[radius=1pt] node[below=2pt]  {$1$};
\fill (-4.5,2) coordinate (O) circle[radius=1pt] node[left=2pt]  {$1$};
\fill (-4.5,-2) coordinate (O) circle[radius=1pt] node[left=2pt]  {$0$};
\end{tikzpicture}
\caption{The shading shows the value of $|\rho (r,s)|$ on $Z$ and $|2s-1|$ on $X(0,1)$ respectively.}
\end{figure}
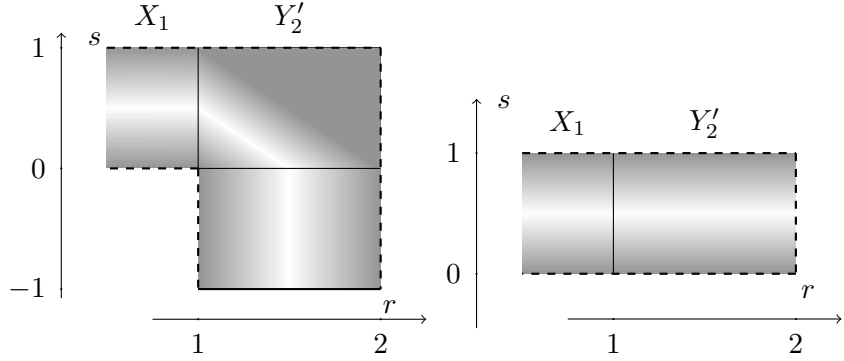

Let $\iota \colon Z \to X(-1,1)$ denote the open embedding. We define a continuous map $f \colon Z \to X(0,1)$ by $f(x,s)=(x,s)$ for $(x,s) \in X_1^\circ (0,1)$ and 
\begin{align*} f(y,r,s)= \left\{ \begin{array}{ll}(y,1-s,{\scriptstyle \frac{\bar{\rho}(r,s)+1}{2}}) & (y,r,s) \in (Y_2')^\circ(-1,0), \\ (y, 1 , {\scriptstyle \frac{\bar{\rho}(r,s)+1}{2}}) & (y,r,s) \in (Y_2')^\circ(0,1). \end{array} \right. 
\end{align*}
Then the $\ast$-homomorphism $f^* \colon C_0(X_2^\circ(0,1)) \to C_0(Z)$ satisfies $f^*(2s-1)=\bar{\rho} \in C_0(Z)$. Moreover, by the constructions, there are unitaries $\Phi_i \colon \tilde{\sP}_i \otimes _{f^*} C_0(Z) \to \bar{\sP_i}$ of Hilbert $C_0(Z) \otimes A$-modules for $i=1,2$ such that $\Phi_2  (\tilde{U} \otimes _{f^*} 1) \Phi _1^* = \bar{V}_1\bar{U}$.
Consequently we obtain that
\[ \ell_{\Gamma, \Lambda } \otimes \boldsymbol{\Pi} = (\beta \otimes [\cP_1, \cP_2, \cQ , V_2] ) \otimes [f^*] \otimes [\iota_*]. \] 
This concludes the proof since $\iota_* \circ f^* \colon C_0(X^\circ_1(0,1)) \to C_0(X^\circ_1(-1,1)) $ is homotopic to the inclusion $X_2^\circ(0,1) \to X_2^\circ(-1,1)$.
\end{proof}

\subsection{Relative Hanke--Schick obstruction}
First, we apply Theorem \ref{prp:bundle} to show a relative version of \cite[Theorem 3.9]{MR3289846}. 

\begin{thm}\label{thm:Karea}
Let $M$ be a compact spin manifold with boundary $N$. Let $\Gamma :=\pi_1(M)$, $\Lambda := \pi_1(N)$ and let $\phi$ be the homomorphism induced from the inclusion $N \to M$.
\begin{enumerate}
\item If $M$ has infinite stably relative C*-$\K$-area, then the relative higher index $\mu_*^{\Gamma, \Lambda} ([M,N])$ does not vanish. 
\item If $M$ has infinite relative C*-$\K$-area, then the relative higher index $\mu_*^{\Gamma, \phi(\Lambda)} ([M,N])$ does not vanish. 
\end{enumerate}
\end{thm}
\begin{proof}
First we show (1). By the assumption, for each $n \in \bN$ there is a C*-algebra $A_n$, a pair of finitely generated projective Hilbert $A_n$-modules $(P_n, Q_n)$ and a $(\frac{1}{n} , \cU)$-flat stably relative bundle $\fv_n:=(\bv_n^1, \bv_n^2, \bv_n^0, \bu_n)$ with the typical fiber $(P_n,Q_n)$ such that $\langle [\fv_n] , [M,N] \rangle \neq 0  \in \K_0(A_n)$. 
Set 
\begin{align*}
B&:=\prod_{n \in \bN} \bB(P_n \oplus Q_n ),\\
p&:=\prod 1_{P_n}, \ \ P=pB,\\
q&:=\prod 1_{Q_n}, \ \ Q=qB.
\end{align*}
We define the stably relative bundle $\fv=(\bv^1, \bv^2,\bv^0 , \bu)$ with the typical fiber $(P,Q)$ as $\bv^i=\{ v_{\mu \nu}^i\}_{\mu , \nu \in I}$, $\bv^0= \{ v_{\mu\nu}^0 \}_{\mu , \nu \in I} $ and $\bu=\{ u_\mu \}_{\mu \in I}$, where
\begin{align*}
v^i_{\mu \nu}(x) &:= \prod_{n \in \bN} (v_n^i)_{\mu\nu} (x) \in \bB(P),\\
v^0_{\mu \nu}(y) &:= \prod_{n \in \bN} (v_n^0)_{\mu \nu}(y) \in \bB(Q),\\
u_\mu &:= \prod_{n \in \bN} (u_n )_\mu \in \bB(P \oplus Q),
\end{align*}
for $i=1,2$, $x \in U_{\mu \nu}$ and $y \in U_{\mu \nu } \cap N$.  

Let $J:=\bigoplus_{n \in \bN} \bB(P_n \oplus Q_n)$, $D=B/J$ and $\tau \colon B \to D$ denote the quotient. Then we have
\begin{itemize} 
\item $v^i_{\mu \nu}(x) v^i_{\nu \sigma}(x) - v_{\mu \sigma}^i(x) \in J$ and
\item $(v^1_{\mu\nu} \oplus v^0_{\mu \nu})(y) - u_\mu(v^2_{\mu\nu} \oplus v^0_{\mu \nu})(y)u_\nu^* \in J$,
\end{itemize}
that is, 
\[ \tau_*\fv:=(\{ \tau(v^1_{\mu \nu})\}  ,\{ \tau(v^2_{\mu \nu})\}  ,\{ \tau(v^0_{\mu \nu})\}  ,\{ \tau(u_{\mu })\}  )\]
is a stably relative flat bundle. Let $\boldsymbol{\Pi} \in \KK(C^*(\Gamma, \Lambda) , D)$ denote the Kasparov bimodule associated to the stably relative representation $\boldsymbol{\alpha} (\tau_*\fv)$ as in Theorem \ref{thm:monodromy}. By Theorem \ref{prp:bundle} we obtain that
\begin{align*}
\alpha_{\Gamma, \Lambda}([M,N]) \hotimes \boldsymbol{\Pi} & = \ell_{\Gamma, \Lambda} \otimes_{C_0(M^\circ)} [M,N] \hotimes_{C^*(\Gamma, \Lambda)} \boldsymbol{\Pi} \\
&=[\tau_*\fv] \hotimes_{C_0(M^\circ)} [M,N]\\
&= \tau_*([\fv] \hotimes_{C_0(M^\circ)} [M,N])\\
&= \tau_*(\prod _n \langle [\fv_n] , [M,N] \rangle ).
\end{align*}
It is non-zero because $\ker \tau_*$ is identified with $\bigoplus \K_0(A_n)$ through the isomorphism $\K_0(B) \cong \prod \K_0(A_n)$.

The claim (2) is proved in the same way. We only remark that in this case $\Pi$ is a relative representation of $(\Gamma, \Lambda)$, which is actually a relative representation of $(\Gamma, \phi(\Lambda))$ by \cite[Remark 6.3]{Kubota3}.
\end{proof}

Together with Theorem \ref{thm:enlarge}, Theorem \ref{thm:Karea} implies the following relative version of the result of \cite{MR2259056,MR2353861}. 
\begin{cor}
Let $(M, g)$ be a Riemannian spin manifold with a collared boundary $N$. If $M_\infty$ is area-enlargeable, then $\mu_*^{\Gamma, \Lambda}([M,N])$ does not vanish.
\end{cor}

\subsection{The Hanke--Pape--Schick codimension 2 obstruction}\label{section:3.2}
The second application of Theorem \ref{prp:bundle} is concerned with the codimension $2$ obstruction of positive scalar curvature metric which is first introduced by Gromov--Lawson~\cite[Theorem 7.5]{MR720933} and generalized by Hanke--Pape--Schick~\cite[Theorem 4.3]{MR3449594}. Here we show the following theorem. 
\begin{thm}\label{thm:codim2}
Let $M$ be an $n$-dimensional closed spin manifold with an embedded codimension $2$ submanifold $N$ satisfying
\begin{itemize}
\item the induced map $\pi_1(N) \to \pi_1(M)$ is injective,
\item the induced map $\pi_2(N) \to \pi_2(M)$ is surjective, \
\item the normal bundle of $N$ is trivial.
\end{itemize}
Let $W \cong N \times \bD^2$ be a closed tubular neighborhood of $N$, let $M_0:=M \setminus W^\circ$, let $N_0:=\partial M_0$, let $\Gamma :=\pi_1(M)$ and let $\Lambda :=\pi_1(N)$. 
Then $\mu ^\Lambda _{n-2}([N]) \neq 0$ implies that $\mu ^{\Gamma, \Lambda}_n([M_0 , N_0]) \neq 0$.
\end{thm}
\begin{rmk}
It is proved in \cite[Corollary 5.3]{Kubota1} that, for a partitioned manifold $M=M_1 \sqcup_N M_2$ with $f_i \colon (M_i, N) \to (B\Gamma _i, B\Lambda)$ with $\Lambda \to \Gamma _i$ injective, the non-vanishing of $\mu ^{\Gamma, \Lambda}([M_i,N])$ implies that of $\mu ^{\Gamma_1*_\Lambda \Gamma_2} _* ([M])$.
We apply this theorem to $M=N \times \bD^2 \sqcup _{N \times S^1} M_0$, $\Gamma_1=\pi_1(N \times \bD_2)$, $\Lambda = \pi_1(N)$ and $\Gamma_2=\pi_1(M_0)$ in the setting of Theorem \ref{thm:codim2}. Then the conclusion of Theorem \ref{thm:codim2} implies the non-vanishing of $\mu^\Gamma([M])$. 
In particular, we obtain that $M$ does not admit any metric with positive scalar curvature, as is proved in \cite[Theorem 4.3]{MR3449594}.

As is remarked at the introduction of \cite{MR3449594}, the stable Gromov--Lawson--Rosenberg conjecture proved by Rosenberg--Stolz \cite{MR1321004} and \cite[Theorem 4.3]{MR3449594} also implies the non-vanishing of the higher index of $M$ if $\Gamma$ satisfies the Baum--Connes injectivity. Here we give a direct proof of this fact without the assumption of Baum--Connes injectivity. 
\end{rmk}

For the proof, we prepare general lemmas on the boundary map of $\K$-theory. 
\begin{lem}\label{lem:Klift}
Let $0 \to I \to D \to D/I \to 0$ be an exact sequence of C*-algebras. For a pair of projections $(q_1,q_2) \in \cM(D/I)^{\oplus 2}$ with $q_1-q_2\in D/I$, the element $\partial [q_1, q_2] \in \K_1(I)$ is represented by a unitary 
\[ \exp(-2\pi i \tilde{q}_1)\exp(2\pi i \tilde{q}_2) \in 1+I, \]
where $\tilde{q}_i \in \cM(D)$ is a self-adjoint lift of $q_i$ such that $\tilde{q}_1 - \tilde{q}_2 \in D$. 
\end{lem}
\begin{proof}
Let $\cI$ denotes the kernel of the homomorphism $\cM(D) \to \cM(D/I)$. It includes $I$ as an ideal and $\cI \cap D = I$ holds. Consider the diagram of exact sequences
\[ \xymatrix@R=1.5em@C=1.5em{
0 \ar[r] & I \ar[d]^{\iota} \ar[r] & D \ar[r] \ar[d] & D/I \ar[r] \ar[d] & 0 \\ 
0 \ar[r] & \cI \oplus_{\cI/I} \cI \ar[r] & \cM(D) \oplus_{\cQ(D)} \cM(D) \ar[r] & \cM(D/I) \oplus_{\cQ(D/I)} \cM(D/I) \ar[r] & 0.
}
\]
The vertical morphisms are inclusions into the first component, which induce isomorphisms of $\K$-theory by the five lemma. Now $(q_1, q_2) \in \cM(D/I) \oplus _{\cQ(D/I)} \cM(D/I)$ is lifted to a self-adjoint element $(\tilde{q}_1,\tilde{q}_2) \in \cM(D) \oplus _{\cQ(D)} \cM(D)$ and hence 
\begin{align*}
\partial [(q_1,q_2)] = [(e^{-2\pi i \tilde{q}_1}, e^{-2\pi i \tilde{q}_2})] = \iota_*[e^{-2\pi i \tilde{q}_1}e^{2\pi i \tilde{q}_2}] \in \K_1(\cI \oplus_{\cI/I} \cI),
\end{align*}
which shows the lemma by commutativity of the boundary map. 
The last equality holds because 
\[ (e^{-2\pi i \tilde{q}_1}, e^{-2\pi i \tilde{q}_2}) = (e^{-2\pi i \tilde{q}_1}e^{2\pi i \tilde{q}_2} , 1)\cdot (e^{-2\pi i \tilde{q}_2}, e^{-2\pi i \tilde{q}_2}), \]
and $[(e^{-2\pi i \tilde{q}_2}, e^{-2\pi i \tilde{q}_2})]$ is in the image of the diagonal inclusion $\cI \to \cI \oplus _{\cI/I} \cI$ (we remark that $\K_1(\cI)=0$).
\end{proof}

Let $A$ be a C*-algebra and let $B:= \bB(\sH_A)$ and $J:= \bK(\sH_A)$. 
Let $\cZ_1$ and $\cZ_2$ be bundles of infinitely generated projective Hilbert $A$-modules with the typical fiber $Z_1$ and $Z_2$ respectively. Then 
\[ \bar{\cZ}_i :=\bB(\cZ_i,\sH_A)/\bK(\cZ_i,\sH_A)\]
(where $\bB(\cZ_i,\sH_A)$ denotes the set of adjointable bounded operators on from $\sH_A$ to $\cZ_i$) is a Hilbert $B/J$-module bundle with $\bB(\bar{\cZ}_i) \cong \cQ(Z_i)$ (the $B/J$-action from the right, the $\cQ(Z_i)$-action from the left and the inner product are induced from the product of operators). 
Suppose that there is a bundle homomorphism $U \colon \cZ_1|_{N_0} \to \cZ _2|_{N_0}$ such that $U^*U-1 \in \bK(C(N_0, \cZ_1))$ and $UU^*-1 \in \bK(C(N_0, \cZ_2))$. Then it induces a unitary operator $\bar{U} \colon \bar{\cZ}_1 \to \bar{\cZ}_2$.

We write $[\partial_{B/J}] \in \KK_1(B/J,J)$ and $[\partial_{C(N_0)}] \in \KK_1(C(N_0), C_0(M_0^\circ ))$ for the $\KK$-classes corresponding to the extensions $0 \to J \to B \to B/J \to 0$ and $0 \to C_0(M_0^\circ ) \to C(M_0) \to C(N_0) \to 0$ respectively. 
\begin{lem}\label{lem:pairext}
Let $\cZ_i$, $U$, $\bar{\cZ}_i$ and $\bar{U}$ as above. Then we have
\begin{align}
[\bar{\cZ}_1 , \bar{\cZ}_2 , \bar{U}] \hotimes _{B/J} [\partial_{B/J}] =- [\cZ_1|_{N_0}, \cZ_2|_{N_0}, U] \hotimes _{C(N_0)} [\partial_{C(N_0)}] \label{form:pairext}
\end{align}
under the isomorphism $\KK(\bC, C_0(M_0^\circ) \otimes J) \cong \KK(\bC, C_0(M_0^\circ) \otimes A)$ given by the Kasparov product with the imprimitivity bimodule $[\sH_A] \in \KK(J,A)$.
\end{lem}
\begin{proof}
First, notice that there are isometries $V_i \colon \cZ_i \to \underline{\sH_A}$ such that $V_2^*V_1 - U \in \bK(\cZ_1, \cZ_2)$. Indeed, let $S$ denote a unitary lift of $\big( \begin{smallmatrix}0 & \bar{U}^* \\ \bar{U} & 0 \end{smallmatrix} \big)$ and let $W \colon \cZ_1 \oplus \cZ_2 \to \underline{\sH_A}$ be an isometry (which exists by the Kasparov stabilization theorem \cite[Theorem 2]{MR587371}). Then $V_1:=WV_1'$ and $V_2:=WSV_2'$, where $V_i' \colon \cZ_i \to \cZ_1 \oplus \cZ_2$ is the embedding to the $i$-th direct summand, is desired isometries. Moreover, by a pull-back with respect to a deformation retract of $N_0$, we may assume that $P_1=P_2$ on a neighborhood $O$ of $N_0$. Let $\psi$ be a continuous function supported on $O$ such that $0 \leq \psi \leq 1$ and $\psi|_{N_0} \equiv 1$ and let $P':=\psi P_1+(1-\psi ) P_2$.

Now we apply Lemma \ref{lem:Klift} to determine the left and right hand side of (\ref{form:pairext}). Since $(P_1,P')$ is a self-adjoint lift of $(q(P_1), q(P_2)) \in \cM(C_0(M_0^\circ) \otimes B/J)^{\oplus 2}$ to $\cM(C_0(M_0^\circ) \otimes B)^{\oplus 2}$ such that $P_1-P' \in C_0(M_0^\circ) \otimes B$, we get
\begin{align*}   
&[\bar{\cZ}_1, \bar{\cZ}_2, \bar{U}] \otimes_{B/J} [\partial _{B/J}] \\
=& \partial[q(P_1), q(P_2)] =[\exp(-2\pi i P_1)\exp(2\pi i P')]=[\exp(2\pi i P')]. 
\end{align*}
Similarly, since $(P',P_2)$ is a self-adjoint lift of $(P_1|_{N_0}, P_2|_{N_0})  \in \cM(C(N_0) \otimes J)^{\oplus 2}$ to $\cM(C(M_0) \otimes J)^{\oplus 2}$ such that $P'-P_2 \in C(M_0) \otimes J$, we get
\begin{align*} 
&[\cZ_1|_{N_0}, \cZ_2|_{N_0}, U] \otimes_{C(N_0)} [\partial _{C(N_0)}]\\
 =& \partial[P_1|_{N_0}, P_2|_{N_0}] =[\exp(-2\pi i P')\exp(2\pi i P_2)]=[\exp(-2\pi i P')]. 
\end{align*}
This completes the proof of the lemma. 
\end{proof}

We fix a base point $x_0 \in N_0$ in order to consider the Galois correspondence of covering spaces. Let $\tilde{M}$ denote the universal covering of $M$. Let $\bar{M}:=\tilde{M}/\Lambda = \tilde{M} \times _\Gamma \Gamma / \Lambda$ and $\bar{\pi} \colon \bar{M} \to M$, $\tilde{\pi} \colon \tilde{M} \to \bar{M}$ denote the projections. Then $\bar{\pi}^{-1}(W)$ is the disjoint union of coverings of $W$ indexed by $g\Lambda \in \Gamma / \Lambda$, each of which has the fundamental group $\Lambda  \cap g\Lambda g^{-1}$. In particular, the connected component $\bar{W}$ including the base point $x_0$ is diffeomorphic to $W$ by $\bar{\pi}$. Let $\bar{N}_0 : = \partial \bar{W} $.   

An essential ingredient of the codimension two obstruction theorem, which is given in the proof of \cite[Theorem 4.3]{MR3449594}, is the existence of a nice $\Lambda \times \bZ$-Galois covering on $\bar{M} \setminus \bar{W}^\circ$. Here we restate it for our convenience.
\begin{lem}\label{lem:codim2}
There is a $\bZ$-Galois covering $\breve{M}_0$ over $\tilde{M}_0:=(\tilde{\pi} \circ \bar{\pi})^{-1}(M_0)$ with the following properties:
\begin{itemize}
\item Its restriction to $\tilde{\pi}^{-1}(\bar{N}_0) \cong \tilde{N} \times S^1$ is the universal covering.
\item Its restriction to $\tilde{\pi}^{-1}(\bar{\pi}^{-1}(N_0) \setminus \bar{N}_0)$ is trivial. 
\end{itemize}
\end{lem}
\begin{proof}
We write $\gamma$ for the closed loop $\{ x_0 \} \times S^1 \subset N \times S^1 \cong N_0$. Then $\gamma$ generates the second component of $\pi_1(N_0) \cong \Lambda \times \bZ [\gamma]$.
Let $i \colon \bar{N} \to \bar{M}_0$ and $j \colon \bar{M}_0 \to \bar{M}$ denote the inclusions. 
It is proved in \cite[Theorem 4.3]{MR3449594} that there is a splitting
\[r \colon \pi_1(\bar{M} \setminus \bar{W}^\circ) \to \Lambda \times \bZ \]
of $i_*$, that is, $r \circ i_*=\id_{\Lambda \times \bZ}$. 

Then the homomorphism $\mathrm{pr}_\Lambda \circ r$ (where $\mathrm{pr}_\Lambda \colon \Lambda \times \bZ \to \Lambda$ is the projection) is equal to $j_*$. Indeed, both $\mathrm{pr}_\Lambda \circ r$ and $j_*$ map $[\gamma]$ to the trivial element and the induced homomorphisms from $\pi_1(\bar{M} \setminus \bar{W}^\circ ) / \langle [\gamma] \rangle $ to $\Lambda$ are the inverse of the composition
\[  \Lambda \hookrightarrow \Lambda \times \bZ \xrightarrow{i_*}\pi_1(\bar{M} \setminus \bar{W}^\circ ) \to \pi_1(\bar{M} \setminus \bar{W} ^\circ )/\langle [\gamma] \rangle. \]
Therefore the covering $\breve{M}_0$ of $\bar{M}_0$ associated to $r$ satisfies $\breve{M}_0/\bZ = \breve{M}_0 \times _{\Lambda \times \bZ} \Lambda \cong \tilde{M}_0$. That is, $\breve{M}_0$ is a $\bZ$-Galois covering on $\tilde{\pi}^{-1}(\bar{M}_0)$.

The equality $r \circ i_* = \id_{\Lambda \times \bZ}$ means that the restriction of $\breve{M}_0$ to $\bar{N}_0$ is the universal covering $\tilde{N} \times \bR$ of $N \times S^1$. 
That is, the restriction of the $\bZ$-Galois covering $\breve{M}_0$ to $\tilde{\pi}^{-1}(\bar{N}) \cong \tilde{N} \times S^1$ is the universal covering. 
At the same time, the restriction of the $\bZ$-Galois covering $\breve{M}_0$ to each connected component of $\tilde{\pi}^{-1}(\bar{\pi}^{-1}(N) \setminus \bar{N})$ is trivial because it is extended to a connected component of $(\tilde{\pi} \circ \bar{\pi})^{-1}(W)$, which is simply connected.
\end{proof}

\begin{lem}\label{lem:infgen}
Under the assumption of Theorem \ref{thm:codim2}, $\bar{M}$ is an infinite covering, that is, $\Gamma/\Lambda$ is an infinite set.
\end{lem}
\begin{proof}
Assume that $\bar{M}$ is a finite covering of $M$, and hence a closed manifold. 
The $\Lambda \times \bZ$-Galois covering $\breve{M}_0 \to \bar{\pi}^{-1}(M_0)$ constructed in Lemma \ref{lem:codim2} extends to a $\Lambda \times \bZ$-Galois covering on a spin manifold $\bar{M} \setminus \bar{W}^\circ$. Since its restriction to the boundary $\bar{N}_0 \cong N_0$ is isomorphic to the universal covering of $N_0$, we obtain that $[N_0,f]=0 \in \Omega^{\mathit{spin}}_{n-1}(B(\Lambda \times \bZ))$ (where $f$ is the reference map associated to the universal covering). This contradicts to the assumption $\mu^\Lambda_{n-2}([N]) \neq 0$ (which implies $\mu^{\Lambda \times \bZ}_{n-1}([N_0]) \neq 0$).
\end{proof}

\begin{proof}[Proof of Theorem \ref{thm:codim2}]
Let $A:=C^*(\Lambda \times \bZ)$. We consider two bundles
\begin{itemize}
\item $\cV_1:= \breve{M}_0 \times _{\Lambda \times \bZ} C^*(\Lambda \times \bZ)$,
\item $\cV_2:= \tilde{M}_0 \times _{\Lambda} C^*(\Lambda \times \bZ)$ (here $\Lambda$ acts on $C^*(\Lambda \times \bZ)$ from the left through the inclusion $\Lambda \to \Lambda \times \bZ$)
\end{itemize}
of Hilbert $A$-modules over $\bar{M}_0$, where $\breve{M}_0$ is as in Lemma \ref{lem:codim2}. We associate to them bundles 
\[ \cZ_i:=\bar{\pi}_!\cV_i = \bigsqcup _{x \in M_0} \bigoplus_{\bar{\pi}(\bar{x})=x} (\cV_i)_{\bar{x}} \]
of infinitely generated (by Lemma \ref{lem:infgen}) Hilbert $A$-module bundles on $M_0$, which are equipped with the canonical flat structures. Let $Z_i:=\bigoplus_{\bar{\pi}(\bar{x})=x_0} (\cV_i)_{\bar{x}}$ is the fiber of $\cZ_i$ on $x_0$ and let $\sigma _i \colon \bar{\Gamma} \to \mathrm{U} (Z_i)$ denotes the associated monodromy representation. Note that $\sigma_2$ factors through $\Gamma$.

By the construction of $\breve{M}_0$ in Lemma \ref{lem:codim2}, we have an isomorphism of flat $A$-module bundles between the restriction of $\cV_1$ and $\cV_2$ on $\bar{\pi}^{-1}(N_0) \setminus \bar{N}_0$. It induces a partial isometry $U \colon \cZ_1|_{N_0} \to \cZ_2|_{N_0}$ such that $\ker U = \cV_1|_{\bar{N}_0} \subset \cZ_1$, $\ker U^* = \cV_2|_{\bar{N}_0} \subset \cZ_2$ and
\[ \sigma_2(g)U_{x_0} = U_{x_0}\sigma_1(g) \]
for any $g \in \Lambda \times \bZ$, where $U_{x_0}$ is a restriction of $U$ to $\bar{\pi}^{-1}(x_0)$.

As in Lemma \ref{lem:pairext}, let $\bar{\cZ}_i$ denote the bundle $\bB(\cZ_i,\sH_A)/\bK(\cZ_i,\sH_A)$ of Hilbert $B/J$-modules and let $\bar{Z}_i:=(\bar{\cZ}_i)_{x_0}=\bB(Z_i,\sH_A) / \bK(Z_i,\sH_A)$ for $i=1,2$. Then $\sigma_i$ and $U$ above induces $\bar{\sigma} _i \colon \bar{\Gamma} \to \mathrm{U}(\cQ(Z_i)) \cong \mathrm{U}(\bar{Z}_i)$ and $\bar{U} \colon \bar{Z}_1 \to \bar{Z}_2$ respectively. Then $\bar{U}$ is a unitary and $\bar{U}_{x_0}\bar{\sigma}_1(g)\bar{U}_{x_0}^* =\bar{\sigma}_2(g) $ holds for any $g \in \Lambda \times \bZ$. This particularly implies that $\bar{\sigma}_1(\gamma) = 1$ (where $\gamma$ is the generator of $\bZ \subset \Lambda \times \bZ$), that is, $\bar{\sigma}_1 \colon \bar{\Gamma} \to \mathrm{U}(\bar{Z}_1)$ factors through $\Gamma$. 

Consequently, we obtain that the triplet $\Pi:=(\bar{\sigma}_1, \bar{\sigma}_2, \bar{U})$ is a relative representation of $(\Gamma, \Lambda)$ and its associated relative $B/J$-module bundle (in the sense of Theorem \ref{prp:bundle}) is $(\bar{\cZ}_1, \bar{\cZ}_2, \bar{U})$. 
Now we apply Theorem \ref{prp:bundle} and Lemma \ref{lem:pairext} to get
\begin{align*}
&((\ell_{\Gamma, \Lambda} \otimes_{C_0(M_0^\circ)} [M_0, N_0]) \otimes_{C^*(\Gamma, \Lambda)} \boldsymbol{\Pi}) \otimes_{B/J} [\partial _{B/J}]\\ 
=& (\ell_{\Gamma, \Lambda} \otimes _{C^*(\Gamma, \Lambda)} \boldsymbol{\Pi}) \otimes_{B/J} [\partial _{B/J}] \otimes_{C_0(M_0^\circ)} [M_0, N_0]\\
=&[\bar{\cZ}_1, \bar{\cZ}_2, \bar{U}] \otimes_{B/J} [\partial_{B/J}] \otimes _{C_0(M_0^\circ)} [M_0, N_0]\\
=&-([\cZ_1|_{N_0}, \cZ_2|_{N_0}, U] \otimes_{C(N_0)} [\partial_{C(N_0)}]) \otimes_{C_0(M^\circ)}[M_0, N_0]\\
=&(-[\cV_1|_{\bar{N}_0}] + [\cV_2|_{\bar{N}_0}]) \otimes_{C(N_0)} [N_0] \\
=& - \mu_{n-1}^{\Lambda \times \bZ}([N \times S^1]) + \mu^{\Lambda}_{n-1}([N \times S^1])\\
=&-\mu^\Lambda_{n-2} ([N]) +0  \neq 0. 
\end{align*}
The last equality is considered under the identification of $\K_{n-2}(C^*(\Lambda))$ with the second direct summand of 
\[ \K_{n-1}(C^*(\Lambda \times \bZ)) = \K_{n-1}(C^*\Lambda \otimes C^*(\bZ)) \cong \K_{n-1}(C^*(\Lambda)) \oplus \K_{n-1}(C^*(\Lambda) \otimes S^{0,1}).\]
For the forth equality, we use the boundary of Dirac is Dirac principle $[\partial _{C(N_0)}] \otimes_{C_0(M_0^\circ)}  [M_0,N_0]=[N_0]$ (for the proof, see for example \cite[Proposition 11.2.15]{MR1817560}). 
\end{proof}

\section{Relative quantitative index pairing} \label{section:5}
In this section, we reformulate the index theorem for the image of the higher index under a quasi-representation developed by Dadarlat~\cite{MR2982445} and generalize it to the relative setting. Instead of Lafforgue's Banach KK-theory, on which the formulation of \cite{MR2982445} is based, we use the quantitative K-theory introduced by Oyono-Oyono and Yu \cite{MR3449163}. 

\subsection{Quantitative $\K$-theory and almost $\ast$-homomorphism}
We start with a quick review of the quantitative K-theory. The basic reference is \cite{MR3449163}. We say that a filtered C*-algebra is a C*-algebra $A$ equipped with an increasing family $\{ A_r \}_{r \in [0,\infty)}$ of closed subspaces of $A$ such that $A_r^*=A_r$, $A_r \cdot A_{r'} \subset A_{r+r'}$ and $\bigcup_r A_r \subset A$ is dense. 

For a unital filtered C*-algebra $A$, $0 \leq \varepsilon \leq \frac{1}{4}$ and $r>0$, let
\begin{align*}
\mathrm{P}_n^{\varepsilon , r}(A)&:= \Big\{ p \in  \bM_n(A_r) \mid p=p^*, \ \| p^2 -p \| <\varepsilon \Big\}_{\textstyle ,} \\
\mathrm{U}_n^{\varepsilon , r}(A)&:= \Big\{ u \in  \bM_n(A_r) \mid \| u^*u-1 \|<\varepsilon, \  \| uu^* -1 \| <\varepsilon \Big\}_{\textstyle ,}
\end{align*}
and $\mathrm{P}_\infty^{\varepsilon , r}(A):= \bigcup _{n \in \bN} \mathrm{P}_n^{\varepsilon , r}(A)$, $\mathrm{U}_\infty^{\varepsilon , r}(A) := \bigcup _{n \in \bN} \mathrm{U}_n^{\varepsilon , r}(A)$.
For $k \in \bN$, let $1_k$ denote the unit of $\bM_k \subset A^+ \otimes \bM_k$.
We introduce the equivalence relation to $\mathrm{P}_\infty^{\varepsilon, r}(A) \times \bN$ and $\mathrm{U}_\infty^{\varepsilon , r}(A)$ as 
\begin{itemize}
\item $(p,k) \sim (q, l)$ if $\diag (p ,  1_l)$ and $\diag (q, 1_k)$ are connected by a continuous path in $\mathrm{P}_\infty^{\varepsilon ,r} (A)$, 
\item $u \sim v$ if $u$ and $v$ are connected by a continuous path in $\mathrm{U}_\infty^{3\varepsilon , 2r}(A)$.
\end{itemize}
The quantitative $\K$-groups are defined by
\begin{align*}
\K_0^{\varepsilon , r}(A)&:= \mathrm{P}_\infty^{\varepsilon , r}(A) \times \bN /\sim , \\
\K_1^{\varepsilon , r}(A)&:=  \mathrm{U}_\infty^{\varepsilon , r} (A) /\sim . 
\end{align*}
We write the elements of quantitative $\K_*$-groups represented by $(p,l) \in \mathrm{P}_\infty ^{\varepsilon, r}(A)$ and $u \in \mathrm{U}_\infty ^{\varepsilon, r}(A)$ as $[p,l]_{\varepsilon, r}$ and $[u]_{\varepsilon, r}$ respectively. 
The summations $[p,k]_{\varepsilon, r} + [q,l]_{\varepsilon, r} := [\diag (p,q) , k+l]_{\varepsilon, r}$ and $[u]_{\varepsilon, r} + [v]_{\varepsilon, r} = [\diag (u,v)]_{\varepsilon, r}$ make $\K_0^{\varepsilon , r}(A)$ and $\K_1^{\varepsilon , r}(A)$ into abelian groups (for the proof, see Lemma 1.14, Lemma 1.15 and Lemma 1.16 of \cite{MR3449163}). 

For a non-unital filtered C*-algebra $A$, the unitization $A^+$ is also equipped with the structure of filtered C*-algebra by $A_r^+:= A_r+ \bC 1$. Let $\rho \colon A^+ \to \bC$ denote the quotient. The quantitative $\K$-group is defined by
\[ \K_0^{\varepsilon ,r}(A):= \ker (\rho_* \colon \K_0^{\varepsilon ,r}(A^+) \to \K_0^{\varepsilon , r}(\bC) \cong \bZ ) \]
and $\K_1^{\varepsilon ,r}(A):= \K_1^{\varepsilon , r}(A^+)$. For any $(\varepsilon , r)$, we write $\iota_A$ for the canonical homomorphism from $\K_*^{\varepsilon , r}(A)$ to $\K_*(A)$.

\begin{rmk}
Hereafter we often use the norm estimates $\| p \| \leq 1+\varepsilon $ for $p \in \mathrm{P}_\infty ^{\varepsilon, r}(A)$ and $\| u \| \leq 1 + \varepsilon /2 $ for $\mathrm{U}_\infty ^{\varepsilon ,r}(A)$.
\end{rmk}

Next, we introduce the notion of complete almost $\ast$-homomorphism between filtered C*-algebras. 
\begin{defn}
Let $A$ and $D$ be filtered C*-algebras. A bounded linear map $\pi \colon A_r \to D_{\kappa r}$ is a complete $(\varepsilon , r , \kappa )$-$\ast$-homomorphism if $\pi(a^*)=\pi(a)^*$ for any $a \in A_r$ and 
\[ \| \pi _n (ab) - \pi_n(a)\pi_n(b)\| \leq \varepsilon  \| a\| \| b \| \]
holds for any $a , b \in A_r \otimes \bM_n$, where $\pi_n:= \pi \otimes \id_{\bM_n}$.
\end{defn}

\begin{rmk}
Let $\pi \colon A_r \to D_{\kappa r}$ be a completely $(\varepsilon , r , \kappa )$-$\ast$-homomorphism. For $a \in A_r \otimes \bM_n$ with $\| a\| =1$ and $\| \pi_n(a) \| > \| \pi _n \| -\varepsilon'$, we have
\[(\| \pi_n \| -\varepsilon')^2 < \| \pi_n(a)^*\pi_n(a) \| \leq \| \pi_n (a^*a) - \pi_n(a)^*\pi_n(a) \| + \| \pi_n(a^*a) \| \leq \varepsilon + \| \pi_n \| . \]
This means that $\| \pi_n \|^2 < \| \pi_n \| + \varepsilon $ and hence $\| \pi_n \| < 1+\varepsilon /2$. That is, $\pi$ is a completely bounded map between operator spaces (a reference on completely bounded maps and operator spaces is \cite[Appendix B]{MR2391387}). In particular, $\pi \otimes \id_B \colon A_r \otimes B \to D_{\kappa r} \otimes B$ is a well-defined completely bounded map for any nuclear C*-algebra $B$ (\cite[Corollary B.8]{MR2391387}).
\end{rmk}

A C*-algebra is said to be \emph{quasi-diagonal} if it admits a faithful representation $\pi \colon A \to \bB(\sH)$ with an increasing sequence $p_n$ of finite rank projections such that $[\pi(a) , p] \to 0$ for any $a \in A$ (for more details, see for example \cite[Section 7]{MR2391387}).
\begin{lem}\label{lem:QDqRep}
Let $B$ be a nuclear quasi-diagonal C*-algebra. Then $\pi \otimes \id_B \colon A_r \otimes B \to D_{\kappa r} \otimes B$ is a complete $(\varepsilon , r , \kappa )$-$\ast$-homomorphism.
\end{lem}
\begin{proof}
First, $\pi \otimes \id_{\prod_{n \in \bN} \bM_n}$ is a complete $(\varepsilon , r , \kappa )$-$\ast$-homomorphism since $A \otimes (\prod_{n} \bM_n)$ is canonically isomorphic to $\prod_n (A \otimes \bM_n)$. Since there is an isomorphism
\[ \Big( \prod_{n\in \bN} \bM_n \Big) / \Big( \bigoplus_{n \in \bN} \bM_n \Big) \cong \varinjlim _{N \to \infty} \Big( \prod_{n \geq N} \bM_n \Big),\] 
we obtain that $\pi \otimes \id_{\prod \bM_n / \bigoplus \bM_n}$ is also a complete $(\varepsilon , r , \kappa )$-$\ast$-homomorphism.

Recall that a nuclear C*-algebra $B$ is quasi-diagonal if and only if there is a faithful $\ast$-homomorphism $\varphi \colon B \to \prod_{n \in \bN} \bM_n / \bigoplus _{n \in \bN} \bM_n$. 
Since the diagram
\[ \xymatrix@C=5em{
A \otimes B \ar[r]^{\pi \otimes \id_B} \ar[d]^{\id_A \otimes \varphi} & D \otimes B \ar[d]^{\id_D \otimes \varphi} \\
A \otimes \frac{\prod \bM_n}{\bigoplus \bM_n} \ar[r]^{\pi \otimes \id_{\prod \bM_n/\bigoplus \bM_n}} & D \otimes \frac{\prod \bM_n}{ \bigoplus \bM_n}
}
\]
commutes, $\pi \otimes \id_B$ is also a $(\varepsilon , r , \kappa)$-$\ast$-homomorphism.
\end{proof}

\begin{prp}\label{prp:qrepK}
Let $A$, $B$ be two unital filtered C*-algebras and let $\pi \colon A_r \to B_{\kappa r}$ be a unital complete $(\varepsilon , r , \kappa)$-$\ast$-homomorphism. Then, for any $\delta \geq 0$ such that $\varepsilon+(1+3\varepsilon)\delta < \frac{1}{4}$, it gives rise to continuous maps
\begin{align*}
\pi &\colon \mathrm{P}_n^{\delta , r}( A ) \to \mathrm{P}_n ^{\varepsilon + (1+3\varepsilon)\delta  ,\kappa r}(B),\\
\pi & \colon \mathrm{U}_n ^{\delta , r}(A ) \to \mathrm{U}_n ^{\varepsilon + (1+3\varepsilon)\delta  ,\kappa r}(B),
\end{align*}
and hence induces homomorphisms 
\[ \pi_\sharp \colon \K_*^{\delta, r}(A) \to \K_*^{\varepsilon+ (1+2\varepsilon)\delta  , \kappa r}(B).\]
\end{prp}
\begin{proof}
Let $p \in \mathrm{P}_n^{\delta , r}(A)$ and $u \in \mathrm{U}_n^{\delta, r}(A)$. Then we have 
\begin{align*}
\| \pi_n(p)^2-\pi_n(p) \| & \leq \| \pi_n(p)^2-\pi_n(p^2) \| + \| \pi_n(p^2-p) \| \\ 
& \leq \varepsilon \| p\|^2 +  \| \pi\|_{\mathrm{cb}} \| p^2-p\| \\
&\leq \varepsilon (1+\delta) + (1+\varepsilon /2) \delta \\
&\leq \varepsilon + (1+2\varepsilon)\delta ,\\
\| \pi_n(u)^*\pi_n(u) -1 \| & \leq \| \pi_n(u)^*\pi_n(u) -\pi_n(u^*u) \| + \| \pi_n(u^*u -1) \| \\
& \leq \varepsilon \| u^* \| \| u \| + \| \pi \|_{\mathrm{cb}} \delta \\
& \leq \varepsilon (1+\delta)^2 + (1+\varepsilon /2) \delta \\
& \leq \varepsilon + (1+3\varepsilon)\delta .
\end{align*}
Similarly we also have $\| \pi_n(u)\pi_n(u)^* -1 \| \leq \varepsilon + (1+3\varepsilon)\delta$.
\end{proof}
\begin{rmk}
For possibly non-unital filtered C*-algebras $A$, $B$ and a $(\varepsilon ,r , \kappa)$-$\ast$-homomorphism $\pi$, it is straightforward to see that the unitization $\ast$-homomorphism $\pi^+ \colon A^+ \to B^+$ defined by $\pi^+|_A=\pi$ and $\pi^+(1_A)=1_B$ is also a complete $(\varepsilon , r, \kappa)$-$\ast$-homomorphism. Therefore $\pi$ induces a homomorphism of quantitative $\K$-groups by Proposition \ref{prp:qrepK}.
\end{rmk}

\subsection{Quantitative index pairing}
Let $\Gamma$ be a finitely generated discrete group and let $e \in \cG_\Gamma \subset \Gamma$ be a finite set generating $\Gamma $. We assume that $\gamma^{-1} \in \cG_\Gamma $ for $\gamma \in \cG_\Gamma $. Let $l _\Gamma$ denote the word length function on $\Gamma$ with respect to $\cG_\Gamma $. Since $l_\Gamma $ satisfies $l_\Gamma (\gamma \cdot \gamma ') \leq l_\Gamma  (\gamma )+ l_\Gamma (\gamma ')$, it gives the structure of a filtered C*-algebra on the group C*-algebra $C^*\Gamma$, that is, 
\[C^*(\Gamma)_r :=\bigg\{ \sum _{\ell (\gamma) \leq r } c_\gamma u_\gamma \in \bC[\Gamma] \bigg\} \subset C^*(\Gamma) \]
forms an increasing sequence of closed subspaces of $C^*(\Gamma)$ such that  $C^*(\Gamma)_r \cdot C^*(\Gamma)_{r'} \subset C^*(\Gamma)_{r+r'}$ and $\bigcup C^*(\Gamma )_r = \bC[\Gamma]$ is dense in $C^*(\Gamma)$. 
For $r \in \bZ_{>0}$, we write $\cG_\Gamma^r$ for the set $\{ \gamma _1 \cdots \gamma_r \mid \gamma _i \in \cG_\Gamma \}$.

For a $(\varepsilon , \cG_\Gamma^r)$-representation $\pi$ of $\Gamma$ on $P$, we use the same letter $\pi$ for the linear map $C^*(\Gamma)_r \to B:=\bB(P)$ given by $\pi(\sum c_\gamma u_{\gamma}):= \sum c_\gamma \pi(\gamma )$. We say that $\pi$ is \emph{self-adjoint} if $\pi(\gamma^{-1})=\pi(\gamma)^*$ holds. Note that for any $(\varepsilon , \cG_\Gamma^r)$-representation $\pi$ there is a self-adjoint $(70\varepsilon , \cG_\Gamma^r)$-representation $\breve{\pi}$ with $d(\pi, \breve{\pi}) < 20\varepsilon$ (\cite[Proposition 5.6]{mathOA150306170}).
\begin{prp}\label{prp:almhom}
Let $\pi$ be a self-adjoint $(\varepsilon , \cG_\Gamma^r)$-representation of $\Gamma$ on $P$.  
 
Then $\pi $ is a unital complete $(|\cG_\Gamma^r|^2 \varepsilon , r, 1) $-$\ast$-homomorphism.
\end{prp}
\begin{proof}
Let $x = \sum_{\gamma \in \cG_\Gamma ^r } a_\gamma u_\gamma$ and $y=\sum_{\gamma \in \cG_\Gamma ^r } b_\gamma u_\gamma$ be elements in $C^*(\Gamma)_r  \otimes \bM_n$, where $a_\gamma$ and $b_\gamma$ are elements of $ \bM_n$. We remark that $\| a_\gamma \| \leq \| x\|$ and $\| b_\gamma \| \leq \| y \|$ for any $\gamma \in \Gamma$. Indeed, let $\tau \colon  C^*\Gamma \to \bC$ denote the tracial state given by $\tau (\sum c_\gamma u_\gamma):=c_e$. Then we have
\[ \| a_\gamma \| = \| (\id _{\bM_n} \otimes \tau )(xu_{\gamma ^{-1}})\| \leq \|x u_{\gamma^{-1}} \| = \| x\|. \]

Now we obtain that
\begin{align*}
\| \pi _n (x)\pi_n(y)-\pi _n  (xy) \| &= \Big\| \sum_{\gamma, \gamma' \in \cG_\Gamma ^r } a_\gamma b_{\gamma '} (\pi(\gamma)\pi(\gamma')-\pi(\gamma\gamma')) \Big\|   \\
& \leq \sum _{\gamma , \gamma' \in \cG_\Gamma ^r } \| a_\gamma \|  \cdot \| b_{\gamma'} \|  \cdot \|\pi(\gamma)\pi(\gamma')-\pi(\gamma\gamma')\|  \\
& \leq \Big( \sum_{\gamma\in \cG_\Gamma ^r } \| a_\gamma \|  \Big) \Big( \sum_{\gamma' \in \cG_\Gamma ^r } \| b_{\gamma'}\| \Big) \varepsilon \\
& \leq |\cG_\Gamma^r| ^2  \|x\| \| y\| \varepsilon. \qedhere
\end{align*}
\end{proof}

Let $X$ be a finite CW-complex and let $\Gamma :=\pi_1(X)$ (note that $\Gamma$ is finitely presented). Let $\cU := \{ U_{\mu}\} _{\mu \in I}$ be a good cover of $X$ and let $\{ \gamma_{\mu \nu}\}_{\mu,\nu \in I}$ be a flat transition function of the universal covering $\tilde{X} \to X$. Let $\cG_\Gamma:=\{ \gamma_{\mu \nu} \}_{\mu, \nu \in I}$. 
Let $\bv= \{ v_{\mu \nu} \} $ be a $\mathrm{U}(P)$-valued \v{C}ech $1$-cocycle. As are mentioned in (\ref{form:projMF}) and Remark \ref{rmk:MF}, the projections 
\begin{align*}
 P_\cV &:= \sum_{\mu,\nu \in I} \eta_\mu \eta_\nu \otimes u_{\gamma _{\mu\nu}} \otimes e_{\mu\nu} \in C(X) \otimes (C^*\Gamma)_1 \otimes \bM_I , \\ 
 p_\bv &:= \sum_{\mu,\nu \in I} \eta_\mu \eta_\nu v_{\mu \nu} \otimes e_{\mu\nu} \in C(X) \otimes B \otimes \bM_I ,
\end{align*}
have the support isomorphic to $\cV$ and $E_\bv$ respectively.
\begin{prp}\label{lem:quantBC}
There is a group homomorphism
\[ \alpha_\Gamma^{\mathrm{alg}} \colon \K_0(X) \to  \K^{0,3}_0(\bK(\sH) \otimes C^*(\Gamma))\]
such that $\iota_{C^*(\Gamma)} (\alpha_\Gamma^{\mathrm{alg}}(\xi))=\alpha_\Gamma (\xi)\in \K_0(C^*(\Gamma))$ for any $\xi \in \K_0(X)$.
\end{prp}
\begin{proof}
Let $(\varphi _1 , \varphi_2 ) \colon C(X) \to \bB(\sH) \triangleright \bK(\sH)$ be a quasi-homomorphism representing $\xi \in \KK(C(X) , \bC)$. Let $P_1:= \varphi_1(P_\cV)$ and $P_2:=\varphi_2(P_\cV)$. Set
\[ V:=\begin{pmatrix}P_2 & 1_I-P_2 \\ 1_I-P_2 & P_2 \end{pmatrix} \in \bM_{2}(\bB(\sH) \otimes C^*(\Gamma)_1 \otimes \bM_I). \]
Then $V$ is a self-adjoint unitary and $V \diag (P_2,1_I-P_2) V = \diag (1_I,0)$ holds. This implies that
\[ V \begin{pmatrix}P_1 & 0 \\ 0 & 1_I-P_2 \end{pmatrix}V  - \begin{pmatrix}1_I & 0 \\ 0 & 0 \end{pmatrix}  \in  \bM_{2}(\bK(\sH) \otimes C^*(\Gamma)_3 \otimes \bM_I),  \]
that is, the pair $(V\diag (P_1, 1_I -P_2)V , \diag (1_I , 0))$ determines a difference class $[V\diag (P_1, 1_I-P_2)v , \diag (1_I , 0)] \in \K_0(\bK(\sH) \otimes C^*(\Gamma))$. Moreover we have the equality of difference classes as
\begin{align*}
[P_1,P_2] &= \bigg[\begin{pmatrix}P_1 & 0 \\ 0 & 1_I -P_2 \end{pmatrix} _{\textstyle ,} \begin{pmatrix}P_2 & 0 \\ 0 & 1_I-P_2 \end{pmatrix}\bigg] \\
& = \bigg[{V\begin{pmatrix}P_1 & 0 \\ 0 & 1_I-P_2 \end{pmatrix}V}_{\textstyle ,} \begin{pmatrix} 1_I & 0 \\ 0 & 0 \end{pmatrix}\bigg] \in \K_0(\bK \otimes C^*\Gamma) .
\end{align*}

Now we define the map $\alpha_\Gamma ^{\mathrm{alg}}$ as
\[ \alpha _\Gamma^{\mathrm{alg}} (\xi):= [ V \diag (P_1, 1_I-P_2) V , |I| ]_{0,3}. \]
It is well-defined independent of the choice of a representative $(\varphi_1, \varphi_2)$ because a homotopy of quasi-homomorphisms gives rise to a homotopy of $(0,3)$-projections $V\diag (P_1,1_n-P_2)V$. The above discussion means that this $\alpha_\Gamma ^{\mathrm{alg}}$ satisfies $\iota_{C^*(\Gamma)} \circ \alpha_\Gamma ^{\mathrm{alg}}=\alpha_\Gamma$. 
\end{proof}
\begin{defn}
We call the map $\alpha^{\mathrm{alg}}_\Gamma$ as in Proposition \ref{lem:quantBC} the \emph{algebraic Mishchenko--Fomenko higher index}. For $r > 3$, we call the composition of $\alpha_\Gamma ^{\mathrm{alg}}$ with $\iota_{0,3}^{\varepsilon, r} \colon \K_*^{0,3}(C^*\Gamma) \to \K_*^{\delta ,r}(C^*\Gamma)$ the quantitative higher index and write it as $\alpha_{\Gamma}^{\delta , r}$. 
\end{defn}

Now we reformulate \cite[Theorem 3.2]{MR2982445} in the framework of quantitative K-theory.
\begin{thm}\label{thm:quant}
There is a constant $C_1 = C_1(\cU)$ depending only on $\cU$ that the following holds: For $ 0< \varepsilon < (4C_1)^{-1}$, $r>3$, $\pi \in \qRep_P^{\varepsilon , \cG^3_\Gamma}(\Gamma)$ and $\xi \in \K_0(X)$, we have
\[ \iota_B \circ (\id_{\bK(\sH)} \otimes \pi)_\sharp (\alpha_{\Gamma}^{\mathrm{alg}}(\xi))= \langle [\beta (\pi)] , \xi \rangle \in \K_0(B). \]
\end{thm}
\begin{rmk}\label{rmk:quantHS}
Here we discuss the relation of Theorem \ref{thm:quant} with the K-theory of group C*-algebra. 
For any $(\delta, r)$ with $|\cG_\Gamma^r|^2\varepsilon + (1+3|\cG_\Gamma^r|^2\varepsilon)\delta < 1/4$, $ (\id_{\bK(\sH)} \otimes \pi)_\sharp $ is defined on $\K_0^{\delta, r}(\bK(\sH) \otimes C^*\Gamma)$. Hence the left hand side of Theorem \ref{thm:quant} is written as $\iota_B \circ (\id_{\bK(\sH)} \otimes \pi)_\sharp  \circ \alpha_{\Gamma}^{\delta, r}(\xi)$. Let $\xi \in \K_*(X)$ be a K-homology class satisfying $\alpha_\Gamma(\xi)=0$. Then there is $(\delta, r)$ with $\delta<1/4$ such that $\alpha_\Gamma^{\delta, r}(\xi) =0$, and hence $\iota_B \circ (\id_{\bK(\sH)} \otimes \pi)_\sharp  \circ \alpha_{\Gamma}^{\delta, r}(\xi)=0$ for any $\pi \in \qRep_P^{\varepsilon , \cG_\Gamma ^r} (\Gamma)$ with $\varepsilon < \min \{  \frac{1}{4C_1} , \frac{1/4 - \delta}{|\cG_\Gamma ^r |^2(1+3\varepsilon)} \} $. By Theorem \ref{thm:quant} we obtain $\langle [\beta (\pi)] , \xi \rangle =0 $. This is a quantitative version of Hanke--Schick theorem \cite[Theorem 3.9]{MR3289846}.
\end{rmk}

For the proof of Theorem \ref{thm:quant}, first of all let $(\varphi _1 , \varphi_2) \colon C(X) \to \bB(\sH) \triangleright \bK(\sH)$ be a quasi-homomorphism representing $\xi \in \K_0(X)$ and let $\cB:=\bK(\sH) + \varphi_1(C(X))$. Note that $\cB$ is nuclear and quasi-diagonal.

Let $\pi$ denote a $(\cG_\Gamma, \varepsilon)$-representation of $\Gamma$ and let $\bv:= \beta (\pi)$. Let $P_1$, $P_2$ and $V$ be as in the proof of Proposition \ref{lem:quantBC}.
Set
\[ p_\pi := (\id_{C(X) \otimes \bM_I} \otimes \pi)(P_\cV) \in C(X) \otimes B \otimes \bM_I. \]
Moreover, let $p_{\pi , i}:= (\varphi_i \otimes \id )(p_{\pi})$, $p_{\bv , i}:= (\varphi_i \otimes \id)(p_\bv) \in \cB \otimes B \otimes \bM_I$ (for $i=1,2$) and 
\[ v_{\pi} := \begin{pmatrix} p_{\pi , 2} & 1_n - p_{\pi , 2} \\ 1_n - p_{\pi , 2} & p_{\pi , 2} \end{pmatrix}_{\textstyle ,} \ \  \ v_\bv:=\begin{pmatrix} p_{\bv , 2} & 1_n - p_{\bv , 2} \\ 1_n - p_{\bv , 2} & p_{\bv , 2} \end{pmatrix} _{\textstyle .} \]

\begin{lem}\label{lem:quant1}
For $0<\varepsilon < (60 |\cG^3_{\Gamma}|^2 )^{-1}$,  both $(\id_\cB \otimes \pi)(V\diag (P_1, 1-P_2)V)$ and $v_\pi \diag (p_{\pi ,1} , 1-p_{\pi , 2})v_\pi$ are $(15|\cG_\Gamma ^3|\varepsilon , 3)$-projections and 
\begin{align*} 
&[ (\id_\cB \otimes \pi)(V\diag (P_1, 1-P_2)V) , |I| ]_{15|\cG_\Gamma^3|\varepsilon ,3} \\
=& [v_\pi \diag (p_{\pi ,1} , 1-p_{\pi , 2})v_\pi , |I| ]_{15|\cG_\Gamma ^3|^3 \varepsilon , 3}
\end{align*}
holds.
\end{lem}
\begin{proof}
By Lemma \ref{lem:QDqRep}, $\id_{\cB} \otimes \pi$ is a $(|\cG_\Gamma^3|^2 \varepsilon , 3 , 1)$-$\ast$-homomorphism. Hence, by Proposition \ref{prp:qrepK}, we have
\[ (\id_\cB \otimes \pi)(V\diag (P_1, 1-P_2)V) \in \mathrm{P}_{|I|}^{|\cG_\Gamma^3|^2 \varepsilon , 3}(\bK(\sH) \otimes B). \]
Moreover, since 
\[ (\id_{\cB} \otimes \pi ) (\varphi_i \otimes \id_{C^*(\Gamma)}) = (\varphi _i \otimes \id_{C^*(\Gamma)})(\id_{C(X)} \otimes \pi )\]
as completely bounded maps, we have $(\id_\cB \otimes \pi)(P_i)=p_{\pi , i}$ for $i=1,2$ and $(\id_\cB \otimes \pi)(V)= v_\pi$. Therefore Proposition \ref{prp:almhom} implies that
\begin{align*}
&\|(\id_\cB \otimes \pi)(V \diag (P_1, 1-P_2) V) - v_\pi \diag (p_{\pi , 1}, 1-p_{\pi, 2}) v_{\pi} \| \\
\leq & \| (\id_\cB \otimes \pi)(V \diag (P_1, 1-P_2) V)- (\id_\cB \otimes \pi)(V \diag (P_1, 1-P_2))(\id_\cB \otimes \pi)( V) \| \\
& + \| (\id_\cB \otimes \pi)(V \diag (P_1, 1-P_2) )v_\pi - (\id_\cB \otimes \pi)(V)(\id_\cB \otimes \pi)( \diag (P_1, 1-P_2) )v_\pi \|  \\
\leq & \| V \diag (P_1 , 1-P_2)\| \cdot \| V \| \cdot |\cG_\Gamma^3|^2 \varepsilon + \| v_\pi \| \cdot  \| \diag (P_1 , 1-P_2)\| \cdot  \| V\| \cdot |\cG_\Gamma^3|^2 \varepsilon  \\
\leq & 3|\cG_\Gamma^3|^2 \varepsilon.
\end{align*}
This shows the lemma by \cite[Lemma 1.7]{MR3449163}, which claims that if $p$ is a $(\varepsilon , r)$-projection and $\| p -q\| <\varepsilon$ then $q$ is a $(5\varepsilon, r)$-projection and $[p]_{5\varepsilon , r}=[q]_{5\varepsilon, r}$.
\end{proof}

\begin{lem}\label{lem:quant2}
For $0<\varepsilon < (800|I|^2)^{-1}$, both $v_\pi \diag (p_{\pi , 1} , 1_I-p_{\pi , 2}) v_\pi$ and $v_\bv \diag (p_{\bv , 1} , 1_I-p_{\bv , 2})v _\bv$ are $(200|I|^2\varepsilon ,3)$-projections and
\begin{align*} 
[v_\pi \diag (p_{\pi , 1} , p_{\pi , 2}) v_\pi , |I|]_{200|I|^2\varepsilon} = [v_\bv \diag (p_{\bv , 1} , p_{\bv , 2}) v_\bv , |I|]_{200|I|^2\varepsilon , r}
\end{align*}
holds.
\end{lem}
\begin{proof}
As is reminded in Remark \ref{rmk:beta}, the \v{C}ech $1$-cocycle $\bv =\beta(\pi)$ satisfies $\| v_{\mu\nu}(x) - \pi(\gamma_{\mu \nu}) \| < 4 \varepsilon$. Then we have 
\begin{align}
\begin{split}
\| p_\pi(x) - p_\bv(x) \| =& \Big\| \sum_{\mu,\nu} \eta_\mu(x) \eta_\nu(x)(\pi(\gamma _{\mu \nu}) - v_{\mu \nu }(x)) \otimes e_{\mu \nu} \Big\| \\
\leq & 4|I|^2 \varepsilon. 
\end{split}\label{form:ppipv}
\end{align}
This implies that $\| p_{\pi , i} - p_{\bv , i} \| =\| \varphi_i(p_{\pi} - p_\bv)\| \leq 4|I|^2\varepsilon$ and hence 
\[ \| v_{\pi} - v_{\bv} \| = \bigg\| \begin{pmatrix} p_{\pi , 2} - p_{\bv , 2}  & p_{\bv , 2} - p_{\pi , 2}  \\ p_{\bv , 2} - p_{\pi , 2}  & p_{\pi , 2} - p_{\bv , 2}  \end{pmatrix} \bigg\| \leq	2 \cdot 4|I|^2\varepsilon =8|I|^2\varepsilon. \]
Therefore we get
\begin{align*} 
&\| v_\pi \diag ( p_{\pi , 1} , 1_I- p_{\pi , 2})v_\pi - v_\bv \diag (p_{\bv , 1}, 1_I-p_{\bv , 2}) v_\bv \| \\ 
\leq & \| (v_\pi- v_\bv) \diag( p_{\bv , 1} , 1_I-p_{\bv , 2}) v_\bv \| + \| v_\pi \diag (p_{\bv , 1} ,1_I-p_{\bv , 2}) (v_\bv - v_{\pi}) \|\\
&+ \| v_\pi  \diag (p_{\pi , 1} - p_{\bv , 1}, p_{\bv , 2}- p_{\pi , 2}) v_\pi \| \\
\leq & \| v_\pi - v_\bv \| + \| v_\pi \| \cdot \| v_{\pi} - v_\bv \| + \| v_{\pi}\|^2 \max \{ \| p_{\pi , 1} - p_{\bv, 1} \| , \| p_{\pi , 2} - p_{\bv , 2} \|  \} \\
 \leq & 8|I|^2\varepsilon +2 \cdot 8|I|^2\varepsilon + 2^2 \cdot 4|I|^2\varepsilon =40|I|^2\varepsilon.   
\end{align*}
Here we use the fact $\| p_{\bv , i} \| =1$, $\| v_{\bv} \|=1$ and $\| v_\pi \| \leq 2$, which follows from $\| v_{\pi}^2-1_{2I} \| \leq 4|I|^2\varepsilon \leq 1$.
Now \cite[Lemma 1.7]{MR3449163} concludes the proof since $v_\bv \diag (p_{\bv , 1}, 1_I-p_{\bv , 2}) v_\bv$ is a projection.
\end{proof}
\begin{proof}[Proof of Theorem \ref{thm:quant}]
Let $C_1:= \max \{ 15|\cG_\Gamma^3|^2 , 200|I|^2 \}$. Then Lemma \ref{lem:quant1} and Lemma \ref{lem:quant2} conclude the proof as
\begin{align*}
\iota_B(\id_\cB \otimes \pi)_\sharp (\alpha _\Gamma^{\mathrm{alg}}(\xi)) =&\iota_B [ \pi ( V \diag (P_1, 1-P_2)V), |I|]_{C_1\varepsilon ,3} \\
=& \iota_B[v_\pi \diag (p_{\pi ,1} , 1-p_{\pi , 2})v _{\pi}, |I|]_{C_1\varepsilon ,3} \\
=& \iota_B[v_\bv \diag (p_{\bv , 1} , p_{\bv , 2}) v_\bv , |I|]_{C_1\varepsilon, 3} \\
=& [p_{\bv , 1} , p_{\bv,2}] = \langle [p_\bv], \xi \rangle = \langle [\beta (\pi)] , \xi \rangle , 
\end{align*}
where $ [p_{\bv , 1} , p_{\bv,2}] \in \K_0(B)$ denotes the difference class.
\end{proof}

Theorem \ref{thm:quant} is related to the Connes--Gromov--Moscovici index formula \cite[Th\'{e}or\`{e}me 10]{MR1042862}, which is generalized in \cite{MR2982445}.
Let $\tau$ be a tracial state on the C*-algebra $A$. For a bundle $E$ of finitely generated Hilbert $A$-modules, let $\ch_\tau (E) \in \Omega ^{\mathrm{even}}(M)$ denote the Chern character defined in \cite[Definition 5.1]{MR2188248}. In particular, if $A=\bC$ and $\tau$ is the identity map, then $\ch_\tau (E)$ is the usual Chern character.

\begin{cor}[{cf.\ \cite[Theorem 3.6]{MR2982445}}]
Let $\pi \in \qRep^{\varepsilon , \cG }_P(\Gamma )$ for $\varepsilon < (4C_1)^{-1}$ and let $\tau$ be a trace on $A$. Then, for any elliptic operator $D$ on $M$ with the principal symbol $\sigma(D)$, we have
\[ (\tau  \circ \pi_\sharp)( \alpha_{\Gamma , \Lambda}^{\delta, r} ([M])) = \int_{T^*M}  \ch _\tau ( E_{\beta(\pi)} ) \ch (\sigma (D)) \Td (T_\bC M). \]
\end{cor}
\begin{proof}
Apply Schick's $L^2$-index theorem \cite[Theorem 6.10]{MR2188248} for the index pairing $ \tau (\langle \bv , [M] \rangle ) = \tau (\ind D_{E_{\beta(\pi)}}) $.
\end{proof}

\subsection{Relative quantitative index pairing}
Now we establish a relative version of the quantitative index pairing in the above subsection. 
Let $\cG=(\cG_\Gamma, \cG_{\Lambda})$ denote a generating set of $(\Gamma , \Lambda)$. We write $\cG^r:=(\cG_\Gamma ^r, \cG_\Lambda^r)$ and $|\cG^r|:=\max \{ |\cG_\Gamma^r| , |\cG_\Lambda^r| \}$. Let $l_\Gamma$ and $l_\Lambda$ denote the word length function on $\Gamma$ and $\Lambda$ with respect to $\cG_\Gamma$ and $\cG_\Lambda$ respectively.
Then the assumption $\phi (\cG_\Lambda) \subset \cG_\Lambda $ implies $\phi(C^*(\Lambda)_r) \subset C^*(\Gamma)_r$. 
We put the structure of a filtered C*-algebra on $C\phi$ as
\[ (C\phi)_r:= \{ (a,b_s) \in C\phi \mid a \in C^*(\Lambda)_r, \ b_s \in C^*(\Gamma)_r \}. \]  
As in Lemma \ref{lem:uniMF}, let 
\begin{align*}
U_\cW &:= -e^{-\pi i \rho_0}P_\cW + 1-P_\cW \in (C_0((Y_2')^\circ) \otimes C^*(\Lambda)_1 \otimes \bM_I)^{+}\\
V_{\cV , s} &:= -e^{-\pi i \rho_s} P_\cV + 1-P_\cV \in (C_0(X^\circ_2) \otimes (C^*\Gamma)_1 \otimes \bM_I)^{+}. 
\end{align*}
Then $(U_\cW, V_{\cV ,s}) $ is a $(0,1)$-unitary of $(C_0(X^\circ_2) \otimes C\phi)^+$ such that $[U_{\cW},V_{\cV, s}]=\ell_{\Gamma, \Lambda}$.

\begin{prp}\label{prp:qrelBC}
There is a group homomorphism
\[\alpha _{\Gamma, \Lambda }^{\mathrm{alg}} \colon \K_0(X,Y) \to \K_1^{0,2}(\bK(\sH) \otimes C\phi ) \]
such that $\iota_{C\phi} (\alpha_{\Gamma, \Lambda}^{\mathrm{alg}}(\xi))=\alpha_{\Gamma, \Lambda}(\xi)$ for any $\xi \in \K_0(X,Y)$.
\end{prp}
\begin{proof}
Let $(\varphi _1  , \varphi_2) \colon C_0(X_2^\circ) \to \bB(\sH) \triangleright \bK(\sH)$ be a quasi-homomorphism representing $\xi \in \K_0(X,Y)$. Let $U_i:=\varphi_i(U_\cW)$ and $V_{i,s}:=\varphi_i(V_{\cV,s})$ for $i=1,2$. Then
\[ (U_1U_2^* , V_{1,s}V_{2,s}^*) \in (\bK(\sH) \otimes C\phi)^+ \]
is a $(0,2)$-unitary. Now we define the map $\alpha_{\Gamma, \Lambda}^{\mathrm{alg}}$ as 
\[ \alpha_{\Gamma, \Lambda}^{\mathrm{alg}}(\xi):= [(U_1U_2^* , V_{1,s}V_{2,s}^*)] \in \K^{0,2}_1(\bK(\sH) \otimes C\phi). \]
Then it is straightforward to check that $\alpha_{\Gamma, \Lambda}^{\mathrm{alg}}$ satisfies $\iota_{C\phi } \circ \alpha_{\Gamma, \Lambda}^{\mathrm{alg}} =\alpha _{\Gamma, \Lambda}$ in a similar fashion to Proposition \ref{lem:quantBC}. It is also checked in the same way as Proposition \ref{lem:quantBC} that the map $\alpha_{\Gamma, \Lambda}^{\mathrm{alg}}$ is well-defined independent of the choice of $(\varphi_1, \varphi_2)$. 
\end{proof}

\begin{defn}
We call $\alpha _{\Gamma, \Lambda}^{\mathrm{alg}}$ as in Proposition \ref{prp:qrelBC} the \emph{algebraic relative Mishchenko--Fomenko higher index}. For $r > 3$, we call the composition of $\alpha_{\Gamma , \Lambda} ^{\mathrm{alg}}$ with $\iota_{0,2}^{\varepsilon, r} \colon \K_1^{0,2}(\bK(\sH) \otimes C\phi ) \to \K_1^{\delta ,r}(\bK (\sH ) \otimes C \phi )$ the quantitative higher index and write it as $\alpha_{\Gamma , \Lambda}^{\delta , r}$. 
\end{defn}

Next we construct a $(\varepsilon, r , \kappa)$-$\ast$-homomorphism from $C\phi$ associated to a stably relative quasi-representation. Hereafter let $(X,Y)$ be a pair of finite CW-complexes with a good open cover $\cU$. Let $\Gamma :=\pi_1(X)$ and $\Lambda:=\pi_1(Y)$. Moreover, we choose a translation function $\{ \gamma _{\mu \nu} \}$ of $\tilde{X}$ and $\{ \lambda _{\mu \nu} \}_{\mu,\nu \in I}$ of $\tilde{Y}$ such that $\phi(\lambda _{\mu \nu})=\gamma_{\mu \nu}$ for $\mu, \nu \in I$ such that $U_{\mu\nu} \cap Y \neq \emptyset$. Let $\cG_\Gamma = \{ \gamma_{\mu \nu} \} $ and $\cG_\Lambda = \{ \lambda_{\mu \nu} \}$. We write as $\cG^r:=(\cG_\Gamma ^r, \cG_\Lambda^r)$ and $|\cG^r|:= \max \{ |\cG_\Gamma ^r|, |\cG_\Lambda^r| \}$.

Let 
\[\cS:= \begin{pmatrix} C_0[-1,1) & C_0[-1, 0) \\ C_0[-1, 0) & C[-1,0] \end{pmatrix}_{\textstyle ,} \ \ \cS_0 := \begin{pmatrix} C_0(-1,1) & C_0(-1, 0) \\ C_0(-1, 0) & C_0(-1,0) \end{pmatrix} \]
and let $\hat{\cS}:= \{ (f,g) \in \cS \oplus \cS \mid f-g \in \cS_0 \}$. Then the embedding $C_0(-1,1) \to \cS_0$ to the left upper component induces a $\KK$-equivalence and hence $\K_*(\cS_0 \otimes B) \cong \K_{*-1}(B)$. We write $\theta$ for the quasi-homomorphism $(\pr _1 , \pr_2) \colon \hat{\cS} \to \cS \triangleright \cS_0$, where $\pr_i$ (for $i=1,2$) denotes the projection to the $i$-th component.

Let $\boldsymbol{\pi}=(\pi_1,\pi_2,\pi_0,u) \in \qRep_{P,Q}^{\varepsilon , \cG^r }(\Gamma, \Lambda)$ such that each $\pi_i$ is a self-adjoint representation. 
Pick a continuous path $\{ \bar{u}_s \}_{s \in [1,2]}$ of unitaries in $\mathrm{U}(\bB((P \oplus Q)^{\oplus 2}))$ such that $\bar{u}_1= \diag (u, u^*)$ and $\bar{u}_2=1$. We associate $\boldsymbol{\pi}$ with continuous families of maps from $\cG ^r_\Lambda$ to $\bB((P \oplus Q)^{\oplus 2})$ parametrized by $s \in [1,2]$ defined as
\begin{align*}
\tilde{\pi}_{1,s}'(\gamma)&:= (s-1)( \diag ( \pi_1(\phi(\gamma)), \pi_0(\gamma) , 1_{P\oplus Q}) ) \\
& \ \ \ \ \  + (2-s)\bar{u}_1^*(\diag (\pi_2(\phi(\gamma)), \pi_0(\gamma) , 1_{P \oplus Q}))\bar{u}_1, \\
\tilde{\pi}_{2,s}(\gamma)& := \bar{u}_s^* (\diag (\pi_2(\phi(\gamma)) ,  \pi_0(\gamma) , 1_{P \oplus Q}))\bar{u}_s,
\end{align*}
and $\tilde{\pi}_{1,s}(\gamma ):= \tilde{\pi}_{1,s}'(\gamma)(\tilde{\pi}_{1,s}'(\gamma)^*\tilde{\pi}_{1,s}'(\gamma))^{-1/2} $.
Then 
\[ \bar{\pi}(a,b)(s):=\left\{ \begin{array}{ll} (\pi_1(b_{s}), \pi_2(b_s)) & s \in (0 , 1) \\ (\tilde{\pi}_{1,2+s}(a), \tilde{\pi}_{2,2+s}(a))  & s \in (-1, 0] \end{array} \right. \]
determines a linear map $\bar{\pi} \colon (C\phi )_r \to B \otimes \hat{\cS}$. 

\begin{lem}\label{lem:relqhom}
For any $\boldsymbol{\pi} \in \qRep^{\varepsilon , \cG^r}_{P,Q}(\Gamma , \Lambda)$ which is self-adjoint, the above $\bar{\pi}$ is a complete $(10|\cG^r|^2 \varepsilon , r , 1)$-$\ast$-homomorphism.
\end{lem}
\begin{proof}
Since $\|\tilde{\pi}_{1,2}(\gamma) - \tilde{\pi}_{1,s}'(\gamma) \| <  \varepsilon $, we have $\| 1 - \tilde{\pi}_{1,s}'(\gamma)^*\tilde{\pi}_{1,s}'(\gamma) \| < 2 \varepsilon $ and hence
\[ \| \tilde{\pi}_{1,s}(\gamma ) - \tilde{\pi}_{1,2}(\gamma )\| \leq  \|\tilde{\pi}_{1,s}'(\gamma ) - \tilde{\pi}_{1,2}(\gamma ) \| + \|1 - (\tilde{\pi}_{1,s}'(\gamma)^*\tilde{\pi}_{1,s}'(\gamma))^{-1/2} \| < 3\varepsilon.\]
Then we obtain that 
\begin{align}
\begin{split}
& \| \tilde{\pi}_{1,2+s}(\gamma)\tilde{\pi}_{1,2+s}(\gamma') - \tilde{\pi}_{1,2+s}(\gamma\gamma' ) \| \\
\leq & \| \tilde{\pi}_{1,2+s}(\gamma)\tilde{\pi}_{1,2+s}(\gamma') - \tilde{\pi}_{1,2}(\gamma)\tilde{\pi}_{1,2}(\gamma') \| \\
 & + \| \tilde{\pi}_{1,2}(\gamma)\tilde{\pi}_{1,2}(\gamma') - \tilde{\pi}_{1,2}( \gamma \gamma')  \|  + \| \tilde{\pi}_{1,2+s}(\gamma \gamma')  - \tilde{\pi}_{1,2}(\gamma \gamma')  \| \\
\leq & 2 \cdot 3 \varepsilon + \varepsilon  + 3\varepsilon = 10 \varepsilon.   
\end{split}
\label{form:7epsilon}
\end{align}
that is, each $\tilde{\pi}_{i,2+s}$ is a $(10\varepsilon , \cG^r)$-representation. 

Now Lemma \ref{prp:almhom} implies that each evaluation $\ev_s \circ \bar{\pi} \colon (C\phi)_r \to \bM_2 \oplus \bM_2$ is a complete $(10|\cG^r|^2\varepsilon , r , 1)$-$\ast$-homomorphism, which finishes the proof.
\end{proof}

Therefore, by Proposition \ref{prp:qrepK} we get a homomorphism
\[ \theta \circ \iota_{C\phi} \circ (\id \otimes \bar{\pi})_\sharp \colon \K_1^{\delta, r}(C\phi) \to \K_1(\cS_0 \otimes  B) \cong \K_0(B) \]
for $\delta>0$ such that $\varepsilon + (1+3\varepsilon)\delta < 1/4$ and $r>0$.

\begin{thm}\label{thm:relquant}
There is a constant $C_2=C_2(\cU)$ depending only on $\cU$ that the following holds: For $0 < \varepsilon < (4C_2)^{-1}$, $\boldsymbol{\pi} \in \qRep_{P, Q}^{\varepsilon, \cG^2} (\Gamma, \Lambda)$ and $\xi \in \K_0(X,Y)$, we have 
\[ (\theta \circ \iota_{C\phi} \circ (\id_{\bK} \otimes \bar{\pi})_\sharp) (\alpha^{\mathrm{alg}}_{\Gamma, \Lambda}(\xi))= \langle [\boldsymbol{\beta}(\boldsymbol{\pi})] , \xi \rangle \in \K_0(B). \]
\end{thm}
\begin{rmk}
Here is a remark parallel to Remark \ref{rmk:quantHS}. 
For any $(\delta, r)$ with $10 |\cG^r|^2\varepsilon + (1+40|\cG^r|^2\varepsilon)\delta < 1/4$, the left hand side of Theorem \ref{thm:relquant} is written as $\iota_B \circ (\id_{\bK(\sH)} \otimes \pi)_\sharp  \circ \alpha_{\Gamma , \Lambda}^{\delta, r}(\xi)$. Hence if a K-homology class $\xi \in \K_*(X)$ satisfies $\alpha_{\Gamma , \Lambda}(\xi)=0$, then there is $(\delta, r)$ with $\delta<1/4$ such that $\alpha_\Gamma^{\delta, r}(\xi) =0$. By Theorem \ref{thm:relquant}, we have $\langle [\boldsymbol{\beta} (\boldsymbol{\pi})] , \xi \rangle =0 $ for any $\boldsymbol{\pi} \in \qRep_{P,Q}^{\varepsilon , \cG^r} (\Gamma, \Lambda)$ with $\varepsilon < \min \{  \frac{1}{4C_2} , \frac{1/4 - \delta}{10|\cG_\Gamma ^r |^2(1+3\varepsilon)} \} $. This is a quantitative version of Theorem \ref{thm:Karea}.
\end{rmk}

Let $(\varphi_1, \varphi_2) \colon C_0(X^\circ ) \to \bB(\sH) \triangleright \bK(\sH)$ be a quasi-homomorphism representing $\xi \in \K_0(X,Y)$ and let $\cB:= \bK(\sH) + \varphi_1(C_0(X_2^\circ))$. Then $\cB$ is nuclear and quasi-diagonal. Let $U^i$, $V_{s}^i$ be as in the proof of Proposition \ref{prp:qrelBC}. We consider the element
\[ u_{\boldsymbol{\pi}} = (u_{\boldsymbol{\pi} ,s})_{s \in (-1,1)} := (\id_{C_0(X_2^\circ) \otimes \bM_I} \otimes \bar{\pi})(U, V_s) \in C_0(X_2^\circ) \otimes \hat{\cS} \otimes B \otimes \bM_I \]
and set $u_{\boldsymbol{\pi} }^i:=\varphi_i(u_{\boldsymbol{\pi} })$ for $i=1,2$.

\begin{lem}\label{lem:relquant1}
For $0 < \varepsilon < (160|\cG^2|^2)^{-1}$, both $(\id_{\bK} \otimes \bar{\pi})((U^1, V^{1}_s)(U^2, V^2_s)^*)$ and $u_{\boldsymbol{\pi}}^1(u_{\boldsymbol{\pi}}^2)^*$ are $(40|\cG^2|^2\varepsilon , 2)$-unitaries and 
\[ [(\id_{\bK} \otimes \bar{\pi})((U_1, V_{1,s})(U_2, V_{2,s})^*)]_{40|\cG^2|^2\varepsilon , 2} =  [ u_{\boldsymbol{\pi} , s}^1(u_{\boldsymbol{\pi} , s}^2)^*]_{40|\cG^2|^2\varepsilon , 2} \]
holds.
\end{lem}
\begin{proof}
By Proposition \ref{prp:qrepK} and Lemma \ref{lem:relqhom}, the element $(\id_{\bK} \otimes \bar{\pi})((U_1, V_{1,s})(U_2, V_{2,s})^*)$ is a $10|\cG^2|^2\varepsilon$-unitary.
By Lemma \ref{lem:QDqRep}, the tensor product $(\id_{\cB} \otimes \bar{\pi})$ is well-defined as a completely bounded map and $(\id_{\cB} \otimes \bar{\pi})(U^i, V^i_s)(s) = u^i_{\boldsymbol{\pi} ,s}$ holds for $s \in (-1,1)$. Hence we have
\begin{align*}
&\| (\id_{\bK} \otimes \bar{\pi})((U_1, V_{1,s})(U_2, V_{2,s})^*) - u_{\boldsymbol{\pi},s}^1(u_{\boldsymbol{\pi}, s}^2)^* \| 
\leq  10|\cG^2|^2 \varepsilon .
\end{align*}
This shows the lemma by \cite[Lemma 1.7]{MR3449163}, which claims that if $u$ is a $(\varepsilon ,r)$-unitary and $\| u-v\|<\varepsilon$ holds then $v$ is a $(4\varepsilon , r)$-unitary and $[u]_{4\varepsilon , r}=[v]_{4\varepsilon , r}$. 
\end{proof}

For the proof, it is convenient to rephrase Theorem \ref{prp:bundle} in terms of unitaries $(U_\cW , V_{\cV,s})$.
Let 
\[\cC(X,Y):= \begin{pmatrix} C(X_2) & C_0(Y(1,2] ) \\ C_0(Y(1,2]) &C_0(Y[1,2] ) \end{pmatrix}_{\textstyle ,} \ \ \cC_0(X,Y) := \begin{pmatrix} C_0(X_2^\circ) & C_0((Y_{2}')^\circ ) \\ C_0((Y_{2}')^\circ ) &C_0((Y_{2}')^\circ ) \end{pmatrix} \]
and let $\hat{\cC}(X,Y):= \{ (f,g) \in \cC(X,Y) \oplus \cC(X,Y) \mid f-g \in \cC_0(X,Y) \}$. Then the embedding $C_0(X_2^\circ) \to \cC_0(X,Y)$ to the left upper component induces a $\KK$-equivalence. Let $\theta_{X,Y}$ denote the quasi-homomorphism $(\pr _1 , \pr_2) \colon \hat{\cC}(X,Y) \to \cC(X,Y) \triangleright \cC_0(X,Y)$. Then the continuous map $fs$ and $\iota$ as in Theorem \ref{prp:bundle} induces 
\[ \iota_* \circ f^* \colon \cC_0(X,Y)(0,1) \to C_0(X_2^\circ) \otimes \cS_0,  \]
which is extended to a $\ast$-homomorphism from $\cC(X,Y)(0,1)$ to $C_0(X_2^\circ) \otimes \cS$ denoted by the same letter $\iota_* \circ f^*$.

Let $\fv :=\boldsymbol{\beta} (\boldsymbol{\pi})$, $\bv_{j}:=\beta(\pi_j)$ and $\bv_{j,s}:=\beta(\tilde{\pi}_{j,s})$ for $j=1,2$ and $s \in [1,2]$.  
Let $\tilde{p}_{\fv,j} \in \cC(X_1,Y_1) \otimes B \otimes \bM_I$ for $j=1,2$ denote the projections
\[ \tilde{p}_{\fv , j}(x):=\left\{ \begin{array}{ll} p_{\bv_j}(x) & x \in X_1^\circ, \\ p_{\bv _{j,r}}(y) & x=(y,r) \in Y_2'. \end{array}\right. \]
Then $\tilde{p}_{\fv ,1} - \tilde{p}_{\fv ,2} \in \cC_0(X, Y) \otimes B \otimes \bM_I$, that is, $\tilde{p}_{\fv}:=(\tilde{p}_{\fv ,1}, \tilde{p}_{\fv ,2})$ is a projection in $\hat{\cC}(X,Y) \otimes B \otimes \bM_I$, such that $\theta_{X,Y}[(\tilde{p}_{\fv ,1}, \tilde{p}_{\fv ,2})] = [\fv]$. Now, the element
\begin{align*}
u_{\fv , s} := \tilde{p}_{\fv} e^{2\pi i \rho_s} + 1-\tilde{p}_{\fv} \in \hat{\cC}(X,Y)(0,1) \otimes B \otimes \bM_I
\end{align*}
is a unitary satisfying 
\[ \theta_{X,Y}[u_\fv]= [\fv] \otimes \beta \in \K_1(\cC_0(X,Y)(0,1) \otimes B).\]

\begin{lem}\label{lem:relquant2}
For $0< \varepsilon < (1280|I|^2)^{-1}$, both $u_{\boldsymbol{\pi},s}^1(u_{\boldsymbol{\pi},s}^2)^*$ and $u^1_{\fv ,s} (u_{\fv,s}^2)^*$ are $(320|I|^2\varepsilon,2)$-unitaries and
\[ [u_{\boldsymbol{\pi},s}^1(u_{\boldsymbol{\pi},s}^2)^*]_{320|I|^2\varepsilon ,2}=(\iota_* \circ f^*)[u^1_{\fv ,s} (u_{\fv,s}^2)^*]_{320|I|^2\varepsilon , 2}\]
holds.
\end{lem}
\begin{proof}
By the definitions of $f$, $u_{\boldsymbol{\pi} ,s}$ and $u_{\fv , s}$, we have
\begin{align*}
u_{\boldsymbol{\pi},s}(x) &= -e^{-\pi i \rho(s,r)}(\pi_1(P_\cV), \pi_2(P_\cV))(x) + 1- (\pi_1(P_\cV), \pi_2(P_\cV))(x), \\ 
\iota_* f^*(u_{\fv,s})(x) &= -e^{-\pi i \rho(s,r)}(p_{\bv_1}, p_{\bv_2})(x) + 1- (p_{\bv_1}, p_{\bv_2})(x),
\end{align*}
for $(x , s) \in X_2(0,1)$ and 
\begin{align*}
u_{\boldsymbol{\pi},s}(y,r) &= e^{2\pi i (r-1)}(\tilde{\pi}_{1,2+s}(P_\cW), \tilde{\pi}_{2,2+s}(P_\cW))(y) \\
& \hspace{6em} + 1- (\tilde{\pi}_{1,2+s}(P_\cW), \tilde{\pi}_{2,2+s}(P_\cW))(y), \\
i_*f^*(u_{\fv,s})(y,r) &= e^{2\pi i (r-1)}(p_{\bv_{1, 2+s}}, p_{\bv_{2,2+s}})(y ) + 1- (p_{\bv_{1, 2+s}}, p_{\bv_{2,2+s}})(y) ,
\end{align*}
for $(y,r,s) \in Y_2'(-1,0] $.
Hence (\ref{form:ppipv}) and (\ref{form:7epsilon}) implies that 
\begin{align*}
 & \| u_{\boldsymbol{\pi} , s} - \iota_*f^* (u_{\fv, s}) \| \\
 \leq & \|e ^{2\pi i (r-1)} \| \|(\tilde{\pi}_{1,2+s}(P_\cW), \tilde{\pi}_{2,2+s}(P_\cW)) - (p_{\bv_{1, 2+s}}, p_{\bv_{2,2+s}})  \| \\
 & + \|(1- (\tilde{\pi}_{1,2+s}(P_\cW), \tilde{\pi}_{2,2+s}(P_\cW))) - (1-(p_{\bv_{1, 2+s}}, p_{\bv_{2,2+s}})) \| \\
 \leq & 2 \cdot 4|I|^2 \cdot 10 \varepsilon = 80|I|^2 \varepsilon 
 \end{align*}
for $s \in (-1,0]$. By the same argument we also see that $\| u_{\boldsymbol{\pi} , s} - \iota_*f^* (u_{\fv, s}) \| <80 |I|^2\varepsilon$ for $s \in [0,1)$.  Again by \cite[Lemma 1.7]{MR3449163}, this conclude the proof.
\end{proof}

\begin{proof}[Proof of Theorem \ref{thm:relquant}]
Let $C_2:=\max \{ 320|I|^2 , 40|\cG^2|^2 \}$. Then Lemma \ref{lem:relquant1} and \ref{lem:relquant2} prove the theorem as
\begin{align*}
(\id \otimes \bar{\pi})_\sharp (\alpha _{\Gamma,\Lambda}^{\mathrm{alg}}(\xi)) =& [(\id_{\bK} \otimes \bar{\pi})((U_1, V_{1,s})(U_2, V_{2,s})^*)]_{C_2\varepsilon ,2} \\
=& [u_{\boldsymbol{\pi},s}^1(u_{\boldsymbol{\pi},s}^2)^*]_{C_2\varepsilon ,2}\\
=&(\iota_* \circ f^*)[u^1_{\fv ,s} (u_{\fv,s}^2)^*]_{C_2\varepsilon , 2}. \qedhere
\end{align*}
\end{proof}

\begin{cor}
Let $D$ be an elliptic operator on $M$, let $\varepsilon < (4C_2)^{-1}$, $\pi \in \qRep^{\varepsilon , \cG^2 }_{P,Q}(\Gamma , \Lambda)$ and let $\tau$ be a trace on $A$. Then we have
\[ (\tau \circ \theta \circ \iota_{C\phi} \circ (\id_\bK \otimes \bar{\pi})_\sharp )( \mu_0^{\Gamma , \Lambda} ([D])) = \int_{T^*M}  \ch _\tau ( \boldsymbol{\beta}(\boldsymbol{\pi}) ) \ch (\sigma (D)) \Td (T_\bC M). \]
\end{cor}
\begin{proof}
We extend $D$ to an elliptic operator $\hat{D}$ on the invertible double $\hat{M}$. Let $i \colon M^\circ \to M$ denote the open embedding and let $E_1$, $E_2$ be vector bundles on $\hat{M}$ such that $i_*\boldsymbol{\beta}(\boldsymbol{\pi})=[E_1]-[E_2]$. Then Theorem \ref{thm:relquant} and the $L^2$-index theorem \cite[Theorem 6.10]{MR2188248} for the index pairing 
\[ \tau(\langle \boldsymbol{\beta}(\boldsymbol{\pi}) ,  [D] \rangle) = \tau(\langle i_*\boldsymbol{\beta}(\boldsymbol{\pi}) ,  [\hat{D}] \rangle)=\tau(\ind \hat{D}_{E_1} - \ind \hat{D}_{E_2})\] 
shows the corollary since the Chern character form $\ch _\tau ( i_* \boldsymbol{\beta}(\boldsymbol{\pi}) )=\ch_\tau (E_1) - \ch_\tau (E_2)$ is a compactly supported differential form on $M^\circ$ cohomologous to $\ch _\tau ( \boldsymbol{\beta}(\boldsymbol{\pi}) )$ in $H_c^*(M^\circ)$. 
\end{proof}

\section{Dual assembly map and almost flat bundles}\label{section:8}
In this section, we relate the dual higher index map $\beta_{\Gamma, \Lambda }$ defined in Proposition \ref{prp:dual} with the almost monodromy correspondence, Theorem \ref{thm:monodromy}. The goal of this section is to show that the index pairing with elements of the subgroup $\K_{\mathrm{s\mathchar`-af}}^0(X,Y)$ of almost flat $\K$-theory class has rich information enough to detect the non-vanishing of the relative higher index under a certain assumptions on the fundamental groups.

\subsection{K-homology group of mapping cone C*-algebras}
Let $A$ and $B$ be separable C*-algebras and let $\phi \colon A \to B$ be a $\ast$-homomorphism. Let us choose unital $\ast$-representations of unitization C*-algebras $\sigma \colon A^+ \to \bB(\sH)$ and $\tau \colon B^+ \to \bB(\sK)$ such that $\tau$ and $\bar{\sigma} := \sigma \oplus \tau \circ \phi $ are ample representations, that is, $\tau^{-1}(\bK(\sK))=0$ and $\bar{\sigma}^{-1}(\bK(\bar{\sH}))=0$ (where $\bar{\sH}:=\sH \oplus \sK$). Note that we can choose $\sigma$ as the zero representation if $\phi$ is injective.

For a C*-algebra $D$, let $C_u(\fT, D)$ denote the C*-algebra of $A$-valued uniformly continuous functions on $\fT:=[0,\infty)$. Hereafter we identify $\fT$ with $[0,1)$ by a reparametrization $t \mapsto s=t(1+t^2)^{-1/2}$. Following \cite{mathKT160906690}, we define the C*-algebras
\begingroup
\allowdisplaybreaks[4]
\begin{align}
\begin{split}
\fD(A)&:=\{ T \in \bB (\bar{\sH}) \mid [T, \bar{\sigma} (a)] \in \bK(\bar{\sH} ) \ \forall a \in A  \}, \\
\fD(B)&:=\{ T \in \bB (\sK ) \mid [T , \tau(b)] \in \bB(\sK) \ \forall b \in B \}, \\
\fC(A)&:=\{ T \in \fD(A) \mid T \bar{\sigma}(a) \in \bK(\bar{\sH}) \ \forall a \in A \}, \\
\fD_L(A)&:=\{ T_s \in C_u(\fT , \fD(A)) \mid [T_s , \bar{\sigma} (a)]\in C_0([0,1) , \bK (\bar{\sH})) \ \forall a \in A \}, \\
\fC_L(A)&:= C_u(\fT , \fC(A)) \cap \fD_L(A), \\
\fD_L^0(A)&:=\{ T_s \in \fD_L(A) \mid T_0=0 \}, \\
\fC_L^0(A)&:=\{ T_s \in \fC_L(A) \mid T_0=0 \}. 
\end{split}
\label{defn:dual}
\end{align}
\endgroup
Note that $\fD(B) \subset \fD(A)$ as C*-subalgebras of $\bB(\bar{\sH})$. We write $\phi _\fD$ for this inclusion. 

\begin{lem}\label{lem:Kisom}
The inclusions 
\begin{itemize}
\item $\iota_1 \colon \fC_L^0(A) \to \fC_L(A)$,
\item $\iota_2 \colon \fC_L^0(A) \to \fD_L^0(A)$ and
\item $\iota_3 \colon \fD(A)(0, 1) \to \fD_L^0(A)$
\end{itemize}
induce isomorphisms of $\K$-groups.
\end{lem}
\begin{proof}
Note that $\iota_3$ is homotopic to the inclusion of $\fD(A)(0,1) \cong \fD(A)(0,\frac{1}{2})$ into $\fD_L^0(A)$.
They follow from the vanishing of $\K$-groups of $\fC _L(A)/\fC_L^0(A) \cong \fC(A)$, $\fD_L^0(A)/\fC_L^0(A)$ and $\fD_L^0(A)/\fD(A)(0,\frac{1}{2}) \cong \fD_L(A)$, which are proved in \cite[Proposition 5.3.7]{MR1817560}, \cite[Proposition 4.3 (b)]{mathKT160906690} and \cite[Proposition 4.3 (a)]{mathKT160906690} respectively.
\end{proof}

We consider two homomorphisms
\[\Theta_{A,*} \colon \K_{1-*}(\fD _L^0(A)) \to \KK_*(A, C_0(0,1)) \]
for $\ast = 0,1$ given by
\begin{align*}
\Theta _A ([u_s]) &:= \bigg[ \bar{\sH}(0,1) \oplus \bar{\sH}(0,1)^{\mathrm{op}} , \sigma \oplus \sigma ,  \begin{pmatrix}0 & u_s^* \\ u_s & 0 \end{pmatrix} \bigg]_{\textstyle .} \\
\Theta _{A, 1} ([p_s])&:=[\bar{\sH}(0,1), \bar{\sigma} , 2p_s-1 ],
\end{align*}
for $u_s \in \mathrm{U}(\bM_N(\fD _L(A)^0)^+)$ and $p_s \in \mathrm{P}(\bM_N(\fD_L^0(A)^+))$.

\begin{lem}
The above $\Theta _{A , 0}$ and $\Theta_{A, 1}$ are isomorphisms.
\end{lem}
\begin{proof}
By Lemma \ref{lem:Kisom}, it suffices to show that the composition
\[ \Psi_{A, *} \colon \K_{1-*}(\fD(A)(0, 1)) \xrightarrow{(\iota_3)_*} \K_{1-*}(\fD_L^0(A)) \xrightarrow{\Theta _{A, *}} \KK_*(A,C_0(0,1)) \]
is an isomorphism. 

For a locally compact space $X$, let
\begin{align*}
\fD(A, X )&:=\{ T \in C_b^{\mathrm{st}}(X, \bB(\sH)) \mid [T, \sigma(a)] \in C_0(X, \bK(\sH)) \}, \\
\fD _0(A, X)&:=\overline{C_0(X) \cdot\fD(A, X)},
\end{align*}
where $C_b^{\mathrm{st}}(X, \bB(\sH))$ denotes the C*-algebra of bounded strictly continuous $\bB(\sH)$-valued functions on $X$, which is isomorphic to the bounded operator algebra on the Hilbert $C_0(X)$-module $\sH \otimes C_0(X)$. By the duality of KK-theory \cite[Theorem 3.2]{MR1814167} and Kasparov's generalized Voiculescu theorem \cite[Theorem 5]{MR587371}, the homomorphisms $ \tilde{\Theta} _{A,X,*} \colon \K_{1-*}(\fD(A, X)) \to \KK_*(A, C_0(X))$ given by
\begin{align*} \tilde{\Theta}_{A,X,0}([u_x])&:= \bigg[ C_0(X , \bar{\sH} \oplus \bar{\sH}^{\mathrm{op}} ) , \sigma \oplus \sigma ,  \begin{pmatrix}0 & u_x^* \\ u_x & 0 \end{pmatrix} \bigg]_{\textstyle ,}\\
\tilde{\Theta}_{A,X,1}([p_x])&:= [ C_0(X , \bar{\sH} ) ,  \sigma ,  2p_x-1 ],
\end{align*}
are isomorphic. 

The remaining task is to show that the inclusions
\begin{enumerate}
\item $\fD(A)(0,1) \to \fD_0(A, (0,1))$ and
\item $\fD_0(A,(0,1)) \to \fD(A, (-1, 2) )$
\end{enumerate}
induce isomorphisms of $\K$-groups. Indeed, the composition of these two inclusions is homotopic to the inclusion $\fD(A)(0,1) \to \fD(A, (0,1))$.

For (1), apply the five lemma for the map between long exact sequences of $\K$-groups associated to 
\[
\xymatrix{
0 \ar[r] & \fD(A)(0, 1) \ar[r] \ar[d] & \fD(A)[0,1) \ar[r] \ar[d] & \fD(A) \ar[r] \ar@{=}[d] & 0 \\
0 \ar[r] & \fD _0 (A, (0,1)) \ar[r] & \fD _0 (A, [0,1)) \ar[r] & \fD(A) \ar[r] & 0 \\
}
\]
Note that $\fD(A)[0,1)$ and $\fD_0(A, [0,1))$ has trivial K-groups since they are contractible.
For (2), observe that  
\[ \fD(A, (-1,2))/\fD_0(A, (0,1))  \cong \fD(A, (-1,0]) \oplus \fD(A, [1,2))\]
and
\[ \K_*(\fD(A, [0,1))) \cong \KK_{1-*}(A, C_0[0,1)) =0. \qedhere \]
\end{proof}

It is proved in \cite[Proposition 4.2]{mathKT160906690} that $\fD_L(A)/\fC_L(A)$ is canonically isomorphic to $C_u(\fT, \fD(A))/C_u(\fT , \fC(A))$. Hence the $\ast$-homomorphism $\fD(A) \to C_u([0,1), \fD(A))$ mapping $T\in \fD(A)$ to the constant function with value $T$ induces a $\ast$-homomorphism 
\[ c \colon \fD(A) \to C_u(\fT, \fD(A) )/C_u^0(\fT , \fC(A)) \cong \fD_L(A)/\fC_L^0(A),\]
where $C_u^0([0,1), \fC(A)):= \{ T_s \in C_u(\fT, \fC(A)) \mid T_0=0\}$.
Set 
\[
 \fD_L(\phi):= \{ T_s \in \fD_L(A) \mid T_0 \in \fD(B), \ T_s-T_0 \in \fC(A) \}.
\]
Then there is a commutative diagram of exact sequences
\[\xymatrix{
0 \ar[r] &\fC_L^0(A) \ar[r] \ar@{=}[d] \ar@{}[rd]& \fD_L(\phi) \ar[d] \ar[r] & \fD(B) \ar[r] \ar[d]^{c \circ  \phi_\fD } & 0 \\
0 \ar[r] & \fC_L^0(A) \ar[r] & \fD_L(A) \ar[r] & \fD_L(A)/\fC_L^0(A) \ar[r] & 0.
}\]
Let $\iota_4$ denote the inclusion $\fC_L^0(A) \to \fD_L(\phi)$ and let $q$ denote the quotient $\fD_L(\phi) \to \fD(B)$.

\begin{lem}\label{lem:diag1}
The diagram
\[\xymatrix{
\K_* (\fD(B)(0,1)) \ar[r]^{\partial} \ar[d]^{\Psi_{B , *}} & \K_* (\fC_L^0(A)) \ar[d]^{\Theta_{A, *} \circ (\iota_2)_*} \\
\KK_{1-*} (B, C_0(0,1) ) \ar[r]^{\phi^*} & \KK_{1-*} (A , C_0(0,1))
}\]
commutes.
\end{lem}
\begin{proof}
Let $k \colon \fC_L^0(A) \to \fD_L(\phi )$ denote the inclusion. We regard an element $f \in Ck$ as a $\fD(A)$-valued continuous function on $[0,1]_t \times [0,1)_s$ such that $f(0,{\cdot}) \in \fC_L^0(\phi)$, $f(t, {\cdot}) \in \fD_L(\phi)$ for $t \in (0,1)$ and $f(1, {\cdot})=0$. Let 
\[ \varphi \colon Ck \to S(\fD_L(\phi)/\fC_L^0(A)) \cong S\fD(B)\]
denote the quotient (in other words, the evaluating homomorphism at $s=0$) and let $l \colon Ck \to \fC_L^0(A)$ denote the evaluating $\ast$-homomorphism at $t=0$. Since $\varphi_*$ is an isomorphism and $l_* \circ (\varphi_*)^{-1}= \partial$, it suffices to show that the diagram
\[\xymatrix{
\K_*(Ck) \ar[r]^{l_*} \ar[d]^{\varphi_*} & \K_*(C_L^0(A)) \ar[d]^{(\iota_2)_*} \\
\K_*(\fD(B)(0,1)) \ar[r]^{\fD\phi \circ \iota_3} \ar[d]^{\Psi_{A,*}} & \K_*(\fD_L^0(A) ) \ar[d]^{\Theta_{A, *}} \\
\K_*(\fD(B, (0,1))) \ar[r]^{\phi^*} &\K_*(\fD(A, (0,1))) 
}\]
commutes.
The lower square commutes by definition. Since the continuous path
\[ \theta_\kappa (f)(s)=\left\{ \begin{array}{ll}f (s ,s \tan (\pi\kappa /2) ) & \kappa \in [0,1/2], \\ f (s \tan (\pi(1-\kappa) /2) , s ) & \kappa \in [0,1/2], \end{array} \right.  \]
of $\ast$-homomorphisms from $Ck$ to $\fD_L^0(A)$ for $\kappa \in [0,1]$ satisfies $\theta_0 =\iota_2 \circ l$ and $\theta_1 = \varphi$, we obtain that the upper square also commutes.
\end{proof}

Let $\tilde{\sH}$ denote a Hilbert $C_0(-1,1)$-module $\sH (-1,0) \oplus \sK(-1,1)$.
We define the $\ast$-homomorphism $\tilde{\sigma} \colon C\phi \to \bB(\tilde{\sH})$ by
\[ \pi(a,b_s)(s)=\left\{ \begin{array}{ll} \bar{\sigma}(a) & s \in (-1,0), \\ \sigma(b_s) & s \in [0,1), \end{array} \right. \]
and  the group homomorphism
\[\Theta _\phi \colon \K_1(\fD_L(\phi)) \to \KK(C\phi , C_0(\bR)) ) \]
by 
\[ \Theta _\phi ([u_s]):= \bigg[ \tilde{\sH} \oplus \tilde{\sH}^{\mathrm{op}}, \pi \oplus \pi, \begin{pmatrix}0 & u_{-s}^* \\ u_{-s}& 0 \end{pmatrix} \bigg]_{\textstyle .}\]
Here we extend $u_s$ to $(-1,1)$ as $u_s=u_0$ for $s<0$.
In other words, $\Theta_\phi ([u_s])$ is represented by the quasi-homomorphism $(\Ad (u_{-s}) \circ \tilde{\sigma} , \tilde{\sigma}) \colon C\phi \to \bB(\tilde{\sH}) \triangleright \bK(\tilde{\sH})$.

\begin{lem}\label{lem:diag2}
The diagram
\[
\xymatrix@C=4em{
\K_1(\fC_L^0(A)) \ar[r]^{(\iota_4)_*} \ar[d]^{\Theta_{A , 0} \circ (\iota_2)_*}& \K_1(\fD_L(\phi ) ) \ar[r]^{q_*({\cdot }) \hotimes_{\bC} \beta } \ar[d]^{\Theta_\phi } & \K_0(\fD(B)(0,1)) \ar[d]^{\Psi_{B , 1} }  \\
\KK(A,C_0(-1,1)) \ar[r]^{\theta^*} & \KK(C\phi , C_0(-1,1 ) ) \ar[r]^{\beta \hotimes_{C_0(0,1)}  \psi^*({\cdot}) } & \KK_{1} (B,C_0(-1,1)) 
}
\]
commutes. 
\end{lem}
\begin{proof}
Let $u_s \in \bM_N(\fC_L^0(A))^+ $ be a unitary. Then we have 
\begin{align*} 
(\theta^* \circ (\iota_2)_* \circ \Theta_{A,0})([u_s])&= \bigg[ \bar{\sH}(0,1) \oplus \bar{\sH}^{\mathrm{op}}(0,1), \bar{\sigma} \oplus \bar{\sigma} , \begin{pmatrix} 0 & u_s^* \\ u_s & 0\end{pmatrix} \bigg] \\
&=-\bigg[ \bar{\sH}(-1,0) \oplus \bar{\sH}^{\mathrm{op}}(-1,0), \bar{\sigma} \oplus \bar{\sigma} , \begin{pmatrix} 0 & u_{-s}^* \\ u_{-s} & 0\end{pmatrix} \bigg] \\
&= -\bigg[ \tilde{\sH} \oplus \tilde{\sH}^{\mathrm{op}}, \tilde{\sigma} \oplus \tilde{\sigma}, \begin{pmatrix}0 & u_{-s}^* \\ u_{-s}& 0 \end{pmatrix} \bigg] \\
&= -(\Theta_\phi \circ (\iota_4)_*)([u_{s}]).
\end{align*}
This means that the left square commutes. 

Next, let $v_s \in \bM_N(\fD_L(\phi))$ be a unitary. Let $\tilde{\tau}$ denote the $\ast$-homomorphism from $B(0,1)$ to $\bB(\sK(-1, 1))$ given by $\tilde{\tau}(b)(s)=\sigma(b_{s})$ for $b=(b_s)_{s \in (0,1)} \in B(0,1)$. Then we have
\begin{align*} 
(\psi^* \circ \Theta_\phi)([v_s])&= \bigg[ \tilde{\sH} \oplus \tilde{\sH}^{\mathrm{op}}, \tilde{\sigma} |_{B(0,1)} \oplus \tilde{\sigma}|_{B(0,1)} , \begin{pmatrix}0 & v_{-s}^* \\ v_{-s}& 0 \end{pmatrix} \bigg] \\
&= \bigg[ \sK(0,1) \oplus \sK(0,1)^{\mathrm{op}}, \tilde{\tau} \oplus \tilde{\tau} , \begin{pmatrix}0 & v_{0}^* \\ v_{0}& 0 \end{pmatrix} \bigg]  \\ 
&=\tilde{\Theta}_{B,\pt , 0}([v_0]) \otimes j_* \in \KK(B(0,1), C_0(-1,1)),
\end{align*}
where $j \colon C_0(0,1) \to C_0(-1,1)$ is the inclusion, which induces a $\KK$-equivalence. Now we recall that
\[ \Psi_{B,1} ([v_0] \otimes \beta) = \tilde{\Theta}_{B , \pt , 0}([v_0]) \otimes \beta \in \KK_{-1}(B, C_0(-1,1))\]
by the definition of $\tilde{\Theta}_{B, \pt , 0}$ and $\Psi _{B,1}$. Therefore we get 
\[ \beta \hotimes_{C_0(0,1)} (\psi^* \circ \Theta_\phi)([v_s])  = \tilde{\Theta}_{B,\pi, 0}([v_0]) \otimes \beta = \Psi_{B,1} (q_*([v_s]) \otimes \beta ). \]
This means that the right square commutes. 
\end{proof}

\begin{thm}\label{thm:relKhom}
The homomorphism $\Theta_\phi$ is an isomorphism.
\end{thm}
\begin{proof}
Here we write as $S:=C_0(-1,1)$. Apply the five lemma to the diagram of exact sequences
\[\mathclap{
\xymatrix@C=1em{
\K_1(S\fD(B)) \ar[r] \ar[d]^{\Psi_{B,0}} &\K_1(\cC_L^0(A)) \ar[r] \ar[d]^{\Theta_{A,0} \circ (\iota_2)_*}& \K_1(\fD_L(\phi ) ) \ar[r] \ar[d]^{\Theta_\phi } & \K_0(S\fD(B)) \ar[r] \ar[d]^{\Psi_{B,1} }  & \K _0(\fC_L^0(A)) \ar[d]^{\Theta _{A, 1} \circ (\iota_2)_*} \\
\KK(B,S) \ar[r] & \KK(A, S) \ar[r] & \KK(C\phi , S) \ar[r] & \KK_{1} (B,S) \ar[r] & \KK_{1} (A, S),
}}
\]
which commutes by Lemma \ref{lem:diag1} and Lemma \ref{lem:diag2}.
\end{proof}

Lastly we consider the case that $A$ and $B$ are unital and $\phi \colon A \to B$ preserves the unit. 
Let $(\sigma , \sH)$ and $(\tau , \sK)$ are unital ample $\ast$-representations of $A$ and $B$ respectively and $(\bar{\sigma},\bar{\sH}) := (\sigma \oplus \tau , \sH \oplus \sK) $. 
Then the $\ast$-representations $\sigma ^+ :=\sigma \oplus 0_\sH$ onto $\sH^+:=\sH^{\oplus 2}$ and $\tau^+:=\tau \oplus 0_\sK$ onto $\sK^+:=\sK ^{\oplus 2}$ (where $0_\sH$ is the zero representation to $\sH$) extend to unital ample representations of $A^+$ and $B^+$ respectively.  
Here we use $\sigma^+$ and $\tau^+$ for the definition of C*-algebras as in (\ref{defn:dual}). We also define the C*-algebras $\fD_{L}^{\mathrm{u}}(\phi)$ as
\[ \fD_{L}^{\mathrm{u}}(\phi):= \fD_L(\phi) \cap C_u(\fT , \bB(\bar{\sH})) = p\fD_L(\phi)p, \]
where $p$ denotes the projection onto $\bar{\sH} \subset \bar{\sH}^+$, namely $p=\bar{\sigma}(1)$. 
\begin{lem}\label{lem:Khomunital}
The corner embedding $\fD_L^{\mathrm{u}}(\phi) \to \fD_L(\phi)$ induces an isomorphism of K-theory. 
\end{lem}
\begin{proof}
Since $[\sigma(1_B), T_0] \in \bK(\sH)$ and $[\bar{\sigma}(1_A) , T_s] \in \bK(\bar{\sH})$, the off-diagonal part $p\fD_L(\phi)(1-p)$ is of the form
\[ \fC:= \{ T_s \in C_0[0,1) \otimes \bK(\bar{\sH}) \mid T_0 \in \bK(\sH) \},\] 
which has trivial K-groups. Similarly, the corner subalgebra $(1-p)\fD_L(\phi)(1-p)$ is of the form
\[ \fB:= \{T_s \in C_u(\fT , \bB(\bar{\sH})) \mid T_s-T_0 \in \bK(\sH )\} .\]
By the six term exact sequence associated to the extension
\[ 0 \to \{ T_s \in C_u(\fT , \bK(\sH)) \mid T_0=0 \} \to \fB \to \bB(\sH) \to 0, \]
the $\K$-group of $\fB$ turns out to be zero. Hence the composition 
\[  \fD_L^{\mathrm{u}}(\phi) \to  \fD_L(\phi) \to \fD_L(\phi)/ \bM_2 \fC \cong  \begin{pmatrix}\fD_L^{\mathrm{u}}(\phi)/\fC & 0 \\ 0 & \fB/\fC \end{pmatrix} \]
induces an isomorphism of $\K$-theory. This finishes the proof since the quotient $\fD_L(\phi ) \to \fD_L(\phi)/\bM_2 \fC$ also induces the isomorphism of $\K$-theory. 
\end{proof}

\subsection{Range of the dual assembly map}
Let $(X,Y)$ be a pair of finite CW-complexes. Now we determine the rational relative and (stably) almost flat $\K^0$-groups $\K_{\mathrm{af}}^0(X,Y)_\bQ$ and $\K_{\mathrm{s\mathchar`-af}}(X,Y)_\bQ$ under the assumption that $\Gamma := \pi_1(X)$ and $\Lambda := \pi_1(Y)$ satisfy (\ref{cond:BC1}), (\ref{cond:BC2'}), (\ref{cond:BC3}) exhibited in pp.\pageref{cond:BC3} and
\begin{itemize}
\item[\eqnum \label{cond:BC4}] Both $\Gamma$ and $\Lambda$ are residually amenable. 
\end{itemize}
Here we say that a discrete group $\Gamma$ is residually amenable if for any nontrivial element $\gamma \in \Gamma$ there is a homomorphism from $\Gamma$ to an amenable group $\Gamma'$ which maps $\gamma$ to a nontrivial element.
For example, all residually finite groups are residually amenable. In particular, all finitely generated linear groups \cite{MR0003420}
and $3$-manifold groups \cite{MR895623} (thanks to Perelman's proof of the geometrization theorem) are examples of residually amenable groups. Note that they also satisfy the condition (\ref{cond:BC1}).

\begin{lem}\label{lem:intermediate}
Let $\Gamma $ be a residually amenable group and let $\sA$ denote the family of unitary representations of $\Gamma$ factoring through amenable quotients of $\Gamma$. Then the completion $C^*_\sA(\Gamma)$ of $\bC[\Gamma]$ by the norm $\| x \|_{\sA}:= \sup_{\pi \in \sA} \| \pi(x) \|$ is an intermediate completion, that is, there are quotient maps
\[C^*_{{\max}} (\Gamma) \xrightarrow{\epsilon_{\max, \sA}^\Gamma } C^*_\sA (\Gamma ) \xrightarrow{\epsilon_{\sA, r}^\Gamma} C^*_{r}(\Gamma ) \]
such that $\epsilon_{\sA,r}^\Gamma \circ \epsilon_{\max , \sA}^\Gamma = \epsilon ^\Gamma$.
\end{lem}
\begin{proof}
Since $\Gamma $ is residually amenable, there is a decreasing sequence $N_n$ of normal subgroups of $\Gamma$ such that $\Gamma _n:= \Gamma /N_n$ is amenable and $\bigcap_{n} N_n=\{ e \}$. Let $\lambda _n$ denote the left regular representation $\Gamma \to \mathrm{U}(\ell^2(\Gamma_n))$ and let $\lambda $ denote the left regular representation $\Gamma \to \mathrm{U}(\ell^2(\Gamma))$. Now it suffices to show that $\lambda$ is weakly contained in $\bigoplus_n \lambda_n$. 

Let $\varepsilon >0$, let $F \subset \Gamma$ be a finite subset and let $\xi \in L^2(\Gamma)$. Pick a compactly supported function $\eta \in c_c(\Gamma) \subset \ell^2(\Gamma )$ such that $\| \eta \| \leq \| \xi \|$ and $\| \xi - \eta \| < (2\| \xi \|)^{-1} \varepsilon $. For a sufficiently large $n$, the restriction of the quotient $q_n \colon \Gamma \to \Gamma _n$ to $(\mathop{\mathrm{supp}} \eta )^{-1} \cdot F \cdot (\mathop{\mathrm{supp}} \eta )$ is injective. Let us choose a section $s \colon q_n(\mathop{\mathrm{supp}} \eta) \to \mathop{\mathrm{supp}} \eta$ of $q_n$. Then we have
\begin{align*}
| (\lambda(\gamma ) \xi, \xi) - (\lambda_n (\gamma) s^*\eta , s^*\eta ) | =& |(\lambda(\gamma ) \xi, \xi) - (\lambda (\gamma) \eta , \eta ) | \\
\leq & 2\|\xi \| \cdot (2\| \xi \|)^{-1}\varepsilon = \varepsilon   
\end{align*}
for any $\gamma \in F$. This concludes the proof.
\end{proof}

\begin{lem}\label{lem:QD}
For a residually amenable group $\Gamma$, the intermediate completion $C^*_\sA (\Gamma )$ is quasi-diagonal. Moreover, a homomorphism $\phi \colon \Lambda \to \Gamma $ between residually amenable groups induces the $\ast$-homomorphism $\phi_\sA \colon C^*_\sA(\Lambda) \to C^*_\sA(\Gamma)$. 
\end{lem}
\begin{proof}
Let $\Gamma _n$ and $\lambda _n$ be as in Lemma \ref{lem:intermediate}. By the Tikuisis--White--Winter theorem \cite{MR3583354}, the group C*-algebra $C^*( \Gamma _n )$ is quasi-diagonal. Pick a dense sequence $\{ a_n \}_{n \in \bN}$ of $C^*_\sA(\Gamma)$. Then, for each $n>0$ there is an increasing sequence $\{ p_{n,m} \in \bB(\ell^2(\Gamma_n)) \}_{n \leq m}$ of finite rank projections such that $\| [\lambda_n(a_l), p_{n,m}]\| < 2^{-m}$ for all $l \leq m$. Then, $p_m:=\bigoplus p_{n,m}$ is an increasing sequence of finite rank projections in $\bigoplus \ell^2(\Gamma_n)$ such that $\| [\bigoplus _n\lambda _n (a_l), p_m] \| \to 0$ for all $l \in \bN$. Since $\bigoplus _n \lambda_n$ is a faithful representation of $C^*_\sA(\Gamma)$, the proof of the first part of the lemma is completed.

The second part follows from the fact that $\phi^*(\sA_\Gamma ) \subset \sA _\Lambda $ since amenability is passed to subgroups.
\end{proof}

\begin{thm}[{\cite[Corollary 4.4]{MR3275029}}]\label{thm:nonrelativedBC}
Let $\Gamma$ be a residually amenable group. Then, for any finite CW-complex $X$ with a reference map $ f \colon X \to B\Gamma $, any element in $\Im (\beta_{\Gamma} \circ \epsilon^\Gamma) \subset \K^0(X)$ is almost flat. Moreover, if $\Gamma$ has the $\gamma$-element (e.g.\ $\Gamma$ is coarsely embeddable into a Hilbert space), any element $x \in \Im (f^*)$ is almost flat modulo torsion. 
\end{thm}
\begin{proof}
By Lemma \ref{lem:QD}, any element in the image of  $(\epsilon^\Gamma _\sA )^* \colon \KK(C^*_\sA \Gamma, \bC) \to \KK(C^*\Gamma, \bC)$ is quasi-diagonal in the sense of \cite[Definition 2.2]{MR3275029}, and hence mapped to an almost flat element in $\K^*(X)$ by \cite[Corollary 4.4]{MR3275029}. Now Remark \ref{rmk:nonrel} concludes the proof.
\end{proof}
 
Now we develop the relative version of Theorem \ref{thm:nonrelativedBC}. To this end, firstly we define the intermediate relative group C*-algebra
\[ C^*_\sA(\Gamma, \Lambda ) := SC(\phi_\sA  \colon C^*_\sA \Lambda \to C^*_\sA \Gamma ) \]
and  a finite rank approximation of a representative of each element $x \in \KK(C^*_\sA(\Gamma, \Lambda), \bC)$. 
Let $(\sigma , \sH)$ and $(\tau, \sK)$ be unital $\ast$-representations of $C^*_\sA(\Lambda)$ and $C^*_\sA(\Gamma)$ respectively such that $\tau$ and $\bar{\sigma}:=\sigma \oplus \tau \circ \phi_\sA$ are ample. 
By Theorem \ref{thm:relKhom} and Lemma \ref{lem:Khomunital}, the $\KK$-group $\KK(C^*_\sA (\Gamma, \Lambda), \bC)$ is isomorphic to the K-group of $\fD_L^\mathrm{u}(\phi _\sA )$ by the map $\Theta_\phi$.

Let $\cB:= \bB(\bar{\sH}) \oplus _{\cQ(\bar{\sH})} \bB(\bar{\sH})$. Note that the inclusion $\iota \colon \bK(\bar{\sH}) \to \cB$ to the first component induces the isomorphism of $\K$-groups. Let $p , q \in \bB(\bar{\sH})$ denote the projection onto $\sH$ and $\sK$ and set $P:= (p,p) \cB$, $Q:=(q,q)\cB$ (note that $Q=0$ if $\tau$ is the zero representation). Let $\Pi_u:=(\pi_1 , \pi_2 , \pi_0, \tilde{\pi},1)$ denote the stably h-relative representation of $(\Gamma, \Lambda)$ on $(P,Q)$ defined by $\pi_1:=(\Ad (u_0) \circ \sigma , \sigma)$, $\pi_2:= (\sigma , \sigma)$, $\pi_0 :=(\tau , \tau)$ and $\tilde{\pi}_\kappa$ is a continuous family of representations of $\Lambda$ onto $P \oplus Q $ defined as
\[\tilde{\pi}_\kappa (\gamma ) := \left\{ 
\begin{array}{ll}
(u_0 \bar{\sigma}(\gamma)u_0^* , \bar{\sigma}(\gamma))  & \kappa =1, \\
(u_0u_{2-\kappa}^*\bar{\sigma}(\gamma)u_{2-\kappa}u_0^* , \bar{\sigma}(\gamma))  & \kappa \in (1,2]. \\
\end{array}
\right.
\]
We write $\boldsymbol{\Pi}_u$ for the element of $\KK(C \phi  , \cB (-1,1))$ associated to $\Pi_u$ as in (\ref{form:qhom}). Since $\sigma$ and $\tau$ factors through $C^*_\sA(\Gamma)$ and $C^*_\sA(\Lambda)$ respectively, the Kasparov bimodule representing $\boldsymbol{\Pi}_u$ actually determines an element of $\KK(C\phi_\sA , \cB (-1,1))$.

On the other hand, for a sufficiently large $s \in [0,1)$, we have
\[ \|  u_su_0^* (u_0\bar{\sigma} (\gamma )u_0^*)u_0u_s^* - \bar{\sigma} (\gamma ) \| <\varepsilon \]
for all $\gamma \in \cG_\Lambda$. That is, $\boldsymbol{\pi}_{u, \varepsilon}:=(\pi_1, \pi_2,\pi_0, u_su_0^* )$ is a stably relative $(\varepsilon, \cG)$-representation of $(\Gamma, \Lambda)$ onto $(P,Q)$ (here, only the intertwiner $u_su_0^*$ breaks the condition of stable relative representation up to $\varepsilon $). 

\begin{lem}\label{lem:Pibeta}
For a unitary $u \in \mathrm{U}(\bM_N(\fD_L^{\mathrm{u}}(\phi_\sA)))$ and any $\varepsilon < (4+4|I|^2)^{-1}$, the $\KK$-cycle $\boldsymbol{\Pi}_u$ satisfies
\[\ell_{\Gamma, \Lambda} \otimes _{C^*(\Gamma, \Lambda) } \boldsymbol{\Pi}_u = [\boldsymbol{\beta}(\boldsymbol{\pi}_{u,\varepsilon})] \in \K_0(X,Y;\cB). \]
\end{lem}
\begin{proof}
We write $\bv_i$ for the \v{C}ech $1$-cocycle $\beta(\pi_i)$ as in  for $i=1,2,0$ and let $p_{\bv_i}$ be the corresponding projecion as in Remark \ref{rmk:MF}. Since $\tilde{\pi}_\kappa= \Ad (u_0u_{2-\kappa}^*) \circ (\pi_2 \oplus \pi_0)$, $\Ad(u_0u_{2-\kappa}^* \otimes 1_{\bM_I})(p_{\bv_2} \oplus p_{\bv_0})$ gives a continuous family of projections connecting $p_{\bv_1} \oplus p_{\bv_0}$ and $p_{\bv_2} \oplus p_{\bv_0}$. Therefore, there is a continuous path of partial isometries $(v_s)_{s \in [1,2]}$ such that 
\begin{itemize}
\item $v_sv_s^* = p_{\bv_2} \oplus p_{\bv_0}$, 
\item $v_s^*v_s = \Ad(u_0u_{2-s}^*\otimes 1_{\bM_I})(p_{\bv_2} \oplus p_{\bv_0})$ for $s \in (1,2]$, 
\item $v_1^*v_1=p_{\bv_1} \oplus p_{\bv_0}$ and 
\item $v_2=p_{\bv_2} \oplus p_{\bv_0}$.
\end{itemize} 
By the continuity of $v_s$, there is $s_0 \in (1,2]$ such that $ \| v_{s_1} - v_{s_2} \| <\varepsilon$ for any $s_1, s_2 \in [1,s_0]$. Set
\[ w_s := \left\{ \begin{array}{ll} (p_{\bv_2}|_Y \oplus p_{\bv_0})(u_{2-s}u_0^*\otimes 1_{\bM_I}) & s \in [s_0 , 2], \\ (p_{\bv_2}|_Y \oplus p_{\bv_0})(u_{2-s_0}u_0^*\otimes 1_{\bM_I})v_{s_0}^*v_s & s \in [1,s_0]. \end{array} \right. \]
Then $w_s$ also satisfies $w_sw_s^* = p_{\bv_2}|_Y \oplus p_{\bv_0}$, $w_s^*w_s = \Ad(u_0u_{2-s}^*\otimes 1_{\bM_I})(p_{\bv_2}|_Y \oplus p_{\bv_0})$ for $s \in (1,2]$, $w_1^*w_1=p_{\bv_1}|_Y \oplus p_{\bv_0}$ and $w_2=p_{\bv_2}|_Y \oplus p_{\bv_0}$. 

Let $E_{\bv_i}$ denote the $P$-bundle $p_{\bv_i} \underline{P}^I_X = \tilde{X} \times _{\pi_i} P$. Then
\[ \ell_{\Gamma, \Lambda} \otimes _{C^*(\Gamma, \Lambda) } \boldsymbol{\Pi}_u = [ E_{\bv_1} , E_{\bv_2} , E_{\bv_0}, w_1 ] \in \K^0(X,Y;\cB ) \]
 by Theorem \ref{prp:bundle}. At the same time, we also have 
 \[ [\boldsymbol{\beta}(\boldsymbol{\pi}_{u, \varepsilon})] = [E_{\bv_1} , E_{\bv_2} , E_{\bv_0} , w_1 ]. \]
Indeed, as is mentioned in (\ref{form:bbeta}) we have $\boldsymbol{\beta} (\boldsymbol{\pi}_{u, \varepsilon}) = (\bv_1, \bv_2 , \bv_0, \Delta_I(u_su_0^*))$. Hence $[\boldsymbol{\beta} (\boldsymbol{\pi}_{u, \varepsilon})] = [(E_{\bv_1}, E_{\bv_2} , E_{\bv_0}, \bar{w})]$, where $\bar{w} \in C(Y) \otimes \bB(P \oplus Q) \otimes \bM_I$ is the partial isometry constructed in Remark \ref{rmk:relative} (3). In particular, $\bar{w}$ satisfies the inequality $\| \bar{w} - (p_{\bv_2|_Y \oplus \bv_0})(u_{2-s_0}u_0^* \otimes 1_{\bM_I} ) \| < |I|^2 \varepsilon$. On the other hand, we have 
\begin{align*} 
&\|  w_1  - (p_{\bv_2}|_Y \oplus p_{\bv_0})(u_{2-s_0}u_0^* \otimes 1_{\bM_I} )  \| \\
 =& \| (p_{\bv_2}|_Y \oplus p_{\bv_0}) (u_{2-s_0}u_0^*\otimes 1_{\bM_I})(v_{s_0}^*v_1-1) \| \\
\leq &\| v_{s_0}^*v_1-1 \| < \varepsilon  ,  
\end{align*}
and hence $\| w_1 - \bar{w} \| < (1+|I|^2) \varepsilon <1/4$. 
\end{proof}

\begin{thm}\label{thm:dBC}
Let $\phi \colon \Lambda \to \Gamma$ be a homomorphism between countable discrete groups. Assume that $(\Gamma , \Lambda )$ satisfies (\ref{cond:BC1}), (\ref{cond:BC2'}) and (\ref{cond:BC4}). Let $(X,Y)$ be a pair of finite CW-complexes with a reference map $f \colon (X,Y) \to (B\Gamma, B\Lambda)$. Then any element $x \in \Im (\beta_{\Gamma , \Lambda } \circ j_\phi (\gamma_\Gamma)) \subset \K^0(X,Y )$ is stably almost flat. Moreover, it is almost flat if $\phi$ is injective.
\end{thm}
\begin{proof}
By the assumption (\ref{cond:BC2'}), the reduced relative group C*-algebra $C^*_r(\Gamma, \Lambda)$ is defined as in (\ref{form:reduced}). The C*-algebra $C^*_\sA(\Gamma , \Lambda)$ is an intermediate completion of relative group C*-algebras in the sense that there are quotient maps
\[ C^*_{{\max}} (\Gamma, \Lambda ) \xrightarrow{\epsilon_{\max, \sA}^{\Gamma , \Lambda} } C^*_\sA (\Gamma ,\Lambda) \xrightarrow{\epsilon_{\sA, r}^{\Gamma , \Lambda}} C^*_{r}(\Gamma ,\Lambda). \]
By Theorem \ref{thm:BC} (2), it suffices to show that any element of $\Im ( \beta_{\Gamma, \Lambda } \circ \epsilon _{\mathrm{max}, \sA} ^{\Gamma, \Lambda}) \subset \K ^0(X,Y) $ is stably almost flat. By Theorem \ref{thm:relKhom}, Lemma \ref{lem:Khomunital} and Lemma \ref{lem:Pibeta}, any element of $\Im ( \beta_{\Gamma, \Lambda } \circ \epsilon _{\mathrm{max}, \sA} ^{\Gamma, \Lambda})$ is of the form $[\boldsymbol{\beta}(\boldsymbol{\pi}_{u, \varepsilon})]$ by some unitary $ u \in \mathrm{U}(\bM_N(\fD_L^{\mathrm{u}}(\phi_\sA)))$ and small $\varepsilon >0$, under the identification $\K^0(X,Y;\cB) \cong \K^0(X,Y)$. Here we show that $[\boldsymbol{\beta}(\boldsymbol{\pi}_{u, \varepsilon})]$ is represented by an $(\varepsilon, \cU)$-flat stably relative vector bundle $\fv$ on $(X,Y)$ for any small $\varepsilon>0$. 

By Lemma \ref{lem:QD} and the fact that $u:=u_{2-s_0}u_0^*$ satisfies $u-1 \in \bK(\bar{\sH})$, there are finite rank projections $e \in \bK(\sH)$ and $f \in \bK(\sK)$ such that 
\begin{itemize}
\item $\| [\pi_1(\gamma) , e] \|<\varepsilon$ for $\gamma \in \cG_\Gamma$,
\item $\| (\pi_1(\gamma) - \pi_2(\gamma))e ^\perp \| < \varepsilon$ and $\| e^\perp(\pi_1(\gamma) - \pi_2(\gamma))\| < \varepsilon$  for any $\gamma \in \cG_\Gamma $, 
\item $\| [\pi_0(\gamma) , f] \|<\varepsilon$ for $\gamma \in \cG_{\Lambda}$,
\item $\| [u , e \oplus f] \|<\varepsilon $ and $ \|(e^\perp \oplus f^\perp)(u-1) (e^\perp \oplus f^\perp) \|<\varepsilon$.
\end{itemize}
We define the map $\pi_i^{e} \colon \cG_\Gamma \to e\cB e = \bB(e \cB)$ as $\pi_i^e(\gamma) := e\pi_i(\gamma)e \in e \cB e$. Similarly we also define $\pi_i ^{e^\perp}$, $\pi _0^f$ and $\pi_0^{f^\perp}$. Let $u^{e \oplus f}$ denote the unitary component of the polar decomposition of $(e \oplus f) u (e \oplus f)$, namely  
\[ u^{e \oplus f} := (e \oplus f) u (e \oplus f) ((e \oplus f) u^* (e \oplus f)u (e \oplus f)) ^{-1/2} \in (e \oplus f) \cB (e \oplus f). \]
Similarly we also define $u^{e^\perp \oplus f^\perp} := (e^\perp \oplus f^\perp ) u (e^\perp \oplus f^\perp )$.  
Then we have
\begin{enumerate}
\item[(i)] $\pi_i^e$ and $\pi_i^{e^\perp}$ are $(2\varepsilon, \cG_\Gamma)$-representation of $\Gamma $ for $i=1,2$,
\item[(ii)] $\pi_0^f$ and $\pi_0^{f^\perp}$ are $(2\varepsilon , \cG_\Lambda)$-representation of $\Lambda$,
\item[(iii)] $u^{e \oplus f} \in \Hom _{5\varepsilon} (\pi_1^e \phi \oplus \pi_0^f , \pi_2^e \phi \oplus \pi_0^f)$ and $u^{e^\perp  \oplus f^\perp } \in \Hom _{5\varepsilon} (\pi_1^{e^\perp} \phi \oplus \pi_0^{f^\perp} , \pi_2^{e^\perp} \phi \oplus \pi_0^{^\perp})$,
\item[(iv)] $ \| \pi_1^{e^\perp}  (\gamma) -  \pi_2^{e^\perp} (\gamma )\|  < \varepsilon$ for any $\gamma \in \Gamma$ and $ \| u^{e^\perp \oplus f^\perp} -1 \| <\varepsilon$.
\end{enumerate}
Since (i), (ii) and (iv) are straightforward, here we check (iii). For simplicity of notations, let $\bar{e}:=e \oplus f $. Since $\| \bar{e} u^* \bar{e} u \bar{e} - \bar{e} \| <\varepsilon$, we have
\begin{align*}
\| u^{e \oplus f}  - \bar{e} u \bar{e} \|  =\| \bar{e} u \bar{e} (1 - (\bar{e} u^* \bar{e} u \bar{e})^{-1/2}) \| \leq  \varepsilon. 
\end{align*}
This inequality and
\begin{align*}
&  \| \bar{e}u\bar{e}((\pi_1\phi \oplus \pi_0)(\gamma))\bar{e}u^* \bar{e}  -\bar{e}((\pi_2\phi \oplus \pi_0)(\gamma))\bar{e}  \|  \\
\leq & 2\| [u,\bar{e}] \| +  \|  \bar{e}(u(\pi_1\phi \oplus \pi_0)(\gamma)u^* -(\pi_2\phi \oplus \pi_0)(\gamma)) \bar{e}\| < 3\varepsilon
\end{align*}
concludes
\begin{align*}
&\| u^{e \oplus f} ((\pi_1^e\phi \oplus \pi_0^f)(\gamma)) (u^{e \oplus f})^* - \bar{e} ((\pi_2^e\phi \oplus \pi_0^f)(\gamma)) \bar{e} \| \\
 \leq & 2\| u^{e \oplus f} - \bar{e}u\bar{e} \|  + \| \bar{e}u\bar{e}((\pi_1\phi \oplus \pi_0)(\gamma))\bar{e}u^* \bar{e}  -\bar{e}((\pi_2\phi \oplus \pi_0)(\gamma))\bar{e}  \| < 5\varepsilon.
\end{align*}

Now (i), (ii), (iii) says that 
\begin{align*}
\boldsymbol{\pi}^{e, f} &:= (\pi_1^{e} , \pi_2^{e} , \pi_0^{f},  u^{e \oplus f}), \\
\boldsymbol{\pi}^{e^\perp, f^\perp} &:= (\pi_1^{e^\perp} , \pi_2^{e^\perp} , \pi_0^{f^\perp},  u^{e^\perp \oplus f^\perp}),
\end{align*}
are stably relative $(5\varepsilon , \cG)$-representations of $(\Gamma, \Lambda)$ and 
\[ d ( \boldsymbol{\pi}_{u,\varepsilon} , \boldsymbol{\pi}^{e,f} \oplus \boldsymbol{\pi}^{e^\perp , f^\perp}) < \varepsilon .\]
Moreover, (iv) implies that 
\[ d(\boldsymbol{\pi}^{e^\perp , f^\perp} , (\pi_1^{e^\perp}, \pi_1^{e^\perp}, \pi_0^{f^\perp}, 1)) < \varepsilon.\] 
By Theorem \ref{thm:monodromy} we have 
\begin{align*}
d ( \boldsymbol{\beta}(\boldsymbol{\pi}_{u, \varepsilon}) , \boldsymbol{\beta}(\boldsymbol{\pi}^{e,f}) \oplus \boldsymbol{\beta}(\boldsymbol{\pi}^{e^\perp , f^\perp})) &< 3C_{\mathrm{am}} \varepsilon, \\
d(\boldsymbol{\beta}(\boldsymbol{\pi}^{e^\perp , f^\perp}) , (\beta(\pi_1^{e^\perp}), \beta(\pi_1^{e^\perp}), \beta(\pi_0^{f^\perp}), 1)) &< 3C_{\mathrm{am}} \varepsilon .
\end{align*}
The second inequality together with Remark \ref{rmk:module} and Remark \ref{rmk:relative} (1) implies that 
\[ [\boldsymbol{\beta}(\boldsymbol{\pi}^{e^\perp , f^\perp})] = [ (\beta(\pi_1^{e^\perp}), \beta(\pi_1^{e^\perp}), \beta(\pi_0^{f^\perp}), 1)] = 0, \]
if $\varepsilon >0$ is sufficiently small. Consequently we obtain that 
\[ [\boldsymbol{\beta}(\boldsymbol{\pi}_{u, \varepsilon})] = [\boldsymbol{\beta}(\boldsymbol{\pi}^{e,f})] + [\boldsymbol{\beta}(\boldsymbol{\pi}^{e^\perp , f^\perp})] = [\boldsymbol{\beta}(\boldsymbol{\pi}^{e,f})] \]
for sufficiently small $\varepsilon >0$.

Since $e$ and $f$ are finite rank projections in $\bK(\bar{\sH}) \subset \cB$, the quadruple $ ( \pi_1^e, \pi_2^e,\pi_0^f,u^{e \oplus f})$ also determines a $(\varepsilon ,\cG)$-representation of $(\Gamma ,\Lambda)$ on a pair of finite rank vector spaces $(e\sH, f\sK)$, which is denoted by $\boldsymbol{\pi}'$. Now 
\[ \iota_*[\boldsymbol{\beta} (\boldsymbol{\pi}')]  = [\boldsymbol{\beta}(\boldsymbol{\pi}^{e,f})] =[\boldsymbol{\beta}(\boldsymbol{\pi}_{u, \varepsilon})] \in \K^0(X,Y;\cB)
\]
 finishes the proof. 

As is remarked at the beginning of Section \ref{section:8}, we can choose $\tau$ as the zero representation if $\phi$ is injective. Then the projection $f$ in the above argument is the zero projection, and hence the obtained $ \boldsymbol{\beta} (\boldsymbol{\pi}') $ is an $(\varepsilon, \cU)$-flat relative vector bundle on $(X,Y)$. Therefore, a given element $x \in \Im (\beta_{\Gamma, \Lambda} \circ j_\phi (\gamma _\Gamma))$ is almost flat.
\end{proof}

For a pair of (not necessarily finite) CW-complexes $(X,Y)$, we say that an element $x \in \K^0(X,Y)$ is (resp.\ stably) almost flat if $f^*x$ is (resp.\ stably) almost flat for any continuous map $f$ from a pair of finite CW-complexes $(Z,W)$ to $(X,Y)$. Then Theorem \ref{thm:dBC}, together with Theorem \ref{thm:BC} (2), implies the following.
\begin{cor}\label{cor:BG}
Let $\phi \colon \Lambda \to \Gamma$ be a homomorphism between countable discrete groups. Assume that $(\Gamma , \Lambda )$ satisfies (\ref{cond:BC1}), (\ref{cond:BC2'}), (\ref{cond:BC3}) and (\ref{cond:BC4}). 
\begin{enumerate}
\item Any element $x \in \K^0(B\Gamma, B\Lambda )$ is stably almost flat modulo torsion. 
\item If $\phi$ is injective, any element $x \in \K^0(B\Gamma, B\Lambda )$ is almost flat modulo torsion. 
\end{enumerate}
\end{cor}

Equivalently, we characterize infiniteness of K-area by the characteristic class.
\begin{cor}
Let $M$ be a compact spin manifold with a boundary $N$ such that $\Gamma := \pi_1(M)$ and $\Lambda := \pi_1(N)$ satisfies  (\ref{cond:BC1}), (\ref{cond:BC2'}), (\ref{cond:BC3}) and (\ref{cond:BC4}). Let $f$ denote the reference map from $(M,N)$ to $(B\Gamma, B\Lambda)$.
\begin{enumerate}
\item Then $(M,N)$ has infinite stably relative K-area if and only if $\ch (f_*[M , N])=0 \in H_{\mathrm{ev}}(B\Gamma , B\Lambda ;\bQ)$. 
\item If $\phi \colon \Lambda \to \Gamma$ is injective, then $(M,N)$ has infinite relative K-area if and only if $\ch (f_*[M , N])=0 \in H_{\mathrm{ev}}(B\Gamma , B\Lambda ;\bQ)$.
\end{enumerate}
\end{cor}
\begin{proof}
It immediately follows from Corollary \ref{cor:BG}. We only remark that the Chern character gives an isomorphism between $\K^0(B\Gamma , B\Lambda)_\bQ$ and 
\[ H^{\mathrm{ev}}(B\Gamma , B\Lambda ;\bQ): = \prod_{n \in \bN} H^{2n}(B\Gamma , B\Lambda ;\bQ) \cong \Big( \bigoplus_{n \in \bN} H_{2n}(B\Gamma , B\Lambda ;\bQ) \Big)^*_{\textstyle .} \qedhere\]
\end{proof}

\bibliographystyle{alpha}
\bibliography{bibABC,bibDEFG,bibHIJK,bibLMN,bibOPQR,bibSTUV,bibWXYZ,arXiv}

\end{document}